\documentclass[a4paper,11pt]{article}
\usepackage{graphicx}
\usepackage{amsmath,amsthm,amssymb,euscript,mathrsfs}

\usepackage{CJKutf8} 

\usepackage{dsfont}
\usepackage{xcolor}
\usepackage{empheq}
\usepackage{fancyhdr}
\usepackage{geometry}
\usepackage{tcolorbox}
\usepackage{color}
\catcode`\@=11 \@addtoreset{equation}{section}
\usepackage[colorlinks=true, linkcolor=blue, citecolor=red, urlcolor=blue]{hyperref}
\usepackage{stmaryrd}
\usepackage[labelformat=simple]{subcaption}

\usepackage{enumerate,enumitem}
\usepackage{url}
\usepackage{tikz}
\usetikzlibrary{patterns}
\usepackage[normalem]{ulem}
\usepackage{cancel}

\allowdisplaybreaks

\newtheorem{Theorem}{Theorem}[section]
\newtheorem{Proposition}[Theorem]{Proposition}
\newtheorem{Lemma}[Theorem]{Lemma}
\newtheorem{Corollary}[Theorem]{Corollary}

\theoremstyle{definition}
\newtheorem{Definition}[Theorem]{Definition}

\newtheorem{Remark}[Theorem]{Remark}

\newcommand{\bD}{\mathbb D}
\newcommand{\bI}{\mathbb I}
\newcommand{\bS}{{\mathbb S}}

\newcommand{\bDu}{\bD(\vu)}

\newcommand{\up}{{\rm up}}
\newcommand{\Up}{{\rm Up}}
\newcommand{\Fup}{{\rm F}_h^{\eps}}

\newcommand{\vrh}{\vr_h}

\newcommand{\vuh}{\vu_h}

\newcommand{\vb}{\mathbf{b}}
\newcommand{\vB}{\mathbf{B}}
\newcommand{\vC}{\mathbf{C}}

\newcommand{\vBh}{\vB_h}
\newcommand{\vCh}{\vC_h}

\newcommand{\p} {\partial}

\newcommand{\co}[2]{{\rm co}\{ #1 , #2 \}}
\newcommand{\Ov}[1]{\overline{#1}}

\newcommand{\us}{ u_\sigma }

\newcommand{\ve}{\mathbf{e}}
\newcommand{\vei}{{\ve}_i}
\newcommand{\vu}{\mathbf{u}}

\newcommand{\vv}{\vc{v}}

\newcommand{\nG}{\vn_{\sigma}}

\newcommand{\vh}{\vv_h}

\newcommand{\bfx}{x}
\newcommand{\xx}{\bfx}

\newcommand {\Qh} {Q_h}
\newcommand {\bQh}{Q_h^d} 
\newcommand {\Wh} {W_h}
\newcommand {\Whi} {\Wh^{(i)}}
\newcommand {\bWh} {{\bf W}_{h}}

\newcommand{\bfphi}{\boldsymbol{\phi}}

\newcommand{\PiQ}{\Pi_Q}

\newcommand{\PiW}{\Pi_W}
\newcommand{\PiWi}{\PiW^{(i)}}
\newcommand{\PiF}{\Pi_{\faces}}
\newcommand{\PiFi}{\PiF^{(i)}}
\newcommand{\PiB}{\Pi_B}
\newcommand{\PiBi}{\Pi_B^{(i)}}

\newcommand{\ith}{ i^{\rm th}}
\newcommand{\sumi}{\sum_{i=1}^d}
\newcommand{\sumK}{ \!\! \sum_{K \in \mesh} \!\!}

\newcommand{\sumSKIn}{ \!\! \sum_{\sigma \in \facesK \cap \facesint} \!\!}
\newcommand{\sumSi}{ \!\! \sum_{\sigma \in \facesi} \!\!}

\newcommand{\ds}{\,{\rm d}S_x}

\newcommand{\mesh}{\mathcal{T}}
\newcommand{\meshd}{{\cal D}}
\newcommand{\meshdi}{\meshd_i}
\newcommand{\grid}{\mathcal{T}}
\newcommand{\dgrid}{\meshd}
\newcommand{\Di}{\dgrid_i}

\newcommand{\TS}{\Delta t}

\newcommand{\aleq}{\stackrel{<}{\sim}}

\newcommand{\vr}{\varrho}

\newcommand{\tvr}{\widetilde \vr}
\newcommand{\tvu}{{\widetilde \vu}}

\newcommand{\tvm}{{\widetilde \vm}}
\newcommand{\tvB}{\widetilde \vB}

\newcommand{\vm}{\vc{m}}
\newcommand{\vn}{\vc{n}}
\newcommand{\vc}[1]{{\bf #1}}

\newcommand{\Div}{{\rm div}}
\newcommand{\Grad}{\nabla}
\newcommand{\Curl}{{\rm curl}}
\newcommand{\Lap}{\Delta}

\newcommand{\Gradh}{\nabla_h}
\newcommand{\Divh}{{\rm div}_h}
\newcommand{\Curlh}{{\rm curl}_h}
\newcommand{\Laph}{\Delta_h}
\newcommand{\Gradd}{\nabla_\faces}

\newcommand{\dx}{\,{\rm d} {x}}

\newcommand{\dt}{\,{\rm d} t }

\newcommand{\jump}[1]{\left\llbracket#1\right\rrbracket}
\newcommand{\abs}[1]{\left\lvert#1\right\rvert}
\newcommand{\avs}[1]{\left\{\hspace{-3pt}\left\{ #1 \right\} \hspace{-3pt}\right\} }
\newcommand{\avg}[1]{\left\{\hspace{-3pt}\left\{ #1 \right\} \hspace{-3pt}\right\} }

\newcommand{\norm}[1]{\left\lVert#1\right\rVert}

\newcommand{\dxdt}{\dx \dt}
\newcommand{\intTn}{\int_0^{t^n} }
\newcommand{\intQ}[1]{\int_{Q} #1 \dx }
\newcommand{\intTnO}[1]{\int_0^{t^n} \int_{Q} #1 \dxdt}

\newcommand{\intTauO}[1]{\int_0^\tau \int_{Q} #1 \dxdt}
\newcommand{\intTO}[1]{\int_0^T \int_{Q} #1 \dxdt}

\newcommand{\intQB}[1]{  \int_{Q} \left( #1 \right) \dx}
\newcommand{\intK}[1]{\int_{K} #1  \dx}

\newcommand{\intfacesint}[1]{\int_{\facesint}  #1 \ds}

\newcommand{\intfacesext}[1]{\int_{\facesext}  #1 \ds}
\newcommand{\facesint}{\mathcal{E}_{\rm int}}
\newcommand{\facesext}{\mathcal{E}_{\rm ext}}

\newcommand{\ep}{\varepsilon}
\newcommand{\R}{\mathbb{R}}

\newcommand{\prst}{\Bbb{P}}

\newcommand{\newcom}{\newcommand}
\newcommand{\beq}{\begin{equation}}
\newcommand{\eeq}{\end{equation}}
\newcom{\ben}{\begin{eqnarray}}
\newcom{\een}{\end{eqnarray}}
\newcom{\beno}{\begin{eqnarray*}}
\newcom{\eeno}{\end{eqnarray*}}
\newcom{\bali}{\begin{aligned}}
\newcom{\eali}{\end{aligned}}

\newcommand{\f}{\frac}

\def\softd{{\leavevmode\setbox1=\hbox{d}%
          \hbox to 1.05\wd1{d\kern-0.4ex{\char039}\hss}}}
\definecolor{Cgrey}{rgb}{0.85,0.85,0.85}
\definecolor{Cblue}{rgb}{0.50,0.85,0.85}
\definecolor{Cred}{rgb}{1,0,0}
\definecolor{fancy}{rgb}{0.10,0.85,0.10}
\definecolor{forestgreen}{rgb}{0.13, 0.55, 0.13}
\newcommand{\cred}{\color{Cred}  }
\newcommand{\cblue}{\color{blue}}

\newcommand{\pdedge}{\eth_\faces}
\newcommand{\pdedgei}{\eth_\faces^{(i)}}

\newcommand{\pdmesh}{\eth_\grid}
\newcommand{\pdmeshi}{\eth_\grid^{(i)}}

\newcommand{\Divmesh}{{\rm div}_{\mathcal{T}}}

\date{}
\newcommand{\D}{{\rm d}}
\newcommand{\pd}{\partial}
\newcommand{\pdt}{\pd _t}

\newcommand{\RE}{\mathcal{R}_E}

\newcommand{\Hc}{\mathcal{P}}
\newcommand{\eps}{\varepsilon}
\newcommand{\faces}{\mathcal{E}}
\newcommand{\facesi}{\faces^i}
\newcommand{\facesK}{\faces(K)}

\newcommand{\expe}[1]{ \mathbb{E} \left[ #1 \right] }
\newcommand{\expeB}[1]{ \mathbb{E} \Big[ #1 \Big] }

\newcommand{\br}{ \nonumber \\ }
\newcommand{\Dev}[1]{ \mbox{Dev} \left[ #1 \right] }

\begin{document}

\title{
Error analysis of the Monte Carlo  method for compressible magnetohydrodynamics
}

\author{Eduard Feireisl\thanks{The work of E.F.  was  supported by the
Czech Sciences Foundation (GA\v CR), Grant Agreement 24--11034S.
The Institute of Mathematics of the Academy of Sciences of
the Czech Republic is supported by RVO:67985840.
\newline
\hspace*{1em} $^\clubsuit$The work of M.L. was supported by the Deutsche Forschungsgemeinschaft (DFG, German Research Foundation) - project number 233630050 - TRR 146 and  project number 525853336 -  SPP 2410 ``Hyperbolic Balance Laws: Complexity, Scales and Randomness".
She is also grateful to  the Gutenberg Research College and  Mainz Institute of Multiscale Modelling  for supporting her research.
\newline
\hspace*{1em} $^\spadesuit$The work of  B.S. was supported by
National Natural Science Foundation of China under grant No. 12201437.
\newline
\hspace*{1em} $^\dagger$The work of  Y.Y. was supported by Natural Science Foundation of Jiangsu Province under grant No. BK20241364.
}
\and M\' aria Luk\' a\v cov\' a -- Medvi\softd ov\' a$^{\clubsuit}$
\and Bangwei She$^{\spadesuit}$
\and Yuhuan Yuan$^{\dagger}$
}


\maketitle

\centerline{$^*$ Institute of Mathematics of the Academy of Sciences of the Czech Republic}
\centerline{\v Zitn\' a 25, CZ-115 67 Praha 1, Czech Republic}
\centerline{feireisl@math.cas.cz}

\medskip

\centerline{$^\clubsuit$Institute of Mathematics, Johannes Gutenberg-University Mainz}
\centerline{Staudingerweg 9, 55 128 Mainz, Germany}
\centerline{lukacova@uni-mainz.de}

\medskip

\centerline{$^\spadesuit$Academy for Multidisciplinary studies, Capital Normal University}
\centerline{West 3rd Ring North Road 105, 100048 Beijing, P. R. China}
\centerline{bangweishe@cnu.edu.cn}

\medskip

\centerline{$^\dagger$School of Mathematics, Nanjing University of Aeronautics and Astronautics}
\centerline{Jiangjun Avenue No. 29, 211106 Nanjing, P. R. China}
\centerline{yuhuanyuan@nuaa.edu.cn}

\begin{abstract}
We study random compressible viscous magnetohydrodynamic flows. Combining the Monte Carlo method with a deterministic finite volume method we solve the random system numerically.
Quantitative error estimates including statistical and deterministic errors are analyzed up to a stopping time of the exact solution. On the life-span of an exact strong solution we prove the convergence of the numerical solutions.
Numerical experiments illustrate rich dynamics of random viscous compressible magnetohydrodynamics.
\end{abstract}




\section{Introduction}
\label{i}

There is a rapidly growing number of examples that originally smooth solutions of
models of compressible fluids may develop singularities in a finite time. This is less surprising for models of ideal (inviscid) fluids, where this type of behaviour is usually associated with nonlinear fluxes and the development of shock waves. The models of real (viscous) fluids are partially parabolic, where possible singularities, if any, must be of implosion (blow-up) type of the pressure, cf. the regularity criterion established first by Sun, Wang, and Zhang \cite{Sun} and followed by many others. Recently, Buckmaster et al. \cite{BuCLGS} and Merle et al. \cite{MeRaRoSz}, \cite{MeRaRoSzbis} rigorously identified a class of data giving rise to this scenario for the compressible isentropic Navier--Stokes system considered on the full physical space $\Bbb \R^3$.
Moreover, the results of Cao-Labora et al. \cite{CLGSShiSta} indicated that the blow-up may occur in a general bounded domain of $\Bbb \R^3$.

Still solutions of the compressible viscous fluid models are approached numerically, and the results typically coincide with the expected behaviour of the observed real fluids.
A possible explanation, adopted in the present paper, is that singularities are statistically negligible.
In principle, the data for numerical simulations cannot be exactly identified and random errors always occur. The Monte Carlo method discussed in the present paper is based on the Strong Law of Large Numbers applied both to the exact solutions and to their numerical approximations. Given a random set of data, the method identifies the expected value of the associated solution. Our goal in this paper is to perform a rigorous numerical analysis of this approach applied to the compressible magnetohydrodynamics (MHD) system.

Numerical approximation of compressible MHD is challenging even in deterministic case. In addition to satisfying balance equations for the density, momentum and the magnetic field, the divergence of magnetic field  has to vanish.
For exact (strong or weak) solutions the divergence
constraint is imposed only on the initial data and then it is preserved by time evolution.
It expresses the fact that there are no magnetic monopoles.
In numerical discretizations the preservation of such intrinsic constraint
is not attained automatically in general. As demonstrated by  Brackbill and Banes \cite{BB}
in compressible MHD calculations, divergence errors can be  generated and amplified by
possible discontinuities. These errors may accumulate yielding  a
breakdown of classical numerical schemes, making it impossible to calculate MHD
solutions with those methods.
Consequently, one of the major tasks in numerics for MHD
is the control of divergence errors.

In the past decades, several finite element, finite volume, and discontinuous Galerkin methods have been developed to solve numerically compressible MHD models, and the Maxwell equations in particular, see, e.g., Chen et al.~\cite{ChWuXi}, Chertock et al.~\cite{ChKuReWu}, Houston at el.~\cite{Per}, Hu et al.~\cite{HMX}, Torrilhon~\cite{T}, Wu et al.~\cite{WuJiShu}. However, besides the divergence-free constraint, a numerical solution has to preserve the positivity of the density and dissipation of  the energy. These properties are often not controlled explicitly in standard numerical schemes. Consequently, rigorous convergence/error analysis for compressible MHD is missing in general. In this context, we refer to some available convergence results for weak (or dissipative weak) solutions of compressible MHD system presented by Ding and Mao~\cite{Ding}, Li and She~\cite{LiShe,LiShe1} using a finite element method and a combined finite element-finite volume method, respectively.

In the present paper we study random compressible MHD by means of a Monte Carlo finite volume method. The results presented here seem to be optimal as we prove the convergence of the numerical  method on the life-span $[0, T_{\rm max})$ of exact strong solutions, and quantitative error estimates on any interval $[0,T]$, $T < T_{\rm max}$.

The following are the main difficulties to be addressed:
\begin{itemize}
\item The \emph{error estimates} can be obtained only in the class of regular solutions that may not exist globally-in-time.

\item The norm of regular solutions that controls the multiplicative constant in the error estimates blows up
at $t \to T_{\rm max}$.

\item The data, the solutions, and $T_{\rm max}$ are random variables, meaning Borel measurable quantities with respect to suitable topologies.

\item
To the best of our knowledge, numerical stability results for the compressible MHD system under non-homogeneous boundary conditions are not available, even in the deterministic case, and a suitable numerical method has to be identified. Further, error estimates both for \emph{deterministic} as well as \emph{random data} need to be derived.

\end{itemize}

\subsection{Compressible MHD system}
\label{iI}
The compressible MHD system describes the time evolution of viscous, compressible, and electrically conducting fluids in the presence of a magnetic field. We deliberately ignore the effect of thermal changes to incorporate possible blow-up solutions resulting from the superlinear isentropic equation of state (EOS),
$p(\vr) =a \vr^{\gamma} + b \vr$, where $a > 0, b \geq 0$, $\gamma > 1$ are constant.
Accordingly, the compressible MHD system reads:
\begin{subequations}\label{pde}
\begin{equation}\label{i1}
\pd_t \vr  + \Div (\vr \vu ) = 0,
\end{equation}
\begin{equation}\label{i2}
\pd_t (\vr \vu)  + \Div (\vr \vu \otimes \vu  ) + \Grad p= \Div \bS+\Curl\vB\times \vB +\vr \Grad G,
\end{equation}
\begin{equation}\label{i3}
 \pd_t \vB=\Curl (\vu \times \vB)-  \Curl  (\zeta \Curl \vB),\quad  \Div \vB=0
\end{equation}
\end{subequations}
in the time-space domain $(0,T)\times Q$, where $Q = \mathbb{T}^2 \times [-1,1]$ is a slab periodic in the horizontal direction.
The main state variables are the fluid density $\vr = \vr(t,x)$, the velocity field $\vu = \vu(t,x)$, or, alternatively,  the momentum $\vm = \vr \vu$, and the magnetic field $\vc{B} = \vB(t,x)$. The symbol $p = p(\vr)$ denotes the pressure. The fluid is Newtonian, with the viscous stress tensor
\[
\bS=\bS(\bDu )= 2\mu \bDu  + \lambda \Div \vu \bI \quad \mbox{with} \quad  \bDu=\frac{\Grad \vu + \Grad^t \vu}{2},
\]
where $\mu>0$ and $\lambda$, $\lambda + \frac{2}{3} \mu \geq 0$, are constant viscosity coefficients. Similarly, we suppose that the magnetic resistivity
 $\zeta > 0$ is a positive constant.
%
The vector $\Grad G$ represents the gravitational force. For the sake of simplicity, we consider
\[
	\Grad {G} = \vc{g} \in\Bbb \R^3 \ \mbox{-- a constant vector.}
\]

The system of equations \eqref{pde} is supplemented by the
Dirichlet type conditions on the lateral boundary,	
\begin{equation}\label{bc}
	\vu|_{x_3 = \pm 1} = 0, \quad
	(\vn \times \vB)|_{x_3 = \pm 1} = \vb^{\pm},
\end{equation}
where $\vb^{\pm}= (b_{1}^{\pm}, b_{2}^{\pm}, 0)$ is a constant tangential vector. Moreover, setting
\begin{align}\label{BB}
	\vB_B=\frac12\Big(-b_{2}^{-}(1- x_3) + b_{2}^{+} (1+x_3), \,  b_{1}^{-} (1-x_3) - b_{1}^{+}(1+x_3), \, 0\Big)
\end{align}
we easily observe $\vb^{\pm} = \vn \times \vB_B|_{x_3 = \pm 1}$.

To close the system we prescribe the initial conditions
\begin{equation}\label{ini}
	(\vr, \vu,\vB)|_{t=0}=(\vr_0,\vu_0,\vB_0)
	\text{ with } \vr_0 >0  \text{ and } \Div \vB_0=0.
\end{equation}

\begin{Remark} \label{R1}
	Note that the system \eqref{pde} can be easily reformulated in the 2D setting with
	\begin{align}
		\Curl \vc{B} \times \vc{B} &= \Div \left( \vc{B} \otimes \vc{B} - \frac{1}{2}|\vc{B}|^2 \bI  \right), \br
		\Curl (\vc{B} \times \vc{u} ) &= \Div \Big( \vc{B} \otimes \vu - \vu \otimes \vc{B} \Big), \br
		\Curl (\zeta \Curl \vc{B}) &= \Div \left( \zeta \Big( \Grad^t \vc{B} - \Grad \vc{B} \nonumber
	\Big) \right),
	\end{align}
	or using following notations
	\begin{align*}
		& \Curl \vv := \pd_{x_1} v_2 -\pd_{x_2} v_1, \quad \vv := (v_1, v_2) \in \R^2, \\
		& \Curl v = {\nabla^\perp v} := (\pd_{x_2} v, -\pd_{x_1} v), \quad v \in \R, \\
		& \vv \times \vc{w} = v_1 w_2 -  v_2 w_1,\quad \vc{w} := (w_1, w_2) \in \R^2.
	\end{align*}
\end{Remark}

\begin{Remark}
As $\zeta$ is constant, it is customary to use the identity
\[
\Curl ( \Curl \vc{B}) = \Grad \Div \vc{B} - \Delta \vc{B} = - \Delta \vc{B}
\]
converting \eqref{i3} into a conventional parabolic equation

\[
\partial_t \vB = \zeta \Delta \vc{B} + \Curl (\vu \times \vB).
\]
Note carefully that the solenoidality condition $\Div \vB = 0$ can be enforced by the initial data as long as the boundary condition
$\Div \vc{B}|_{\partial Q} = 0$ is imposed.
\end{Remark}

\subsection{Strategy and main results}

The quantity $D = (\vr_0, \vu_0, \vB_0; \vc{b}^\pm, \vc{g}, \mu, \lambda, \zeta)$
represents the \emph{data} for the compressible MHD system. Given data $D$ enjoying certain regularity, the problem \eqref{pde}--\eqref{ini} admits a unique
regular solution $(\vr, \vu, \vc{B}) = (\vr, \vu, \vc{B})[D]$ defined on a maximal time interval $[0, T_{\rm max})$, $T_{\rm max} > 0$, see Section \ref{S}. As pointed out above, we anticipate the possibility $T_{\rm max} < \infty$.

The time $T_{\rm max}$ is {\it a priori} characterized as a blow-up time of a suitable norm determined by the regularity of the initial data. Note that the compressible MHD system is of mixed hyperbolic-parabolic type, and no smoothing effect is expected, at least at the level of the density $\vr$.

Our approach leans essentially on the conditional regularity result proved
recently in \cite{FeiKwon23}, namely the life-span $T_{\rm max}$ is characterized by the
property
\begin{equation} \label{i5}
	\limsup_{t \to T_{\rm max}- } \| (\vr, \vu, \vB) (t, \cdot) \|_{L^\infty(Q; \R^7)}
	= \infty.
\end{equation}

We will work with a structure-preserving finite volume scheme producing a family of \emph{approximate} solutions $(\vr_h, \vu_h, \vB_h) =
(\vr, \vu, \vB)_h [D]$ depending on the data $D$ and a discretization parameter $h > 0$.
In accordance with \eqref{i5}, the error estimates
on the distance $\| (\vr, \vm, \vB)_h - (\vr, \vm, \vB) \|_X$ in a suitable Banach space $X$
will depend only on the $L^\infty-$norm of the limit solution and the norm of the data, see Section \ref{ERE} for details.

Finally,  we introduce a stopping time for the exact solution,
\begin{align}
	T_M = \sup_{\tau \in [0, T_{\rm max}) } \left\{ \sup_{t \in [0,\tau]}
	\| (\vr,\vu,\vB)(t, \cdot) \|_{L^\infty (Q; \R^7)}
	< M \right\} = \inf_{\tau {\in [0, T_{\rm max})}}  \Bigg\{ \| (\vr,\vu,\vB)(\tau, \cdot) \|_{L^\infty (Q; \R^7)} \geq M  \Bigg\},
	\nonumber
\end{align}
where $M > 0$ is an arbitrary constant. Similarly to $T_{{\rm max}}$, the time
$T_M$ can be viewed as a function of the data, $T_M = T_M[D]$.

The Monte Carlo method relies on repeated numerical treatment of large
samples of random data yielding statistical solutions to problems that may be
deterministic in principle. The data $D$ are considered as vector valued random
variables on a standard probability basis. For a given
sequence $(D_n)_{n=1}^\infty$ of i.i.d.\ copies of $D$, the numerical
method provides the associated sequence of (random) approximate solutions
$(\vr, \vu, \vc{B})_h [D_n]$, $h > 0$, $n=1, \dots$

\begin{itemize}
	\item
Our first result concerns the error estimates of the Monte Carlo finite volume method up to the stopping time $T_M$. Specifically, we show that
\begin{equation} \label{i6}
	\expe{ \left\|	\frac{1}{N} \sum_{n=1}^N (\vr, \vm, \vB)_h (t \wedge T_M) [D_n] -
		\expe{ (\vr, \vm, \vB) (t \wedge T_M) [D] } \right\|_X } \leq c(M) \left( N^{-\frac{1}{2}} + \ell(h,\Delta t) \right), 	
\end{equation}
where the factor $N^{-\frac{1}{2}}$ reflects the statistical error, while
$\ell(h, \Delta t) \to 0$ as $h \to 0,\ \Delta t \to 0$ is the (deterministic) discretization error depending on space and time discretization parameters $h$ and $\TS (= \mathcal{O}(h))$, see Section~\ref{G}. Here, $X$ stands for the associated \emph{energy space} specified in Section \ref{G}, and $\mathbb{E}$ denotes the expected value.  The result is optimal in the sense that the discretization error depends on the norm of the limit solution $(\vr, \vu, \vB)$ in higher order Sobolev norm controlled by the amplitude of $(\vr, \vu, \vB)$ in $L^\infty$.

\item

Second, we show convergence up to the maximal existence time $T_{\rm max}$,
specifically,
\begin{equation} \label{i7}
	\expe{ \left\|	\frac{1}{N} \sum_{n=1}^N (\vr, \vm, \vB)_h \mathds{1}_{[0, T_{\rm max})} [D_n] -
		\expe{ (\vr, \vm, \vB)  \mathds{1}_{[0, T_{\rm max})} [D] } \right\|_X } \to 0 
\end{equation}
as $N \to \infty,\ h \to 0, \ \Delta t \to 0,$ see Section \ref{G}.

\end{itemize}
\begin{Remark} \color{black}
The fact that the estimate \eqref{i6} depends on the stopping time
$T_M$ related to the limit (strong) solution may seem awkward at the first glance. Indeed a more ``natural'' stopping time should be related to the approximate solutions $(\vr, \vu, \vB)_h$. This alternative approach has been exploited in our companion paper \cite{FeLmShe2024}, where the concept of ``boundedness in probability'' has been introduced. Accordingly, the present paper can be seen as complementary to \cite{FeLmShe2024}.
\end{Remark}

\medskip

The present paper is organized in the following way. In the next section we summarize some theoretical results on the local-in-time existence of the strong solution of \eqref{pde}, the conditional regularity result and formulate the energy inequality together with the relative energy. Statistical error estimates in the corresponding spaces up to a stopping time and on the life-span are discussed in Section~\ref{r}. Section~\ref{Sec4} is devoted to a deterministic structure-preserving finite volume  method. We prove its structure-preserving properties, such as the conservation of mass, positivity of density, discrete energy balance as well as the divergence-free property. Furthermore, we prove that the numerical method is consistent, which means that its numerical solution satisfies weak formulation of \eqref{pde} modulo some consistency errors. The latter vanish as the discretization parameters vanish, $h \to 0, \ \Delta t \to 0.$ In Section~\ref{ERE} we derive deterministic error estimates via the relative energy. Combining the latter with the statistical error estimates we finally obtain the convergence and error estimates for the Monte Carlo finite volume method in Section~\ref{G}. The paper is closed with a series of numerical simulations for random MHD system with various initial configurations illustrating rich dynamics of the compressible MHD and convergence rates of the Monte Carlo finite volume method. In the appendices A, B, C some helpful results are summarized.

\section{Strong solutions to the compressible MHD system}
\label{S}

The compressible MHD system \eqref{pde} shares a similar structure with the isentropic Navier--Stokes system. Specifically, it can be viewed as a partially dissipative perturbation of a symmetric hyperbolic system in the sense of Kawashima and Shizuta \cite{KawShi} or Serre \cite{Serr2}.

\subsection{Local existence and conditional regularity}

The problem of existence of local-in-time regular solutions to the compressible
MHD and similar systems is nowadays well understood. There are two basic frameworks available in the literature.  One is the Hilbert space (energy) approach
based on the scale of Sobolev spaces $W^{k,2}$, $k=1,2,\dots$ used in the
work of Matsumura and Nishida \cite{MANI}, and later developed by
Valli \cite{Vall1}, \cite{Vall2}, Valli and Zajaczkowski \cite{VAZA}, among others.
 Another is the $L^p-$approach proposed in the seminal paper of Solonnikov \cite{SoloI},
later adapted by Danchin \cite{Danch2010}, Kotschote \cite{KOT6} and, quite recently, generalized by Danchin and Tolksdorf \cite{DanTol}. Although a vast part of the above references addresses only the compressible Navier--Stokes system, the technique is easy to adapt to the MHD setting, see e.g. Fan and Yu \cite[Theorem 3.1]{FanYu}.

We adopt the local existence result by Kagei and Kawashima \cite[Theorem 2.4]{KagKaw} stated in the Hilbert space framework.

\begin{Proposition}[{\bf Local existence, smooth data}] \label{SP3}
Let $Q \subset \R^3$ be a bounded domain of class $C^{k+1}$, $k \geq 3$.
Suppose the initial data belong to the class
\begin{equation}\label{S7} 	
(\vr_0,\vu_0,\vB_0) \in W^{k,2}(Q;\R^7), \ \inf_Q \vr_0>0, \ \Div \vB_0 =0, \  \vu_0|_{\pd Q} = 0, \ \vn \times \vB_0|_{\pd Q} = \vb^{\pm}
\end{equation}
	and satisfy the compatibility conditions,
	\begin{align}
		\partial^{(j)}_t \vu_0 |_{\partial Q} =
		\partial^{(j)}_t \vc{B}_0|_{\partial Q} = 0,\ j = 1,\dots, \left[ \frac{k-1}{2} \right].
		\label{S8}
	\end{align}	
	
	Then there exists $0 < T_{\rm max} \leq \infty$ such that the compressible MHD system \eqref{pde}, with the boundary conditions \eqref{bc},
	admits a strong solution $(\vr, \vu, \vc{B})$
	in $[0, T_{\rm max}) \times Q$ unique in the
	class
	\begin{align}
	 &\partial^{(j)}_t (\vr,\vu,\vB)  \in C([0,T]; W^{k-2j,2}(Q; \R^7)),\ j = 0,\dots, \left[ \frac{k}{2} \right],\br
	&	\partial^{(j)}_t (\vu,\vc{B}) \in L^2([0,T]; W^{k-2j+1,2}(Q;\Bbb \R^6)),\ j = 0,\dots, \left[ \frac{k+1}{2} \right], 	
		\label{S9}
	\end{align}
	for any $0 < T < T_{\rm max}$.
	Moreover, if $T_{\rm max} < \infty$, then
	\begin{equation} \label{S10}
		\|  (\vr,\vu,\vB)(t, \cdot) \|_{W^{k,2}(Q; \R^7)} + \| \vr^{-1} (t, \cdot) \|_{L^\infty(Q)} \to \infty \ \mbox{as}\ t \to T_{\rm max}.	
	\end{equation}
\end{Proposition}

\begin{Remark} \label{SR1}
	The ``time derivatives'' of the initial data in \eqref{S8} are evaluated recursively from the field equations \eqref{pde}.
	
\end{Remark}

\begin{Remark} \label{SR2}
As a matter of fact, the local existence result in \cite{KagKaw} is formulated for problems with homogeneous Dirichlet boundary conditions. 	
Introducing a new variable $\widetilde{\vB} = \vc{B} - \vc{B}_B$, where $\vc{B}_B$
is the extension of the tangential boundary data \eqref{BB}, we can rewrite
the induction equation \eqref{i3} as a conventional parabolic system
\[
\partial \widetilde{\vB} = \zeta \Delta  \widetilde{\vB} +
\Curl (\vu \times (\widetilde{\vB} + \vc{B}_B )),
\]
with the homogeneous Dirichlet boundary conditions
\begin{equation} \label{dbc1}
\widetilde{B}_1|_{\partial Q} = \widetilde{B}_2|_{\partial Q} = 0, 	
\end{equation}
supplemented with the homogeneous Neumann boundary condition for $\widetilde{B}_3$,
\begin{equation} \label{dbc2}
\Grad \widetilde{B}_3 \cdot \vc{n}|_{\partial Q} =
\partial_{x_3} \widetilde{B}_3 |_{x_3 = \pm 1} = 0.
\end{equation}
Given the flat geometry of the slab $Q$, the boundary conditions \eqref{dbc1} and \eqref{dbc2} yield $\Div \widetilde{\vB}|_{\partial Q}=0$. In particular,
the solenoidality condition $\Div \widetilde{\vB} = 0$ is inherited by the solution from the initial data. At this moment, we would like to point out that replacing \eqref{dbc2} by a Dirichlet boundary condition would in general give rise to an \emph{ill-posed} problem.

\end{Remark}

\begin{Proposition}[{\bf Conditional regularity}] \label{SP4}
	Let $(\vr, \vu, \vc{B})$ be the strong solutions of the compressible MHD system in $[0, T_{\rm max}) \times Q$.
	Under the hypotheses of Proposition \ref{SP3}, we have 	
	\begin{align}
		&\sup_{t \in [0,T]}  \Big(	\|  (\vr,\vu,\vB)(t, \cdot) \|_{W^{k,2}(Q;\R^7)} + \| \vr^{-1} (t, \cdot) \|_{L^\infty(Q)} \Big) \br
		&\leq \Lambda \left( \frac1{\mu}, \mu, \lambda,\frac1{\zeta}, \zeta, |\vc{b}^{\pm}|, |\vc{g}|, \| (\vr_0, \vu_0, \vB_0) \|_{W^{2,k}(Q;\R^7)}, \| \vr_0^{-1} \|_{L^\infty(Q)},
 \sup_{t \in (0,T)} \| (\vr, \vu, \vc{B})(t, \cdot) \|_{L^\infty(Q;\Bbb \R^{7})} , T \right)\label{S12}	
\end{align}
	for any $0 < T < T_{\rm max}$, where $\Lambda$ is bounded for bounded arguments.			
\end{Proposition}

\begin{Corollary} \label{SC1}

The maximal existence time $T_{\rm max}$ in Proposition \ref{SP3} is 	
the same for any $k \geq 2$, and it is characterized by the property
\begin{equation} \label{S13}
\limsup_{ t \to T_{\rm max}-} \|  (\vr,\vu,\vB)(t, \cdot) \|_{L^\infty(Q;\Bbb \R^7)} = \infty.	
	\end{equation}

	\end{Corollary}
	
\begin{Remark} \label{SR3}
	
As a matter of fact, the conditional regularity result stated in \cite[Theorem 2.1]{FeiKwon23} was proved for $k=2$. Extension to higher order derivatives claimed in Proposition \ref{SP4} is straightforward, cf.~\cite{FeNoSu2014}.
	
	\end{Remark}	

\subsection{Data dependence}

Motivated by Proposition \ref{SP4}, we introduce a family of stopping times,
\begin{align}
T_M &= \sup_{\tau \in [0, T_{\rm max}) } \left\{ \sup_{t \in [0,\tau]}
\|  (\vr,\vu,\vB)(t, \cdot)  \|_{L^\infty (Q;\Bbb \R^7)}
< M \right\} \br &= \inf_{\tau \in [0, T_{\rm max})}  \Bigg\{  \|  (\vr,\vu,\vB)(\tau, \cdot)  \|_{L^\infty (Q;\Bbb \R^7)} \geq M  \Bigg\} ,
\label{S14}
\end{align}
where $(\vr, \vu, \vc{B})$ is a strong solution of the compressible MHD system. In accordance with the conditional regularity result stated in Proposition \ref{SP4}, we have
\begin{equation} \label{S15}
	0 \leq T_M < T_{\rm max},\ T_M \nearrow T_{\rm max} \ \mbox{as}\ M \to \infty.
\end{equation}	

Next, we introduce the data space
\begin{align}
	X_D = &\left\{ (\vr_0, \vu_0, \vc{B}_0; \vc{b}^{\pm}, \vc{g}; \mu, \lambda, \zeta)\ \Big| \right. \br
	&\left. (\vr_0, \vu_0, \vc{B}_0) \in W^{k,2}(Q;\Bbb \R^7), \ \vc{b}^{\pm} \in \Bbb \R^3,\  \vc{g} \in\Bbb \R^3,\
	\mu > 0,\ {\lambda + \frac{2}{3} \mu}  \geq 0,\ \zeta > 0
	\right\}.
	\label{S16}
\end{align}
As the strong solution is uniquely determined by the data $D \in X_D$, we may consider the mapping
\[
D \in X_D \mapsto T_M[D] \in [0, \infty).
\]
Exactly as in \cite[Theorem 2.1]{FeiLuk2022}, we deduce
\begin{equation} \label{S17}
	D \in X_D \mapsto T_M[D] \ \mbox{is lower semi-continuous},
\end{equation}	
and, by virtue of \eqref{S15}, $T_{\rm max} = \sup_{M > 0} T_M$; whence
\begin{equation} \label{S18}
D \in X_D \mapsto T_{\rm max}[D] \ \mbox{is lower semi-continuous.}
\end{equation}

\subsection{Energy equality, relative energy}
\label{REE}

Smooth solutions of the compressible MHD system satisfy the energy equality
\begin{align}
\frac{\D }{\dt} &\intQ{ \left[ \frac{1}{2} \vr |\vu|^2 + \mathcal{P}(\vr) + \frac{1}{2}
|\vc{B} - \vc{B}_B|^2  \right] } 
+ \intQ{ \Big[ \mathbb{S}(\Grad \vu): \Grad \vu + \zeta |\Curl \vc{B}|^2 \Big] } \br
&= \intQ{ \Big[ (\vc{B} \times \vu )\cdot \Curl \vc{B}_B + \zeta \Curl \vc{B} \cdot \Curl \vc{B}_B + \vr \vc{g} \cdot \vu \Big] },
	\label{S20}
\end{align}
where $\vc{B}_B$ is the solenoidal extension of the boundary data introduced in \eqref{BB},
and $\mathcal{P}$ is the pressure potential related to the barotropic pressure through
the formula
\begin{equation}\label{cpp}
\mathcal{P}'(\vr) \vr - \mathcal{P}(\vr) = p(\vr).
\end{equation}

In addition, we introduce the \emph{relative energy},
\begin{equation} \label{S21}
\RE\big((\vr,\vm,\vB)\,\big|\,(\tvr,\tvm,\tvB)\big)=
\frac{1}{2} \vr \left|\frac{\vm}{{\vr}} - \frac{\tvm}{{\tvr}} \right|^2 + \mathcal{P}(\vr) - \mathcal{P}'(\tvr)(\vr - \tvr)
	- \mathcal{P}(\tvr) + \frac{1}{2}|\vc{B} - \widetilde{\vc{B}} |^2	
	\end{equation}
expressed in terms of the conservative variables $\vr, \vm = \vr \vu$, and  $\vc{B}$. The relative energy can be identified with the so-called Bregman distance associated to the convex energy
\[
E ( \vr, \vm , \vc{B} ) = \frac{1}{2} \frac{|\vm|^2 }{\vr} + \mathcal{P}(\vr) + \frac{1}{2} |\vc{B}|^2.
\]
Indeed
the function $E$ can be extended to a convex l.s.c. function on $\R^7$ by setting
\begin{align}
E ( \vr, \vm , \vc{B} ) &= \infty \ \mbox{if}\ \vr < 0 \ \mbox{or}\ \vr = 0, \ \vm \ne 0, \br
E ( \vr, \vm , \vc{B} ) &= \frac{1}{2} |\vc{B}|^2 \ \mbox{if}\ \vr = 0,\ \vm = 0.
\label{S21a}
\end{align}

\section{Solutions with random data}
\label{r}
Here and hereafter, the symbol $\big(\Omega, \mathfrak{M}, \Bbb P \big)$ denotes a probability basis with a $\sigma-$algebra of measurable sets $\mathfrak{M}$, and a complete probability measure $\Bbb P$.

We consider random data, meaning a Borel measurable mapping
\[
\omega \in \Omega \mapsto D(\omega) \in X_D.
\]
For the sake of simplicity, we restrict ourselves to deterministically bounded data, meaning
\begin{align}
& \| (\vr_0, \vu_0, \vB_0) \|_{W^{2,k}(Q;\R^7)} + \|\vr_0^{-1} \|_{L^\infty(Q)} \leq K,\
| \vc{b}^{\pm} | \leq K,\ |\vc{g}| \leq K, \br
&\frac{1}{K} \leq \mu \leq K,\
0 \leq {\lambda + \frac{2}{3} \mu } \leq K,\ \frac{1}{K} \leq \zeta \leq K,
	\label{Br1}	
	\end{align}
a.s., where $K > 0$ is a deterministic constant.

\subsection{Strong Law of Large Numbers applied up to a stopping time}

One of our  main tools is the Hilbert space version of the Strong Law of Large Numbers (SLLN), see Ledoux and Talagrand
\cite{LedTal}. 
Consider $(D_n)_{n=1}^\infty$ to be a sequence of i.i.d.\ copies of random data $D$, bounded as in \eqref{Br1}. In accordance with the conditional regularity estimates from Proposition~\ref{SP4}, the associated sequence of solutions
\[
t \mapsto (\vr, \vu, \vc{B})(t \wedge T_M)[D_n], \ n = 1,\dots
\]
is bounded in the Hilbert space $W^{k,2}(Q;\Bbb \R^7)$ by a deterministic constant
depending on $M$ and $K$. Applying the Hilbert space version of SLLN
(Ledoux, Talagrand \cite[Corollary 7.10]{LedTal})
we conclude
\begin{equation} \label{r2}
\frac{1}{N} \sum_{n=1}^N (\vr, \vu, \vc{B})(t \wedge T_M)[D_n] \to
	\expeB{ (\vr, \vu, \vc{B})(t \wedge T_M)[D] }\ \mbox{in}\ {W^{k,2}(Q;\Bbb \R^7)}
	\ \prst-\mbox{a.s.}
\end{equation}
for any $t \geq 0$.
Actually, using the fact that i.i.d.\ copies are uncorrelated we obtain a convergence rate
\begin{equation} \label{r3}
	\expe{ \left\| \sum_{n=1}^N \Big( (\vr, \vu, \vc{B})(t \wedge T_M)[D_n] -
	\expeB{ (\vr, \vu, \vc{B})(t \wedge T_M)[D] } \Big) \right\|_{W^{k,2}(Q;\Bbb \R^7)}^2 }  \leq c(K,M) N
\end{equation}
for any $t \geq 0$ yielding
\begin{equation} \label{r3a}
	\expe{  \left\| \frac{1}{N} \sum_{n=1}^N (\vr, \vu, \vc{B})(t \wedge T_M)[D_n] -
		\expeB{ (\vr, \vu, \vc{B})(t \wedge T_M)[D] }\right\|_{W^{k,2}(Q;\Bbb \R^7)} }  \leq c(K,M) N^{-\frac{1}{2}}
\end{equation}
for any $t \geq 0$, cf.\ \cite[Theorem 10.5]{LedTal}.

\subsection{Strong Law of Large Numbers on the life-span}

Unlike in the previous part, integrability of the initial data does not imply boundedness of moment of the energy norm up to the time $T_{\rm max}$. Instead,
the energy balance \eqref{S20} must be used. In view of the deterministic boundedness of the data stated in \eqref{Br1}, we easily deduce
\begin{equation} \label{r4}
\expe{ \Big( \| \sqrt{\vr} \vu \|_{L^2(Q;\Bbb \R^3)}^2 + \| \vr \|^\gamma_{L^\gamma(Q)} + \| \vc{B} \|^2_{L^2(Q;\Bbb \R^3)} \Big)^m } \aleq 1 	
	\end{equation}
for any $m \geq 0$. Here, $a \aleq b$ means $a \leq C b$ for some positive constant $C$.
Similarly to the preceding section, we conclude
\begin{equation} \label{r5}
	\frac{1}{N} \sum_{n=1}^N (\vr, \vm, \vc{B}) \mathds{1}_{[0, T_{\rm max)}} [D_n] \to
	\expe{ (\vr, \vm, \vc{B})\mathds{1}_{[0, T_{\rm max)}}[D] }\ \mbox{in}\
	L^\gamma \times L^{\frac{2 \gamma}{\gamma + 1}} \times L^2(Q;\Bbb \R^7),
	\ \mathcal{P}-\mbox{a.s.}
\end{equation}
for any $t \geq 0$, where $\vm = \vr \vu$. Here, we have used H\"older's inequality
\[
\| \vm \|_{L^{\frac{2\gamma}{\gamma + 1}}(Q; \Bbb \R^3)} \leq
\| \sqrt{\vr} \vu \|_{L^2(Q; \Bbb \R^3)} \| {\vr} \|_{L^{\gamma}(Q)}^{\frac{1}{2}}.
\]

\begin{Remark} \label{Rss1}
Similarly to \eqref{r3a}, we can derive a convergence rate in \eqref{r5}
in terms of $N$ and the convexity properties of the norm (type) in the
Banach spaces $L^\gamma$, $L^{\frac{2 \gamma}{\gamma + 1}}$, and $L^2$, respectively, see Ledoux and Talagrand \cite[Chapter 9]{LedTal}.
\end{Remark}

\section{Structure-preserving finite volume method}
\label{Sec4}

The aim of this section is to introduce an upwind finite volume  method for \eqref{pde}--\eqref{ini} and show its structure-preserving properties, such as the positivity of density and discrete energy balance. Further, we will study its stability and consistency including the divergence-free property of the magnetic field. These are fundamental for convergence and error analysis of a numerical scheme.

\subsection{Notations}
We present our results on the convergence and error estimates for the Monte Carlo FV method for $d=3$. Note that the results for $d=2$ will directly follow. For completeness numerical method will be described in what follows for both cases, $d=2,3.$
We begin with the introduction of space and time discretizations, function spaces, and necessary notations.

\vspace{-0.4cm}
\paragraph{Mesh.}  Let $\grid$ be a uniform structured mesh formed by cubes for $d=3$ (or squares for $d=2$)
with mesh size $h\in(0,1)$ such that $Q = \bigcup_{K \in \grid} K$.
We denote by $\faces$ the set of all faces of $\grid$, $\facesint$ the set of all interior faces, $\facesext$ the set of all exterior faces, and by $\facesi$ the set of all faces that are orthogonal to the $\ith$ basis vector $\ve_i$ of the canonical system.  The set of all faces of an element $K$ is denoted by $\facesK$.
Moreover, we write $\sigma= K|L$ if $\sigma \in \faces$ is the common face of elements $K$ and~$L$.
Further, we denote by $|K|$  (resp. $|\sigma|$) the Lebesgue measure of an element $K\in\grid$ associated with the  barycenter $x_K$ (resp. $x_\sigma$).

\vspace{-0.4cm}
\paragraph{Dual mesh.} For any $\sigma=K|L\in \facesint$, we define a dual  cell  $D_\sigma := D_{\sigma,K} \cup D_{\sigma,L}$, where the  polygon $D_{\sigma,K}$ (resp.~$D_{\sigma,L}$) is  a half of $K$ (resp. $L$), i.e.,
\begin{align*}
& D_{\sigma,K} = \left\{ x\in K\mid  x_i \in \co{(x_K)_i}{(x_\sigma)_i} \right\} \quad \mbox{for any } \sigma \in \facesi, i = 1, \dots, d,
\end{align*}
where $\co{A}{B} \equiv [ \min\{A,B\} , \max\{A,B\}]$.
For any $\sigma \in \facesext \cap \facesK$,  we define the dual cell $D_\sigma := D_{\sigma,K}$.
Further, we denote the $\ith$ dual grid $\Di$ as
$$
\Di = \left\{D_{\sigma} \mid \sigma \in \facesi \right\}. 
$$

\vspace{-0.6cm}
\paragraph{Function spaces.}\label{sec:spaces}
We define a piecewise constant function space $Q_h$ on $\mesh$ associated with the projection operator
\begin{equation*}
\PiQ: L^1(Q) \to Q_h, \quad
\PiQ  \phi (x) = \sumK  \frac{ \mathds{1}_{K}(x)}{|K|} \int_K \phi \dx , \quad  \mathds{1}_{K}(x) = \begin{cases}1 & x \in K ,\\0 & \mbox{otherwise.}\end{cases}
\end{equation*}
Analogously, we define $\Whi$ as the space of piecewise constants on the $\ith$ dual grid $\meshdi$ associated with the projection operators $\PiFi$ and $\PiWi$
\begin{align*}
& \PiFi:  
W^{1,1}(Q) \to \Whi, \quad
\PiFi \phi(x) = \sumSi \frac{\mathds{1}_{D_\sigma}(x)}{|\sigma|} \int_{\sigma} \phi \, \ds,\\
& \PiWi: 
 L^1(Q) \to \Whi, \quad
\PiWi \phi (x) = \sumSi \frac{\mathds{1}_{D_\sigma}(x)}{|D_\sigma|} \int_{D_\sigma} \phi \dx.
\end{align*}
Further, we denote $\bWh =\Wh^{(1)} \times \cdots \times \Wh^{(d)}$,  $\PiF \vv = \left(\PiF^{(1)}v_1, \cdots, \PiF^{(d)}v_d \right),$ and $\PiW \vv = \left(\PiW^{(1)}v_1, \cdots, \PiW^{(d)}v_d \right)$, where $\vv = (v_1,\dots, v_d)$.

\vspace{-0.4cm}
\paragraph{Time discretization.}
We suppose $\TS \approx h$ and denote  $t^k= k\TS$ for $k=0,\ldots,N_T(=T/\TS)$. Then we denote a discrete function $v_h$ at time $t^k= k\TS$ by $v_h^k$ and write $v_h \in L_{\TS}(0,T;\Qh)$ if $v_h^k \in \Qh$ for all $k= 0,\dots, N_T$ with  \[ v_h(t,\cdot) =v_h^0 \mbox{ for } t \leq 0,\ v_h(t,\cdot)=v_h^k \mbox{ for } t\in [k\TS,(k+1)\TS),\ k= 0,\ldots,N_T.
\]
We define the discrete time derivative by the backward Euler method
\[
 D_t v_h = \frac{v_h (t) - v_h^{\triangleleft}(t)}{\TS}  \mbox{ for } t\in(0,T) \quad \mbox{ with } \quad   v_h^{\triangleleft}(t) =  v_h(t - \Delta t) .
\]
\paragraph{Jump and average operators.}
For any piecewise continuous function $f$, we define its trace on a generic edge 
as
\begin{equation*}
\begin{aligned}
f^{\rm in}|_{\sigma} = \lim_{\delta \rightarrow 0^+} f(\xx -\delta \nG ), \ \forall\  \sigma \in \faces,
\qquad
f^{\rm out}|_{\sigma} = \lim_{\delta \rightarrow 0^+} f(\xx +\delta \nG ), \ \forall\  \sigma \in \facesint.
\end{aligned}
\end{equation*}
Note that $f^{\rm out}|_{\facesext}$ is determined by the boundary condition.
Further, we define the jump and average operators at an edge $\sigma \in \faces$ respectively as
\begin{equation}\label{op_diff}
\jump{f}_{\sigma} = f^{\rm out} - f^{\rm in} \mbox{ and }
\avg{f} = \frac{f^{\rm out} + f^{\rm in}}{2} .
\end{equation}
\paragraph{Diffusive upwind flux.}
Using the above notations,  we introduce the diffusive upwind flux for any function $r_h \in Q_h$ at a generic face  $\sigma \in \facesint$
\begin{equation}\label{num_flux}
\Fup(r_h,\vuh)
=\Up[r_h, \vuh] -  h^\eps \jump{ r_h }, \quad
\Up [r_h, \vuh]   =r_h^{\rm up}\us
=r_h^{\rm in} [\us]^+ + r_h^{\rm out} [\us]^-,  \quad \eps > -1,
\end{equation}
where $\vuh \in \Qh$ is the velocity field and
\begin{equation*}
\us =\avs{\vuh }_\sigma \cdot \vn_\sigma, \quad
[f]^{\pm} = \frac{f \pm |f| }{2} \quad \mbox{and} \quad
r_h^{\rm up} =
\begin{cases}
 r_h^{\rm in} & \mbox{if } \ \us \geq 0, \\
r_h^{\rm out} & \mbox{if } \ \us < 0.
\end{cases}
\end{equation*}

\vspace{-0.6cm}
\paragraph{Discrete derivative operators.} Next, we introduce several discrete  derivative operators for generic functions $r_h\in Q_h$ and $\vh=(v_{1,h}, \dots, v_{d,h}) \in Q_h^d$
\begin{align*}
&\pdedgei: Q_h \to \Whi, \quad \pdedgei r_h(x) = \sumSi \left( \pdedgei r\right)_{D_\sigma}  \mathds{1}_{D_\sigma} (x),\quad  (\pdedgei r_h)_{D_{\sigma}} =  \frac{\jump{r_h}}{h},
\\
& \Gradd:  Q_h \to \bWh, \quad  
 \Gradd r_h  = \left(\pdedge^{(1)} r_h, \dots, \pdedge^{(d)} r_h \right),
\\
&\Gradh: Q_h \to Q_h^d, \quad \Gradh r(x)  =  \sumK \left( \Gradh r\right)_K  \mathds{1}_K{(x)},  \quad (\Gradh r_h)_K =
\frac{1}{|K|}\sum_{\sigma\in \facesK}|\sigma| \vn  \avg{r_h},
\\
& \pdmeshi: \Whi \to Q_h, \quad \pdmeshi r_h(x) = \sumK \left(\pdmeshi r_h \right)_K \mathds{1}_K{(x)} , \quad
(\pdmeshi r_h)_K := \frac{r_h|_{\sigma_{K,i+}} - r_h|_{\sigma_{K,i-}}}{h}, 
\\
& \Divmesh: \bWh \to Q_h, \quad \Divmesh v_h = \sumi \pdmeshi v_{i,h},
\\
& \Divh: Q_h^d \to Q_h, \quad \Divh \vh(x) = \sumK \left( \Divh \vh\right)_K \mathds{1}_K{(x)} , \quad (\Divh \vh)_K =
\frac{1}{|K|}\sum_{\sigma\in \facesK}|\sigma| \vn \cdot \avg{\vh},
\\
& \Curlh:  Q_h^d \to Q_h^d, \quad  \Curlh \vh(x) = \sumK \left( \Curlh \vh\right)_K \mathds{1}_K{(x)}, \quad (\Curlh \vh)_K =
\frac{1}{|K|}\sum_{\sigma\in \facesK}|\sigma| \vc{n} \times \avg{\vh} ,
\\
&\Laph: Q_h \to Q_h, \quad \Laph r_h = \sumK \left( \Laph r_h\right)_K \mathds{1}_K{(x)}, \quad \left( \Laph r_h\right)_K = \frac{1}{|K|}\sum_{\sigma\in \facesK}|\sigma| \frac{\jump{r_h}}{h},
\end{align*}
where $\sigma_{K,i-}$ (resp. $\sigma_{K,i+}$) is the left (resp. right) edge of $K$ in the $i^{\rm th}$ direction.

\medskip

It is easy to check that
\begin{align*}
& \Gradh r_h = \left( \pdmesh^{(1)}\PiW^{(1)}r_h, \dots,  \pdmesh^{(d)}\PiW^{(d)}r_h\right),\quad
\Divh \vv_h = \sumi \pdmeshi\PiWi  v_{i,h},
\\
& \Curlh \vv_h =
\begin{cases}
\pdmesh^{(1)}\PiW^{(1)} v_{2,h} - \pdmesh^{(2)}\PiW^{(2)} v_{1,h}, & d=2, \\
\left(
\pdmesh^{(2)}\PiW^{(2)} v_{3,h} - \pdmesh^{(3)}\PiW^{(3)} v_{2,h},\
\pdmesh^{(3)}\PiW^{(3)} v_{1,h} - \pdmesh^{(1)}\PiW^{(1)} v_{3,h},\
\pdmesh^{(1)}\PiW^{(1)} v_{2,h} - \pdmesh^{(2)}\PiW^{(2)} v_{1,h}
\right), & d = 3.
\end{cases}
\end{align*}
Consequently, it holds for any $\vh \in Q_h^d$ that
\begin{align}\label{div-curl}
\Divh \Curlh \vh =0 \mbox{ and } \Curlh \Gradh \vh =0.
\end{align}
Moreover, for any $f_h,v_h \in Q_h, \ \vc{f}_h, \vh \in Q_h^d,$ the following integration by parts formulae hold
\begin{itemize}
\item
\begin{subequations}\label{InByPa}
\begin{align}\label{InByPa-1}
 \intQ{ \Gradh f_h \cdot \vh} + \intQ{  f_h \cdot \Divh\vh} = 0, 
\end{align}
if $\jump{f_h}|_{\facesext} = \avs{\vh}\cdot \vn|_{\facesext} = 0$  or $\avs{f_h}|_{\facesext} = \jump{\vh}\cdot \vn|_{\facesext} = 0$.

\item
\begin{align}\label{InByPa-2}
&\intQ{ \Laph f_h \cdot v_h} = -\intQ{\Gradd f_h \cdot \Gradd v_h}, 
\end{align}
if $\avs{v_h}|_{\facesext} = 0$  or $\jump{f_h}|_{\facesext} = 0$.

\item
\begin{align}\label{InByPa-4}
&\intQ{ \Curlh \vc{f}_h \cdot \vh } - \intQ{ \vc{f}_h \cdot \Curlh  \vh } \br
&
= \intfacesext{ \vn \times \avs{\vc{f}_h} \cdot \vh^{in} } - \frac12 \intfacesext{ \vn \times \jump{\vh} \cdot \vc{f}_h^{in} } \br
&= -\intfacesext{ \vn \times \avs{\vh} \cdot \vc{f}_h^{in} } + \frac12 \intfacesext{ \vn \times \jump{\vc{f}_h} \cdot \vh^{in} }. 
\end{align}
%
\end{subequations}
\end{itemize}

\vspace{-0.6cm}
\paragraph{Divergence-free projection.}
In order to keep divergence-free constraint of the initial data at the discrete level, meaning $\Divh \vB_h^0=0$, it is necessary to introduce a suitable projection operator that is compatible with $\Divh$. To this end, we introduce for $\vB=(B_1,\cdots, B_d)$ the following projection operator
\begin{align}\label{eqpib}
&\PiB: W^{1,1}(Q;\R^d) \to Q_h^d, \quad  \PiB \vB = ( \PiB^{(1)}  B_1,\cdots, \PiB^{(d)}  B_d ),\\
&\PiBi: W^{1,1}(Q) \to Q_h, \quad  \PiBi  \phi(x)  = \sumK \mathds{1}_K(x) (\PiBi  \phi)_K, \quad
(\PiBi  \phi)_K =  \frac1{|\ell_K^{(i)} |}  \int_{\ell_K^{(i)}} \phi \ds. \nonumber
\end{align}
Here, 
\begin{align*}
\ell_K^{(i)} = \Big\{ x \mid  x_i = (x_K)_i, \, x_j \in \left[(x_K)_j - h, \, (x_K)_j + h \right] ,\,  j \neq i\Big\}, \quad i = 1, \dots, d,
\end{align*}
is a two-dimensional square of the size $(2h)^2$ if $d=3$ and a one-dimensional line of the length $2h$ if $d=2$.
It is orthogonal to a segment parallel to $\vei$ connecting  $x_K$ and the barycenters of  neighbouring elements,  see for example Figure~\ref{figpib1}  for the 2D case.

\begin{figure}[hbt]
\centering
\begin{subfigure}{0.45\textwidth}
\raggedleft
\begin{tikzpicture}[scale=1]
\draw[-,very thick](0,0)--(6,0)--(6,6)--(0,6)--(0,0);
\draw[-,very thick](0,2)--(6,2);
\draw[-,very thick](0,4)--(6,4);
\draw[-,very thick](2,0)--(2,6);
\draw[-,very thick](4,0)--(4,6);
\draw[dashed, very thick, red](3,1)--(3,5);
\draw[dashed, very thick, blue](1,3)--(5,3);
\path node at (4.4,3.4) {\cblue $\ell_K^{(2)}$};
\path node at (3.4,4.4) {\cred $\ell_K^{(1)}$};

\fill (3,3) circle (2pt) node[below right]{$x_K$};
\fill (1,3) circle (2pt) node[left]{$x_W$};
\fill (5,3) circle (2pt) node[right]{$x_E$};
\fill (3,1) circle (2pt) node[below]{$x_S$};
\fill (3,5) circle (2pt) node[above]{$x_N$};
\end{tikzpicture}\caption{}\label{figpib1}
\end{subfigure}
\hfill
\begin{subfigure}{0.45\textwidth}
\raggedright
\begin{tikzpicture}[scale=1.]
\draw[-,very thick](0,0)--(6,0)--(6,6)--(0,6)--(0,0);
\draw[-,very thick](0,2)--(6,2);
\draw[-,very thick](0,4)--(6,4);
\draw[-,very thick](2,0)--(2,6);
\draw[-,very thick](4,0)--(4,6);
\draw[-,fill=blue!20,very thick, red=90!, pattern=north east lines, pattern color=red!30](1,1)--(1,5)--(5,5)--(5,1)--(1,1);
\fill (1,3) circle (2pt) node[below left]{$x_W$};
\fill (5,3) circle (2pt) node[below right]{$x_E$};
\fill (3,1) circle (2pt) node[below right]{$x_S$};
\fill (3,5) circle (2pt) node[above right]{$x_N$};
\path (1,3) node [above left] { \cred $\ell_{W}^{(1)}$};
\path (5,3) node [above right] { \cred $\ell_{E}^{(1)}$};
\path (3,5) node [above left] { \cred $\ell_{N}^{(2)}$};
\path (3,1) node [below left] {\cred  $\ell_{S}^{(2)}$};
\fill (3,3) circle (2pt) node[below right]{$x_K$};
\path (3,3)  node[above left]{\cred $K_\dagger$};
\path (3,0)  node[below]{$\pd K_\dagger = \ell_{W}^{(1)} \cup \ell_{S}^{(2)} \cup \ell_{E}^{(1)} \cup \ell_{N}^{(2)}$};
\end{tikzpicture}\caption{}\label{figpib2}
\end{subfigure}
\caption{Extended mesh in 2D.}\label{figpib}
\end{figure}
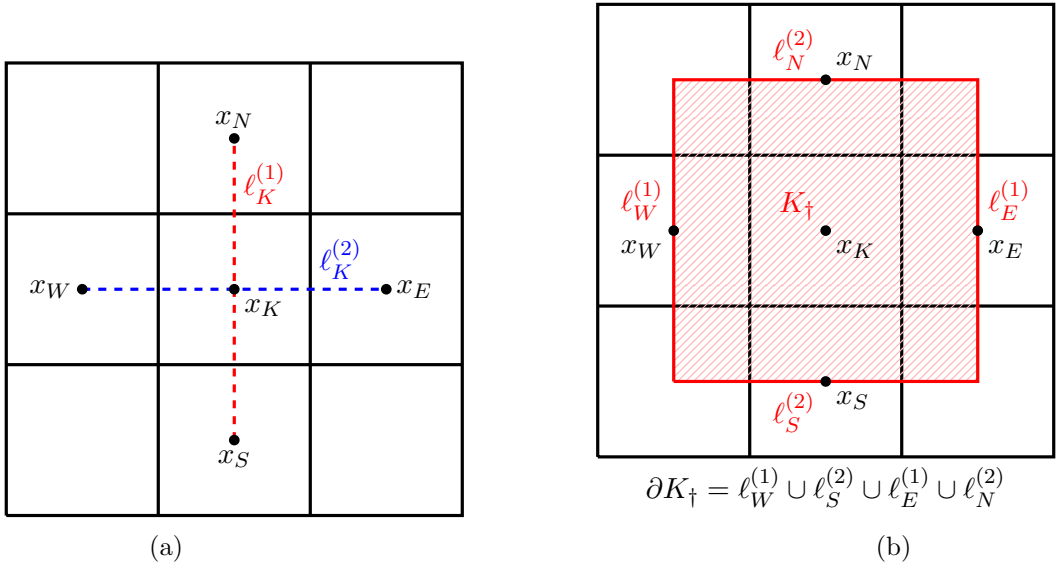

\newpage
\begin{Lemma}
Let $\PiB$ be defined as in \eqref{eqpib}.
\begin{itemize}
\item[$i)$] Let $\vB_B$ be defined as in \eqref{BB}. Then there hold $\Divh \PiB \vB_B =0$ and $\avs{\PiB \vB_B}|_{\facesext} = \vB_B|_{\facesext}$.

\item[$ii)$] Let $\vC \in C^1(Q;\R^d)$ satisfy the divergence-free constraint $\Div\, \vC = 0$. Then it holds $\Divh \PiB \vC =0$.
\end{itemize}

\end{Lemma}
\begin{proof}
The first argument $i)$ directly follows from the linear profile of $\vB_B$.
Next, we show the proof of $ii)$ in 2D, analogous calculations can be done in 3D.

Let $W, E, S, N$ be respectively the left, right, lower and upper neighbours to $K$, and let us denote $K_\dagger$ the domain closed by $\ell_{W}^{(1)}$, $\ell_{S}^{(2)}$, $\ell_{E}^{(1)}$, and $\ell_{N}^{(2)}$, see Figure \ref{figpib2}.
Then  letting $\vC = (C_1,C_2)$ we have
\begin{align*}
&(\Divh \PiB  \vC)_K =   \frac{1}{|K|}\sum_{\sigma\in \facesK}|\sigma| \vn \cdot \avg{\PiB  \vC}
= \frac{ (\PiB^{(1)} C_1 )_E  -  (\PiB^{(1)} C_1)_W } {2h}  +\frac{ (\PiB^{(2)} C_2)_N  -  (\PiB^{(2)} C_2)_S } {2h}
\\& =\frac{1}{|K_\dagger|}  \int_{\pd K_\dagger}   \PiB \vC \cdot \vn \ds=\frac{1}{|K_\dagger|}  \int_{\pd K_\dagger}   \vC \cdot \vn \ds
= \frac{1}{|K_\dagger|} \int_{ K_\dagger}  \Div \vC\dx=0,
\end{align*}
which completes the proof.
\end{proof}

\subsection{Finite volume method}\label{FV}

We are now ready to formulate our upwind finite volume (FV) method.

\begin{Definition}[{\bf FV method}]\label{def_fv}
Let $(\vrh^0, \vuh^0, \vBh^0) =(\PiQ \vr_0, \PiQ \vu_0, \PiB \vB_0)$. We say that  a triple of discrete functions $(\vrh, \vuh, \vBh) \in L_{\TS}(0, T;\Qh^{2d+1}) $
is a numerical approximation of the compressible MHD problem \eqref{pde}--\eqref{ini} if
the following holds:
\begin{subequations}\label{scheme1}
\begin{equation}\label{schemeD1}
(D_t \vrh)  _K + \sumSKIn \frac{|\sigma|}{|K|} \Fup(\vrh ,\vuh ) =0,
\end{equation}
\begin{equation}\label{schemeM1}
\begin{aligned}
 D_t (\vrh  \vuh )_K + \sumSKIn \frac{|\sigma|}{|K|}
 \Fup(\vrh  \vuh ,\vuh )
 =& \mu (\Laph \vuh)_K + \nu (\Gradh \Divh \vuh)_K - (\Gradh p_h)_K \\&+ (\Curlh \vBh  \times \vBh)_K + (\vrh)_K \, g,
\end{aligned}
\end{equation}
\begin{equation}\label{schemeB1}
( D_t \vBh )_K =  \Curlh (\vuh  \times \vBh  - \zeta \Curlh \vBh)_K
\end{equation}
for all $K \in \mesh$ with $\nu = \lambda +\mu $ and $p_h = p(\vrh)$. The above system is accompanied with the discrete boundary conditions
\begin{equation}\label{bc-num}
\avs{\vuh}_{\sigma} = 0, \quad \vn \times \avs{\vBh}_{\sigma } =  \vb^{\pm},  \quad \vn \cdot \jump{\vBh}_{\sigma } = 0, \quad
\sigma \in \facesext
\end{equation}
and artificial boundary conditions
\begin{equation}\label{extension}
 \jump{\Divh \vuh}_{\sigma} = 0, \quad \jump{p_h}_{\sigma} = 0, \quad
 \vn \times \jump{ \vuh\times \vBh - \zeta \Curlh \vBh} _{ \sigma} = 0, \quad \sigma \in \facesext.
\end{equation}
\end{subequations}
\end{Definition}

\noindent
We proceed by presenting an equivalent formulation of above FV method.

\begin{Proposition}
A triple $(\vrh,\vuh,\vBh) \in L_{\TS}(0, T;\Qh^{2d+1})$ is a solution of the FV method \eqref{scheme1} if and only if $(\vrh,\vuh,\vBh)$ solves the following weak formulation
\begin{subequations}\label{scheme}
\begin{equation}\label{schemeD}
\intQ{ D_t \vrh  \phi_h } - \intfacesint{  \Fup(\vrh  ,\vuh  )
\jump{\phi_h}   } = 0 \quad \mbox{for all } \phi_h \in Q_h,
\end{equation}
\begin{equation}\label{schemeM}
\begin{aligned}
  \intQ{ D_t (\vrh  \vuh) \cdot  \vh   } -
  \intfacesint{  \Fup(\vrh   \vuh,\vuh  ) \cdot
\jump{\vh}   }
+  \mu \intQ{ \Gradd \vuh : \Gradd \vh  } +  \nu \intQ{ \Divh \vuh  \,  \Divh \vh  }   \\
- \intQ{p_h  \Divh \vh  }  - \intQ{ (\Curlh \vBh   \times  \vBh ) \cdot \vh}
 =\intQ{\vrh \vh \cdot \vc{g}}  \quad \mbox{for all }  \vh \in Q_h^d, \ \avs{\vh}|_{\facesext} = 0,  
\end{aligned}
\end{equation}
\begin{equation}\label{schemeB}
\intQ{ D_t \vBh  \cdot  \vCh}
- \intQ{ (\vuh  \times \vBh - \zeta \Curlh \vBh) \cdot   \Curlh \vCh
 }=0 \quad \mbox{for all } \vCh \in Q_h^d, \ \vn \times \avs{\vCh}|_{\facesext} = 0 
\end{equation}
 with the discrete boundary conditions
\begin{equation}
\avs{\vuh}_{\sigma} = 0, \quad \vn \times \avs{\vBh}_{\sigma } = \vb^{\pm},  \quad \vn \cdot \jump{\vBh}_{\sigma } = 0, \quad
\sigma \in \facesext.
\end{equation}
\end{subequations}
\end{Proposition}

\begin{Remark}
We point out that FV method in the strong form, i.e.\ \eqref{scheme1}, is simpler for the implementation, whereas the weak form  \eqref{scheme} is more suitable for analysis.
\end{Remark}

\begin{proof}
We first show that if $(\vrh, \vuh, \vBh)$ is a solution of \eqref{scheme1} then $(\vrh, \vuh, \vBh)$ fulfils the weak formulation \eqref{scheme}.
It is easy to get \eqref{schemeD}  and \eqref{schemeM}, respectively, by multiplying \eqref{schemeD1}, \eqref{schemeM1} with discrete test functions and then summing over all mesh cells:  $\sumK \intK{\eqref{schemeD1} \phi_h}$ and $\sumK \intK{\eqref{schemeM1} \cdot \vh}$.   Analogously, setting $\vc{f}_h = \vuh  \times \vBh  - \zeta \Curlh \vBh$ we have
\begin{align*}
&\sumK\intK{ D_t \vBh  \cdot  \vC_h} =\sumK\intK{ \Curlh \vc{f}_h \cdot \vC_h }
\\&
= \intQ{ \vc{f}_h \cdot \Curlh  \vC_h } -\intfacesext{ \vn \times \avs{\vC_h} \cdot \vc{f}_h^{in} } +
 \frac12 \intfacesext{ \vn \times \jump{\vc{f}_h} \cdot \vC_h^{in} }
 \\&
 =\intQ{ \vc{f}_h \cdot \Curlh  \vC_h } .
\end{align*}
Here we have used the integration by parts formula \eqref{InByPa-4}, the boundary condition \eqref{extension}, and  the condition $\vn \times \avs{\vC_h}|_{\facesext} =0$.
 Consequently, we obtain  \eqref{schemeB}
\[
\intQ{ D_t \vBh  \cdot  \vC_h}
= \intQ{ \Curlh \vc{f}_h \cdot    \vC_h}= \intQ{  \vc{f}_h \cdot    \Curlh \vC_h},
\]
which implies that $(\vrh, \vm_h, \vBh)$ fulfills \eqref{scheme}.

Now, let us assume that $(\vrh, \vm_h, \vBh)$ satisfies weak formulation \eqref{scheme}.
Choosing characteristic functions as test functions in \eqref{scheme}, $\phi_h = \mathds{1}_K$, $\vv_h =\mathds{1}_K,$ $\vCh =\mathds{1}_K$ for any $K \in \mesh$, we obtain the elementwise FV formulation of \eqref{schemeD}, \eqref{schemeM} and \eqref{schemeB}. Note that in order to obtain \eqref{schemeM} we have to apply discrete integration by parts formulae \eqref{InByPa-1}, \eqref{InByPa-2} and boundary condition \eqref{bc-num} for velocity. By the same token, integration by parts formula \eqref{InByPa-4}, boundary conditions \eqref{bc-num} and \eqref{extension} lead to \eqref{schemeB}.
\end{proof}

\begin{Lemma}[Existence of a numerical solution]\label{lem-existence}
Suppose that $\vr_0>0$. Then for any $k=1,\dots,N_T,$ there exists a solution $(\vrh,\vuh,\vBh) \in L_{\TS}(0, T;\Qh^{2d+1})$ to the discrete problem \eqref{scheme}.
\end{Lemma}

\begin{proof}
The proof can be done analogously as in \cite{Ding} and \cite[Lemma 11.3]{FeLMMiSh} via the theorem of topological degree,  for more details see Appendix \ref{exist}.
\end{proof}

\subsection{Stability}
In this section, we present the stability of the FV method \eqref{scheme}, including the mass conservation, positivity of the density, the energy balance as well as the a priori uniform bounds.

\begin{Lemma}[{Mass conservation, renormalized continuity, and positivity of density\ \cite{FLMS_FVNS}}]\label{lem_b}
Let $\vrh,\vuh$ satisfy the equation \eqref{schemeD} with the initial data satisfying $\vr_0>0$.
Then there hold
\begin{enumerate}
\item {\bf Mass conservation}
\begin{equation}\label{Mcon}
\intQ{ \vrh (t) } = \intQ{ \vrh (0) } = \intQ{ \vr_0 }   \quad \mbox{ for all }\;  t \in[0,T].
\end{equation}

\item {\bf Renormalized continuity equation}
\begin{equation}\label{r1}
\begin{aligned}
&\intQ{\left(D_t b(\vrh ) -\big(\vrh  b'(\vrh )-b(\vrh )\big) \Divh \vuh  \right) }
\\&\quad\quad =
- \frac{\TS}2 \intQ{ b''(\xi_1)|D_t \vrh |^2  }
- \intfacesint{ b''(\xi_2) \jump{  \vrh  } ^2 \left(h^\eps  +  \frac12 | \us | \right) }
\end{aligned}
\end{equation}
for any $b=b(\vr)\in C^2(0,\infty)$, where $\xi_1 \in \co{\vrh^{\triangleleft}}{\vrh }$ and  $\xi_2 \in \co{\vrh^{\rm in}}{\vrh^{\rm out}}$.
\item {\bf Positivity of the density}
\begin{equation}\label{posit}
\vrh(t)>0  \quad \mbox{ for all }\;  t \in (0,T).
\end{equation}	
\end{enumerate}
\end{Lemma}

 \begin{Lemma}[Energy balance]\label{EnBa}
 Let $(\vrh, \vuh, \vBh)$ be a solution obtained by the FV scheme \eqref{scheme}. 
Then it holds
  \begin{equation}\label{energy_stability}
 \begin{aligned}
&	 D_t \intQ{ \left(\frac{1}{2}  \vrh  |\vuh |^2  + \Hc(\vrh) +\frac12 |\vBh - \PiB \vB_B|^2 \right) }
	 +  \mu \norm{\Gradd \vuh }_{L^2(Q;\R^{d\times d})}^2  + \nu \norm{ \Divh \vuh }_{L^2(Q)}^2
	\\&
	+  \zeta \norm{\Curlh \vBh }_{L^2(Q;\R^{d})}^2 + D_{num}
=
-\intQ{ (\vuh  \times \vBh - \zeta \Curlh \vBh) \cdot   \Curlh \PiB\vB_B} + \intQ{\vrh \vuh \cdot \vc{g}},
\end{aligned}
\end{equation}
where $\vB_B$ is given by \eqref{BB} and $D_{num} > 0$ represents the numerical dissipation, which reads
\begin{align*}
D_{num} = &  \frac{\TS}{2} \intQB{ \Hc''(\xi_1)|D_t \vrh |^2 + \vrh^\triangleleft|D_t \vuh |^2  + |D_t \vBh|^2 }
	+\intfacesint{ \Hc''(\xi_2)  \left(h^\eps  +  \frac12 | \us | \right) \jump{  \vrh  } ^2} \br
	&+  \intfacesint{ \left( \frac12 \vrh^{\rm up} |\us | +h^\eps \avs{ \vrh  } \right)      \abs{\jump{\vuh}}^2   }
\end{align*}
with $ \xi_1 \in \co{\vrh^{\triangleleft}}{\vrh }$ and  $\xi_2 \in \co{\vrh^{\rm in}}{\vrh^{\rm out}}$ for any $\sigma \in \facesint$.
\end{Lemma}

\begin{proof}
First, summing up \eqref{schemeD} and \eqref{schemeM}  respectively with the test functions $\phi_h = -\frac{\abs{\vuh}^2}{2}$ and $\vh = \vuh$ implies
\begin{equation}\label{ke1}
\begin{aligned}
&D_t \intQ{\frac12 \vrh  \abs{\vuh}^2   }
+  \frac{\TS}{2} \intQ{ \vrh^{\triangleleft} \abs{D_t \vuh }^2}
 + \intfacesint{ \left(\frac12 \vrh^{\up} \abs{\us}
+h^\eps \avg{\vrh}   \right)\abs{\jump{\vuh}}^2 }
\\ &
\quad +\mu  \norm{\Gradd \vuh}_{L^2(Q;\R^{d\times d})}^2 +  \nu  \norm{\Divh \vuh}_{L^2(Q)}^2  - \intQ{p_h  \Divh \vuh  } - \intQ{\vrh \vuh \cdot \vc{g}}
\\&
=  \intQ{ ( \Curlh \vBh  \times  \vBh ) \cdot \vuh }
=- \intQ{ \Curlh \vBh   \cdot (\vuh  \times \vBh ) },
\end{aligned}
\end{equation}
where we have used the following equalities
 \begin{equation*}
\begin{aligned}
& \intfacesint{ \left(\Fup(\vrh \vuh,\vuh) \cdot \jump{\vuh} -  \Fup(\vrh ,\vuh) \jump{ \frac{|\vuh|^2}{2}}  \right)}
 = - \intfacesint{ \left(\frac{1}2 \vrh^{\up} \abs{\us}
+h^\eps \avg{\vrh}   \right)\abs{\jump{\vuh}}^2 }
\end{aligned}
\end{equation*}
and
\[
 D_t (\vrh \vuh)  \cdot \vuh-  \frac{|\vuh|^2}{2} D_t \vrh
= D_t \Big(\frac12 \vrh  |\vuh |^2 \Big)   + \frac{\TS}2 \vrh^{\triangleleft}  |D_t\vuh |^2 .
 \]
%
%
%
Next, by setting $\vCh =\vBh -\PiB\vB_B $ in \eqref{schemeB} we obtain
\begin{equation*}
\intQ{ D_t \vBh  \cdot  (\vBh -\PiB\vB_B)}
- \intQ{ (\vuh  \times \vBh - \zeta \Curlh \vBh) \cdot   \Curlh (\vBh -\PiB\vB_B)} =0,
\end{equation*}
which can be reformulated as
\begin{equation}\label{ke2}
\begin{aligned}
& \intQB{ ( \vuh \times \vBh ) \cdot   \Curlh \vBh  - \zeta | \Curlh  \vBh |^2 }
\\& =\intQ{ D_t \vBh  \cdot  (\vBh-\PiB\vB_B)  }
+ \intQ{ (\vuh  \times \vBh - \zeta \Curlh \vBh) \cdot   \Curlh \PiB\vB_B}
\\&
=\intQB{ D_t \frac{|\vBh - \PiB\vB_B|^2}{2}  + \frac{\TS}2 |D_t \vBh|^2    }
+ \intQ{ (\vuh  \times \vBh - \zeta \Curlh \vBh) \cdot   \Curlh \PiB\vB_B}.
\end{aligned}
\end{equation}
Note that in the last equality we have used the fact that $\vB_B$ is independent of time $t$.

\medskip

Further, setting $b(\vr)=\Hc(\vr)$ in the renormalized continuity equation~\eqref{r1} and noticing the equality \eqref{cpp},
we obtain the balance of internal energy
\begin{equation}\label{ke3}
\intQ{ \left( D_t \Hc(\vrh )  - p_h  \Divh \vuh   \right) }
=
- \frac{\TS}2 \intQ{\Hc''(\xi_1)|D_t \vrh |^2  }
- \intfacesint{ \Hc''(\xi_2) \jump{  \vrh  } ^2 \left(h^\eps  + \frac12 |\us| \right) },
\end{equation}
where  $\xi_1 \in \co{\vrh^{\triangleleft}}{\vrh }$ and  $\xi_2 \in \co{\vrh^{\rm in} }{\vrh^{\rm out} }$.

Combining \eqref{ke1}, \eqref{ke2}, and \eqref{ke3}  completes the proof.
\end{proof}

 \begin{Lemma}[Uniform bounds]\label{UnBo}
 Let $(\vrh, \vuh, \vBh)$ be a solution obtained by the FV scheme \eqref{scheme}. Then there hold
\begin{subequations}\label{AP}
\begin{align}\label{ap1}
& \norm{\vrh}_{L^\infty(0,T; L^\gamma(Q))}  + \norm{\vrh \vuh }_{L^\infty (0,T; L^{\frac{2\gamma}{\gamma+1}}(Q;\R^d)) }  + \norm{p_h}_{L^\infty (0,T; L^1(Q))} + \br
& \quad + \norm{\vuh}_{L^2(0,T; L^{2}(Q;\R^d))} + \norm{\Gradd \vuh}_{L^2((0,T)\times Q;\R^{d\times d})} + \norm{\Divh \vuh}_{L^2((0,T)\times Q)} \br
& \quad + \norm{\vBh}_{L^\infty(0,T;L^2(Q;\R^d))} + \norm{\Curlh \vBh}_{L^2((0,T)\times Q;\R^d)}  \leq C ,
\end{align}
\begin{align}\label{ap2}
\frac{\TS}{2} \int_0^T \intQB{ \Hc''(\xi_1)|D_t \vrh |^2 + \vrh^\triangleleft|D_t \vuh |^2  + |D_t \vBh|^2 } \dt \leq C,
\end{align}
\begin{align} \label{ap3}
&  \int_0^T\intfacesint{ \Hc''(\xi_2)  \left(h^\eps  +  \frac12 | \us | \right) \jump{  \vrh  } ^2} \dt
+  \int_0^T \intfacesint{ \left( \frac12 \vrh^{\rm up} |\us | +h^\eps \avs{ \vrh  } \right)      \abs{\jump{\vuh}}^2   } \dt \leq C,
\end{align}
\end{subequations}
where $\xi_1 \in \co{\vrh^{\triangleleft}}{\vrh }$ and  $\xi_2 \in \co{\vrh^{\rm in} }{\vrh^{\rm out} }$. Here $C$ is a generic constant, depending on $\frac1{\mu}$, $\zeta$, $\frac1\zeta$, $|\vb^{\pm}|$,  $|\vc{g}|$,  $\norm{\vr_0}_{L^1(Q)}$ and $\norm{E_0}_{L^1(Q)}$ with $E_0 = \frac12 \vr_0 \abs{\vu_0}^2 +  \mathcal{P}(\vr_0) + \frac12 \abs{\vB_0}^2$.
\end{Lemma}

\begin{proof}
In what follows we derive  a priori bounds from the the energy balance \eqref{energy_stability}.
Let us decompose the right-hand-side of \eqref{energy_stability} into four parts
 \begin{align*}
&-\intQ{ (\vuh  \times \vBh - \zeta \Curlh \vBh) \cdot   \Curlh \PiB\vB_B} + \intQ{\vrh \vuh \cdot \vc{g}}= \sum_{i=1}^4 I_i,
\\&
I_1=-\intQ{ \vuh  \times (\vBh - \PiB\vB_B) \cdot   \Curlh \PiB\vB_B},\quad
I_2 = -\intQ{ \vuh  \times   \PiB\vB_B \cdot   \Curlh \PiB\vB_B},
\\&
I_3 = \intQ{  \zeta \Curlh \vBh \cdot   \Curlh \PiB\vB_B}, \quad
I_4=\intQ{\vrh \vuh \cdot \vc{g}}.
\end{align*}
Applying H\"{o}lder inequality we shall control $I_i, \,  i=1,\dots,4,$ with
\begin{align*}
\abs{I_1} &\aleq \norm{\vB_B}_{W^{1,\infty}(Q;\R^{d})} \left( \delta \norm{\vuh}_{L^2(Q;\R^d)}^2 + \frac1\delta \norm{\vBh - \PiB\vB_B}_{L^2(Q;\R^d)}^2 \right),
\br
\abs{I_2} &\aleq \norm{\vB_B}_{W^{1,\infty}(Q;\R^{d})}^2 \norm{\vuh}_{L^1(Q;\R^d)} \aleq    \delta \norm{\vuh}_{L^2(Q;\R^d)}^2 + \frac1\delta \norm{\vB_B}_{W^{1,\infty}(Q;\R^{d})}^4,
\br
\abs{I_3} &\aleq \zeta \norm{\vB_B}_{W^{1,\infty}(Q;\R^{d})}  \norm{\Curlh \vBh}_{L^2(Q;\R^d)} 
\aleq \delta \zeta  \norm{\Curlh \vBh}_{L^2(Q;\R^d)}^2 + \frac{\zeta}{\delta} \norm{\vB_B}_{W^{1,\infty}(Q;\R^{d})}^2,\\
\abs{I_4} & \aleq |\vc{g}| \norm{\vrh \vuh}_{L^1(Q;\R^d)} \aleq |\vc{g}| \norm{\sqrt{\vrh}}_{L^2(Q)} \norm{\sqrt{\vrh} \vuh}_{L^2(Q;\R^d)} = |\vc{g}|  \norm{\vrh}_{L^1(Q)}^{1/2} \norm{\vrh \abs{\vuh}^2}_{L^1(Q)}^{1/2}\\
&\aleq |\vc{g}| \norm{\vrh}_{L^1(Q)} +  |\vc{g}| \norm{\vrh \abs{\vuh}^2}_{L^1(Q)}.
\end{align*}
Thanks to $\norm{\vB_B}_{W^{1,\infty}(Q;\R^{d})} \aleq |\vb^{\pm}|$  (see \eqref{BB}) and the Poincar\'e inequality
\[
 \norm{\vuh}_{L^2(Q;\R^d)} \aleq \norm{\Gradd \vuh}_{L^2(Q;\R^{d\times d})},
\]
we can choose a suitable $\delta$ such that $0 < \delta < \min(\mu,1).$ Consequently, the energy balance \eqref{energy_stability} implies
\begin{equation}
 \begin{aligned}
&	 D_t \intQ{ \left(\frac{1}{2}  \vrh  |\vuh |^2  + \Hc(\vr) +\frac12 |\vBh - \PiB \vB_{B}|^2 \right) }
	 +  \mu \norm{\Gradd \vuh }_{L^2(Q;\R^{d\times d})}^2  + \nu \norm{ \Divh \vuh }_{L^2(Q)}^2
	\\&+  \zeta \norm{\Curlh \vBh }_{L^2(Q;\R^{d})}^2 + D_{num} \aleq C+|\vc{g}| \norm{\vrh \abs{\vuh}^2}_{L^1(Q)} + \frac{|\vb^{\pm}|}{\delta}\norm{\vBh - \PiB \vB_{B}}_{L^2(Q;\R^d)}^2,
\end{aligned}
\end{equation}
where $C$ is a generic constant, depending on $\frac1{\delta}$,\, $\zeta$,\, $|\vb^{\pm}|$, \, $|\vc{g}|$, \, $\norm{\vr_0}_{L^1(Q)}$.
Applying Gronwall's lemma finishes the proof.
\end{proof}

\subsection{Consistency}\label{sec_Consistency}
In this section we show the consistency of our FV method \eqref{scheme}, i.e. the consistency of the  divergence-free property of the magnetic field,  the consistency of the density, momentum, and energy equations and the consistency of the discrete differential operators, cf.\ the compatibility condition. Note that from now on we only present the theoretical results for $d=3$,  the results for $d=2$ directly follow.

\begin{Lemma}[Divergence-free property]\label{lem:divfree}
Let $(\vrh,\vuh,\vBh)$ be a solution to the FV scheme \eqref{scheme}. Then  it holds
\begin{equation}\label{divfree}
(\Divh \vBh) _K =0 \quad \mbox{ for all } K \in \mesh.
\end{equation}
\end{Lemma}
\begin{proof}
Let us first analyze \eqref{divfree} for an interior cell $K, K\in \grid, K\cap \pd Q = \emptyset$.   By applying the $\Divh$ operator to \eqref{schemeB1} and recalling the identity \eqref{div-curl} we have
\[\Divh \vBh|_K =\Divh \vBh(t-\TS)|_K = \cdots =  \Divh \vBh^0|_K =0.\]
The rest is to prove \eqref{divfree} for a boundary cell.

Without loss of  generality, let us consider the upper boundary $x_3 = 1$ and denote the boundary cell by $K, K\in \grid, K\cap \pd Q \neq \emptyset$. Its external (resp. internal) neighbour in $x_3$-direction is denoted by $N$ (resp.\ $S$), 
see Figure \ref{figpib-1}.
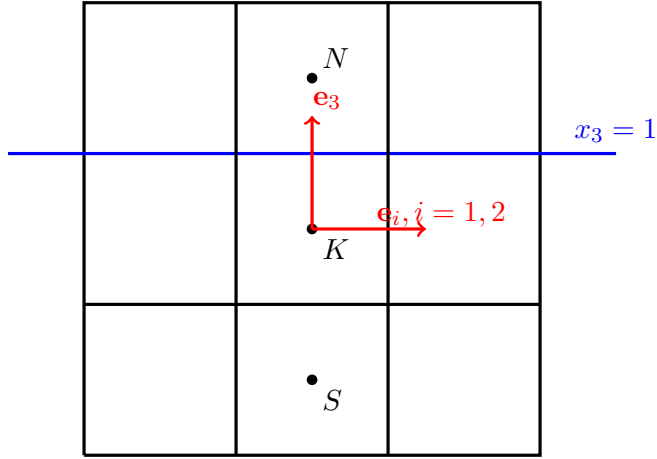
\begin{figure}[hbt]
\centering
\begin{subfigure}{0.45\textwidth}
\raggedright
\begin{tikzpicture}[scale=1.]
\draw[-,very thick](0,0)--(6,0)--(6,6)--(0,6)--(0,0);
\draw[-,very thick](0,2)--(6,2);
\draw[-,very thick,blue = 90!](-1,4)--(7,4);
\path node at (7,4.3) {$\cblue x_3 = 1$};
\draw[-,very thick](2,0)--(2,6);
\draw[-,very thick](4,0)--(4,6);

\fill (3,1) circle (2pt) node[below right]{$S$};
\fill (3,5) circle (2pt) node[above right]{$N$};
\fill (3,3) circle (2pt) node[below right]{$K$};


\draw[->,very thick, red = 90!] (3,3)--(4.5,3);
\path node at (4.7,3.2) {$\cred \ve_i, i = 1,2$};
\draw[->,very thick, red = 90!] (3,3)--(3,4.5);
\path node at (3.2,4.7) {$\cred \ve_3$};
\end{tikzpicture}
\end{subfigure}
\caption{Boundary cell $K$.}\label{figpib-1}
\end{figure}

Letting $\vc{f}_h = \vuh  \times \vBh  - \zeta \Curlh \vBh$, direct calculations give
\begin{align}\label{eq-divfree-1}
& D_t (\Divh \vBh)|_K = D_t \sum_{i=1}^{2} \pdmeshi\PiWi B_K^i + D_t \frac{  B_K^3 - B_S^3 }{2h} = \sum_{i=1}^{2} \pdmeshi\PiWi D_t  B_K^i +  \frac{  D_t  B_K^3 - D_t  B_S^3 }{2h}  \br
&=  \pdmesh^{(1)}\PiW^{(1)} \left( \pdmesh^{(2)}\PiW^{(2)} {f}_K^3 - \pdmesh^{(3)}\PiW^{(3)} {f}_K^2 \right)  + \pdmesh^{(2)}\PiW^{(2)} \left( \pdmesh^{(3)}\PiW^{(3)} {f}_K^1 - \pdmesh^{(1)}\PiW^{(1)} {f}_K^3 \right)  \br
&\quad + \frac{1}{h} \Big( \left( \pdmesh^{(1)}\PiW^{(1)} {f}_K^2 - \pdmesh^{(2)}\PiW^{(2)} {f}_K^1 \right)- \left( \pdmesh^{(1)}\PiW^{(1)} \vc{f}_S^2 - \pdmesh^{(2)}\PiW^{(2)} {f}_S^1 \right) \Big) \br
& = - \pdmesh^{(1)}\PiW^{(1)} \left( \pdmesh^{(3)}\PiW^{(3)} {f}_K^2 - \frac{{f}_K^2 - {f}_S^2}{2h} \right)+ \pdmesh^{(2)}\PiW^{(2)} \left( \pdmesh^{(3)}\PiW^{(3)} {f}_K^1 - \frac{{f}_K^1 - {f}_S^1}{2h}  \right),
\end{align}
where we have applied the boundary condition $\vn \cdot \jump{\vBh}|_{\facesext} = 0$ for the first equality,  and the equation \eqref{schemeB1} for the third equality. Here, $f^i$ (resp.\ $B^i$) means the $i$-th component of $\vc{f}$ (resp.\ $\vB$). On the other hand, the boundary conditions \eqref{extension}, i.e. $\vn \times \jump{\vc{f}_h}|_{\facesext} = 0$, give $f^{i}_N = f^{i}_K, i =1,2$. Consequently, we obtain $D_t (\Divh \vBh)|_K = 0$ with
\begin{align*}
& \pdmesh^{(3)}\PiW^{(3)} {f}_K^2 - \frac{{f}_K^2 - {f}_S^2}{2h} = \frac{1}{2h} \left( {f}_N^2 - {f}_S^2 - ({f}_K^2 - {f}_S^2)\right) =  \frac{1}{2h}  \left( {f}_N^2  - {f}_K^2 \right) = 0, \\
 &\pdmesh^{(3)}\PiW^{(3)} {f}_K^1 - \frac{{f}_K^1 - {f}_S^1}{2h}  = \frac{1}{2h} \left( {f}_N^1 - {f}_S^1 - ({f}_K^1 - {f}_S^1)\right) =  \frac{1}{2h}  \left( {f}_N^1  - {f}_K^1 \right)  = 0,
\end{align*}
and finish the proof of the divergence-free constraint.
\end{proof}

 \begin{Lemma}[Consistency of the density, momentum and energy equations]\label{Tm2}
 Let $(\vrh, \vu_h,\vBh)$ be a solution obtained by the FV method \eqref{scheme} with $(\TS,h)\in (0,1)^2$. 
 Then for any $t^n$, $n=0,\cdots,N_T$, there exist some positive constants $\beta_D, \beta_M$ such that
\begin{subequations}\label{AcP}
 \begin{equation} \label{AcP1}
\intTnO{ \Big( \vrh \partial_t \phi + \vrh \vuh \cdot \Gradh \phi \Big)}
= \left[  \intQ{ \vrh \phi }  \right]_0^{t^n-} + e_{\vr} (t^n, \phi),
 \end{equation}
 for any $\phi \in C^2([0,T] \times \Ov{Q})$;
 \begin{equation} \label{AcP2}
 \begin{split}
  &
\intTnO{ \Big( \vrh \vuh \cdot \partial_t \vv + \vrh \vuh \otimes \vuh  : \Grad \vv  + p_h \Div \vv \Big)}
  -   \intTnO{   \left( \mu \Gradd \vuh + \nu \Divh \vuh \bI \right) : \Grad \vv}
 \\& +  \intTnO{ (\Curlh \vBh \times \vBh ) \cdot  \vv} + \intTnO{\vrh \vc{g} \cdot \vv}
=\left[\intQ{ \vrh {\vuh} \cdot \vv } \right]_0^{t^n-}  + e_{\vm} (t^n, \vv) ,
 \end{split}
 \end{equation}
 for any $\vv \in C^2([0,T] \times \Ov{Q}; \R^3)$, $\vv|_{\pd Q}=0$;
 \begin{equation} \label{AcP3}
   \intTnO{ \Big( \vBh \cdot \pdt \vC -  \zeta  \Curlh \vBh \cdot \Curl\vC +( \vuh \times \vBh )\cdot \Curl\vC \Big)}
=  \left[\intQ{ \vBh \cdot \vC}\right]_0^{t^n-} +
 e_{\vB} (t^n, \vC)
 \end{equation}
  for any $\vC \in C^2([0,T]\times \Ov{Q};\R^3 ),\, \vn \times \vC |_{\p Q}=0$, where
  \begin{equation}\label{CS1}
\abs{e_\vr(t^n, \phi) } \leq   C_\vr \big(\TS  + h^{1+\eps}  +  h^{1+\beta_D} \big),
\end{equation}
\begin{align}\label{CS2}
& \abs{e_{\vm}(t^n,  \vv) } \leq C_{\vm}  \big(\TS + h + h^{1+\eps}  + h^{1+\beta_M}   \big),
\end{align}
\begin{align}\label{CS3}
& \abs{e_{\vB}(t^n,  \vC) } \leq C_{\vB}  h
\end{align}
with $\beta_D
,\beta_M$ defined by
\begin{align*}
& \beta_D =
\begin{cases}
\min\left\{ \frac{\eps+2}{3} , 1 \right\} \cdot \frac{3(\gamma-2)}{2\gamma} & \mbox{if } \gamma \in(1,2), \\
0 &\mbox{if } \gamma \geq 2,
\end{cases}
\quad
\beta_M =
\begin{cases}
\max \left\{ - \frac{\eps+2}{2\gamma}, \frac{\gamma-3}{\gamma},-\frac{3}{2\gamma} \right\}
&
\mbox{ if }  \gamma \leq 2, \\
 \frac{\gamma-3}{\gamma}
&  \mbox{ if } \gamma \in (2,3),\\
0
& \mbox{ if }  \gamma \geq 3.
\end{cases}
\end{align*}
Here $C_{\vr}, C_{\vm}, C_{\vB}$ are generic constants that depend on $\mu$, $\frac1{\mu}$, $\nu$,  $\zeta$, $\frac1\zeta$, $|\vb^{\pm}|$,  $|\vc{g}|$,  $\norm{\vr_0}_{L^1(Q)}$, and $\norm{E_0}_{L^1(Q)}$  but  are independent of the disretization parameters $(\TS, \, h)$.
\end{subequations}
 \end{Lemma}
 \begin{Remark}
Note that above consistency formulations also hold for $d=2$ but with different $\beta_D$ and $\beta_M$. For more details we refer to  \cite[Chapter 11]{FeLMMiSh}.
 \end{Remark}

\begin{proof}
Following \cite[Lemma 2.9]{FLS_IEE} we obtain \eqref{AcP1} and \eqref{CS1} for the density and
\begin{equation*}
\begin{aligned}
 0&= \intTnO{ D_t (\vrh  \vuh) \cdot  \vh   } -
 \intTn \intfacesint{  \Fup(\vrh   \vuh,\vuh  ) \cdot \jump{\vh}   } \dt + \mu \intTnO{  \Gradd \vuh  : \Gradd \vh  }
 \\&\quad
 +  \intTnO{( \nu \Divh \vuh  -p_h ) \Divh \vh  }
 - \intTnO{ (\Curlh \vBh   \times  \vBh ) \cdot \vh}
 -\intTnO{\vrh \vh \cdot \vc{g}}
 \\ &= \left[ \intQ{ \vrh \vuh \cdot \vv }\right]_0^{t^n-}
 - \intTnO{ \left[ \vrh \vuh \cdot \partial_t \vv + \vrh \vuh \otimes \vuh  : \Grad \vv  + p_h \Div \vv \right]},
\\& \quad
 +   \intTnO{  \left( \mu \Gradd \vuh + \nu \Divh \vuh \bI \right) : \Grad \vv}
  - \intTnO{ (\Curlh \vBh   \times  \vBh ) \cdot \vv}
\\& \quad -\intTnO{\vrh \vv \cdot \vc{g}}
  + e_{\vm}
\end{aligned}
\end{equation*}
with $\vv_h = \PiQ \vv$ and
\begin{align*}
\abs{e_{\vm}(\tau,  \vv) } &\aleq \TS + h + h^{1+\eps}  + h^{1+\beta_M}    + \abs{\intTnO{ (\Curlh \vBh   \times  \vBh ) \cdot (\vv - \vv_h)}} \br
& \quad + \abs{\intTnO{\vrh (\vv-\vv_h) \cdot \vc{g}} } \leq  C_{\vm}  \big(\TS + h + h^{1+\eps}  + h^{1+\beta_M}   \big).
\end{align*}
Here $C_{\vm}$ is a generic constant depending on model parameters $\mu, \zeta$ and uniform bounds stated in Lemma~\ref{UnBo}.

\medskip

We are left with the proof of \eqref{AcP3} and \eqref{CS3}. To proceed, we set $\vCh =\PiQ \vC$ together with the extension $\vn \times \avs{\vCh} |_{\facesext}=0$ 
as the test function in \eqref{schemeB} and analyze each term. First, for the time derivative term we have
\[ \begin{aligned}
&  \intTnO{D_t \vBh \cdot \PiQ \vC}
=
- \intTnO{ \vBh\cdot \pdt \vC }
+  \left[\intQ{ \vBh \cdot \vC}\right]_0^{t^n-}
+I_1+I_2,
\end{aligned}
\]
where
\[I_1 = \intTnO{ \vBh(t)\cdot (\pdt \vC - D_t  \PiQ \vC) }, \quad
I_2 =   \intQ{ \vBh^0\cdot \left(\vC(0) - \int_0^{\TS}\PiQ \vC(t)\dt \right)}.
\]
By H\"older's inequality, the estimates \eqref{ap1} and interpolation error estimate we have
\[
\abs{ I_1 } + \abs{I_2}\aleq  \TS \norm{\vBh}_{L^{\infty}L^2} \norm{\vC}_{C^2} \aleq \TS.
\]
Here, we have used the abbreviation $L^p L^q$ for the space $L^p(0,T; L^q(Q))$ for any $p,q\geq 1$. Similar notation holds for $C^p$.
Next, the advection term and the diffusion term shall be controlled analogously with 
\[
 \intTO{ ( \vuh\times \vBh ) \cdot (\Curlh \PiB \vC -  \Curl \vC)}
 \aleq  h \norm{\vuh}_{L^2L^2}  \norm{\vBh}_{L^\infty L^2} \norm{\vC}_{C^2} \aleq h
\]
and
\[
 \int_0^T  \intQ{ \Curlh \vBh  \cdot ( \Curlh \PiB \vC -  \Curl \vC)}\dt
 \aleq  h\norm{\vC}_{C^2} \norm{\Curlh \vBh}_{L^2L^2}
 \aleq h.
\]
Consequently, summing up the above estimates finishes the proof of \eqref{AcP3} and \eqref{CS3}.
\end{proof}

\begin{Lemma}[Compatibility condition]\label{Tm3}
Let $\vBh \in L_{\TS}(0, T; Q_h^3)$ satisfy following boundary condition
\begin{equation}
\vn \times \avs{\vBh}_{\sigma } =  \vb^{\pm},  \quad 
\sigma \in \facesext
\end{equation}
and stability conditions
\[
\norm{\vBh}_{L^\infty(0,T; L^2(Q;\R^3))} \aleq 1, \quad \norm{\Curlh \vBh}_{L^2((0,T)\times Q;\R^3)} \aleq 1.
\]
Let $\vB\in L^2(0,T; W^{1,2}(Q;\R^3))$ satisfy the boundary condition
\begin{equation}
	(\vn \times \vB)|_{x_3 = \pm 1} = \vb^{\pm}.
\end{equation}

Then it holds  for any $\tau \in [0,T]$ that
%
 \begin{align}\label{AcP4}
 \intTauO{\Big((\vB-\vBh) \cdot  \Curl \vC -  (\Curl \vB-\Curlh \vBh) \cdot   \vC  \Big)}  = e_{\Curl \vB} \quad \mbox{with} \quad
\abs{e_{\Curl \vB}} \leq C_{\Curl \vB} h^{1/2}
\end{align}
 for any $\vC \in L^2(0,T; W^{2,2}( Q;\R^3))$.
Here $C_{\Curl \vB}$ is a generic constant depending on  $\norm{\vBh}_{L^\infty(0,T; L^2(Q;\R^3))}$ and  $\norm{\Curlh \vBh}_{L^2((0,T)\times Q;\R^3)}$ but  independent of the discretization parameters $(\TS, \, h)$.
\end{Lemma}

\begin{proof}
The proof can be done by applying the Stokes theorem, the integration by part formula \eqref{InByPa-4} as well as the interpolation error estimates, for more details see Appendix  \ref{compatibility}.
\end{proof}

\begin{Remark}\label{rmk:compat}
Following the proof of Lemma \ref{Tm3} in Appendix  \ref{compatibility}, we shall also control $e_{\Curl \vB}$ with
\begin{equation}
\abs{e_{\Curl \vB}} \aleq  h + \intTauO{\RE \Big((\vrh,\vuh,\vBh)\,\Big|\,(\vr, \vu,\vB)\Big)}
\end{equation}
since
\begin{align*}
\abs{I_2} & \aleq h+  \int_0^{\tau} \int_{O}{\abs{\vBh}} \dxdt \aleq h+  \int_0^{\tau} \int_{O}{\left( 1+ \abs{\vBh - \vB}^2 \right)} \dxdt \br
&\aleq h + \intTauO{\RE \Big((\vrh,\vuh,\vBh)\,\Big|\,(\vr, \vu,\vB)\Big)}
\end{align*}
where $O$ is the set of boundary cells and satisfies $\abs{O} \aleq h$.
\end{Remark}

\begin{Remark}\label{Tm4}
Note that the compatibility conditions also hold for $\Gradd \vuh$ and $\Divh \vuh$. Specifically,
letting $\vuh \in L_{\TS}(0, T;Q_h^3)$ satisfy
\begin{equation*}
\avs{\vuh}_{\sigma } = 0, \quad
\sigma \in \facesext; \quad \norm{\vuh}_{L^2((0,T)\times Q;\R^3)} \aleq 1, \quad \norm{\Gradd \vuh}_{L^2((0,T)\times Q;\R^{3\times 3})} \aleq 1,
\end{equation*}
we have for any $\tau \in [0,T]$ that
\begin{align*}
& \intTauO{ \big( \Gradd \vuh: \Grad \vv  +  \vuh \cdot \Lap \vv   \big)} = e_{\Grad \vu} \quad \mbox{with} \quad
\abs{e_{\Grad \vu}} \leq C_{\Grad \vu} h, \\
& \intTauO{   \big(  v \, \Divh \vuh   + \Grad v  \cdot \vuh   \big) } = e_{\Div \vu} \quad \mbox{with} \quad
\abs{e_{\Div \vu}} \leq C_{\Div \vu} h
\end{align*}
for any $\vv \in L^2(0,T; W^{2,2}( Q;\R^3)), v\in L^2(0,T; W^{2,2}( Q))$, and $C_{\Grad \vu}, C_{\Div \vu}$ are generic constants depending on  $\norm{\vuh}_{L^2((0,T)\times Q;\R^3)}$ and $\norm{\Gradd \vuh}_{L^2((0,T)\times Q;\R^{3\times 3})}$ but  independent of the discretization parameters $(\TS, \, h)$.
For more details see \cite[Section 3]{FLS_IEE}.
\end{Remark}

\section{Error estimates for the deterministic problem}
\label{ERE}

%
%

Having shown the stability and consistency of the FV method \eqref{scheme}, we proceed to analyze the error between a numerical approximation $(\vrh,\vuh,\vBh)$ and the strong solution $(\vr,\vu,\vB)$. To this end, we use the Bregman  distance in term of the relative energy, see \eqref{S21}, introduced in Section \ref{REE}.
Here, it is more convenient to write the relative energy in terms of the velocity rather than the momentum, i.e.
\begin{equation*}
\RE\Big((\vrh,\vuh,\vBh)\,\Big|\,(\vr,\vu,\vB)\Big)
=\frac{1}{2}\vrh |\vuh-\vu|^2
+\Hc(\vrh)-\Hc(\vr)-\Hc'(\vr)(\vrh-\vr)+\f{1}{2} |\vBh-\vB|^2.
\end{equation*}

\begin{Theorem}[{\bf Deterministic error estimates}] \label{Tdis}
Let $Q \subset \R^3$ be a bounded domain of class $C^{k+1}$, $k \geq 6$ and the initial data satisfy
\[(\vr_0,\vu_0,\vB_0) \in W^{k,2}(Q;\R^7), \ \vr_0>0, \ \Div \vB_0 =0, \ \vu_0|_{\pd Q} = 0, \ \vn \times\vB_0|_{\pd Q} = \vb^{\pm}\]
 and the compatibility conditions \eqref{S8}. 
 Let $(\vr,\vu,\vB) \in C([0,T_M];W^{k,2}(Q;\R^7))
$  be the strong solution to the MHD system \eqref{pde}--\eqref{ini} specified in Proposition \ref{SP3}.
Let $(\vrh,\vuh,\vBh)$ be a numerical solution obtained by the FV method \eqref{scheme}.

Then for any  $(h,\TS)\in(0,1)^2$ and any $M>0$ the following error estimate holds
\begin{equation}\label{re0}
\begin{aligned}
 \RE \Big((\vrh,\vuh,\vBh)\,\Big|\,(\vr,\vu,\vB)\Big)(t)
 \leq C (\TS + h^{\alpha}) \mbox{ for } t \in [0,T_M  \wedge T ],
\end{aligned}
\end{equation}
where $C=C\left( \mu, \frac1{\mu}, \nu, \zeta, \frac1\zeta,  |\vc{b}^{\pm}|,  |\vc{g}|, \| \vr_0^{-1} \|_{L^\infty(Q)}, \norm{\vr_0,\vu_0,\vB_0}_{W^{k,2}(Q;\R^7)},M,T\right)$, $T_M$ is the stopping time of the exact solution at the level $M$ introduced in \eqref{S14} and
\begin{equation}
\label{alpha}
\alpha = \min \{ 1, 1+\eps, 1+\beta_M\}.
\end{equation}
\end{Theorem}

\begin{Remark}
For plasma models $\gamma=5/3$ is typically used. The optimal convergence rate is then
$\alpha= 7/13$ with $\eps = -6/13$.
\end{Remark}
\begin{Remark}
Analogously as in \cite[Lemma C.1]{FLS_IEE}, estimates \eqref{re0} and \eqref{REb} imply the following error estimates in the energy space
\begin{equation}\label{error}
\begin{aligned}
&\norm{\vrh - \vr}_{L^\infty L^\gamma}\aleq \TS^{1/r}+ h^{\alpha /r}, \ \  r= \max\{2, \gamma \},
\\&
\norm{\vm_h -\vm}_{L^\infty L^{\frac{2\gamma}{\gamma+1}}} + \norm{\vBh -\vB}_{L^\infty L^{2}} \aleq \TS^{1/2}+  h^{\alpha/2},
\\&
\norm{\vu_h -\vu}_{L^2 L^2} + \norm{\Gradd\vuh - \Grad\vu}_{L^2 L^2} + \norm{\Curlh\vBh - \Curl\vB}_{L^2 L^2} \aleq \TS^{1/2}+  h^{\alpha/2}.
\end{aligned}
\end{equation}

\end{Remark}

\begin{proof}[Proof of Theorem \ref{Tdis}.]
First, let us rewrite the relative energy
\begin{align*}
\RE\Big((\vrh,\vuh,\vBh)\,\Big|\,(\vr,\vu,\vB)\Big)
&= \left(\frac{1}{2}  \vrh  |\vuh |^2  +
\Hc(\vrh) +\frac12 |\vBh - \PiB \vB_{B}|^2 \right) + \vrh \cdot \left( \frac12 \abs{\vu}^2 - \Hc'(\vr)\right) \br
&+ \vrh \vuh \cdot (-\vu) + \vB_h\cdot \left( \vB_B - \vB \right) + \vB_h\cdot \left(  \PiB \vB_{B} - \vB_B \right) \br
&+p(\vr) + \frac{\vB^2 - (\PiB \vB_{B})^2}{2}.
\end{align*}
Accordingly, we collect the energy balance \eqref{energy_stability}, the density consistency formulation \eqref{AcP1} with the test function $\phi = \frac12 \abs{\vu}^2 - \Hc'(\vr)$, the momentum consistency formulation \eqref{AcP2} with the test function $\vv = -\vu$, and the magnetic consistency formulation \eqref{AcP3} with the test function $\vC = \vB_B - \vB$.
Summing them up, we obtain for any $\tau = t^n-, n=1,\cdots, N_T$
\begin{align}\label{REb-1}
 &\left[\intQ{ \RE\Big((\vrh,\vuh,\vBh)\,\Big|\,(\vr,\vu,\vB)\Big)}  \right]_{t=0}^{t=\tau}  +  \intTauO{ \left( \mu \abs{\Gradd \vuh}^2 +  \nu \abs{\Divh \vuh}^2 \right)}
 \br
& \hspace{2cm} + \intTauO{ \zeta \abs{\Curlh \vBh }^2} + \int_0^{\tau} D_{num}\, \dt
\br
 & = -\intTauO{ (\vuh  \times \vBh - \zeta \Curlh \vBh) \cdot   \Curlh \PiB\vB_B} + \intTauO{\vrh \vuh \cdot \vc{g}}\br
&+ \intTauO{\left(\vrh \pd_t \frac{\abs{\vu}^2}{2} + \vrh \vuh \cdot \Grad \frac{\abs{\vu}^2}{2} \right)}
- \intTauO{\Big( \vrh \pd_t \Hc'(\vr) + \vrh \vuh \cdot \Grad \Hc'(\vr) \Big)} - e_{\vr} \left(\tau, \phi \right)
\br
& - \intTauO{\Big( \vrh \vuh  \cdot \pd_t \vu + \vrh \vuh \otimes \vuh : \Grad \vu + p_h \Div \vu -   \mu \Gradd \vuh : \Grad \vu - \nu \Divh \vuh \Div \vu \Big)}
\br
& - \intTauO{\Big(  \Curlh \vBh \times \vBh \cdot \vu \Big)} - \intTauO{\vrh \vu \cdot \vc{g}}  - e_{\vm} (\tau,  \vv)
\br
& + \intTauO{\Big(  \vBh \cdot \pdt (\vB_B - \vB)
   -  \zeta  \Curlh \vBh \cdot \Curl (\vB_B - \vB) +( \vuh \times \vBh )\cdot \Curl (\vB_B - \vB) \Big)} - e_{\vB}(\tau, \vC)
\br
& +\left[\intQ{\vB_h\cdot \left(  \PiB \vB_{B} - \vB_B \right)}  \right]_{t=0}^{t=\tau}
+\intTauO{ \pd_t \Big( p(\vr) +  \frac{\vB^2 - (\PiB \vB_{B})^2}{2}\Big)}.
\end{align}

Second, we follow the calculations in  \cite[Section 3]{FLS_IEE} and use
\begin{align*}
&\pd_t \vr + \Div(\vr \vu)=0, \ \  \vr(\pd_t  \vu + \vu \cdot \Grad \vu ) + \Grad p= \Div \bS+\Curl\vB\times \vB + \vr \vc{g}, \ \
 \pd_t \vB=\Curl (\vu \times \vB -  \zeta \Curl \vB), \\
& \intQ{\Big( \vu \cdot \Grad p(\vr) + p(\vr) \Div \vu \Big)} = 0, \quad \intQ{ \Big( \Div \bS(\Grad\vu) \cdot \vu + \bS(\Grad\vu) : \Grad \vu \Big) } = 0
\end{align*}
to obtain
\begin{align*}
 &\left[\intQ{ \RE\Big((\vrh,\vuh,\vBh)\,\Big|\,(\vr,\vu,\vB)\Big)}  \right]_{t=0}^{t=\tau}  +  \intTauO{ \left( \mu \abs{\Gradd \vuh- \Grad \vu}^2 +  \nu \abs{\Divh \vuh- \Div \vu}^2 \right)}
 \br
& \hspace{2cm} + \intTauO{ \zeta\abs{\Curlh \vBh - \Curl \vB}^2} + \int_0^{\tau} D_{num}\, \dt  = e_S + \sum_{i=1}^7 R_i
\end{align*}
with
\begin{equation*}
\begin{aligned}
e_S &= -e_{\vr} \left(\tau, \phi \right) - e_{\vm} (\tau,  \vv) - e_{\vB}(\tau, \vC)+ \left[\intQ{\vB_h\cdot \left(  \PiB \vB_{B} - \vB_B \right)}  \right]_{t=0}^{t=\tau},\\
R_1 & =
\intTauO{(\vrh-\vr)  (\vu - \vuh) \cdot \left(\pd_t \vu + \vu \cdot \Grad \vu  + \frac{\Grad p(\vr)}{\vr} \right) },
\\R_2 & =
- \intTauO{ \vrh (\vuh-\vu) \otimes (\vuh-\vu) : \Grad \vu},
\\ R_3 & =
 -  \mu \intTauO{ \big( \Gradd \vuh: \Grad \vu  +  \vuh \cdot \Lap \vu   \big)},
\\ R_4 & =
- \nu \intTauO{   \big(   \Divh \vuh \Div \vu + \vuh \cdot \Grad \Div \vu  \big) },
\\  R_5 & =
- \intTauO{\Big(p_h - p'(\vr) (\vrh -\vr) -p(\vr) \Big) \Div \tvu},\br
R_6 &= \intTauO{\vrh (\vuh-\vu) \cdot \vc{g}} + \intTauO{(\vu - \vuh) \cdot \vr \vc{g}} = \intTauO{(\vrh-\vr) (\vuh-\vu) \cdot \vc{g}} ,
\end{aligned}
\end{equation*}
and
\begin{align}
R_7 =& -\intTauO{ (\vuh  \times \vBh - \zeta  \Curlh \vBh) \cdot   \Curlh \PiB\vB_B}  - \intTauO{\Big(  \Curlh \vBh \times \vBh \cdot \vu \Big)}  \br
&+ \intTauO{\Big(  \vBh \cdot \pdt (\vB_B - \vB)
   -  \zeta  \Curlh \vBh \cdot \Curl ( \vB_B - \vB) +( \vuh \times \vBh )\cdot \Curl (\vB_B - \vB) \Big)}
   \br
&+ \intTauO{ \pd_t \Big(  \frac{\vB^2 - |\PiB \vB_{B}|^2}{2}\Big)}
\br
&+\intTauO{\Big( \zeta \abs{ \Curl \vB}^2 - 2 \zeta \Curlh \vBh \Curl \vB \Big)}+ \intTauO{(\vu - \vuh) \cdot \Curl \vB \times \vB}
\label{R7}
\\
=&-\intTauO{ (\vuh  \times \vBh - \zeta \Curlh \vBh) \cdot   \Big( \Curlh \PiB\vB_B - \Curl \vB_B\Big)} \br
&+ \intTauO{\Big( \vu  \times (\vBh-\vB) \cdot  (\Curlh \vBh-\Curl \vB) + (\vu  -  \vuh) \times (\vBh-\vB) \cdot \Curl \vB  \Big)}
    \br
&+ \intTauO{\Big((\vB-\vBh) \cdot  \Curl (\vu \times \vB-  \zeta \Curl  \vB) -  (\Curl \vB-\Curlh \vBh) \cdot   (\vu \times \vB-  \zeta \Curl  \vB)  \Big)}. \label{R7-1}
\end{align}
Compared to \cite[Section 3]{FLS_IEE}, the term $R_7$ is new and need to be estimated. Suitable reformulations are  presented in Appendix~\ref{sec-R7}.  

\medskip

Third, thanks to the regularity of the strong solution, uniform bounds \eqref{AP}, the  H\"older inequality and interpolation error estimates we derive the following estimates: 
\begin{align*}
\abs{e_S} & \aleq  \TS+ h^{\alpha} +h, \; \alpha= \min\{1, 1+\eps, 1+\beta_M \}, \mbox{ see consistency errors  \eqref{CS1}--\eqref{CS3}}, \br
\abs{R_1} &+ \abs{R_6} \aleq h^2 + \intTauO{\RE \Big((\vrh,\vuh,\vBh)\,\Big|\,(\vr, \vu,\vB)\Big)} + \delta \norm{\Gradd \vuh - \Grad \vu }_{L^2((0,T)\times Q;\R^d)}^2, \br
\abs{R_2} & + \abs{R_5}\aleq \intTauO{\RE \Big((\vrh,\vuh,\vBh)\,\Big|\,(\vr, \vu,\vB)\Big)}, \br
\abs{R_3} & + \abs{R_4}  \aleq h,  \mbox{ see Remark  \ref{Tm4}}, \br
\abs{R_7} & \aleq h + \intTauO{\RE \Big((\vrh,\vuh,\vBh)\,\Big|\,(\vr, \vu,\vB)\Big)}  + \norm{\vBh - \vB}_{L^2((0,T)\times Q;\R^d)}^2 \br
&\hspace{1cm} + \delta \norm{\Curlh\vBh - \Curl\vB}_{L^2((0,T)\times Q;\R^d)}^2+ \delta \norm{\Gradd\vuh - \Grad\vu}_{L^2((0,T)\times Q;\R^d)}^2 ,
\end{align*}
where we have applied \cite[Lemma B.4]{FLS_IEE} in the estimates of $R_1, R_6$, and Remark~\ref{rmk:compat}  for the estimates of the third term of $R_7$.

\medskip

Finally, choosing a fixed positive $\delta < \min(\mu, \zeta,1)$ we obtain the following relative energy inequality
\begin{align}\label{REb}
 &\left[\intQ{ \RE\Big((\vrh,\vuh,\vBh)\,\Big|\,(\vr,\vu,\vB)\Big)}  \right]_{t=0}^{t=\tau}  +  \intTauO{ \left( \mu \abs{\Gradd \vuh- \Grad \vu}^2 +  \nu \abs{\Divh \vuh- \Div \vu}^2 \right)}
 \br
& + \intTauO{ \zeta\abs{\Curlh \vBh - \Curl \vB}^2} \,  \aleq  \intTauO{\RE \Big((\vrh,\vuh,\vBh)\,\Big|\,(\vr, \vu,\vB)\Big)} +\TS + h^{\alpha} .
\end{align}
Applying Gronwall's lemma leads to  \eqref {re0}  for $\tau = t^n-$. In order to get  \eqref {re0}  for any $t  \in [t^{n-1},t^n), n=1, \cdots, N_T+1$, we need to  estimate
\begin{align*}
&\intQ{\left(\RE(t) - \RE(t^n-)  \right)}
\\&=
\intQB{\frac{1}{2}\vrh^{n-1} |\vuh^{n-1}-\vu(t)|^2 +\Hc(\vrh^{n-1})-\Hc(\vr(t))-\Hc'(\vr(t))(\vrh^{n-1}-\vr(t))+\f{1}{2} |\vBh^{n-1}-\vB(t)|^2}
\\&  - \intQB{\frac{1}{2}\vrh^{n-1} |\vuh^{n-1}-\vu(t^n-)|^2 +\Hc(\vrh^{n-1})-\Hc(\vr(t^n-))-\Hc'(\vr(t^n-))(\vrh^{n-1}-\vr(t^n-))+\f{1}{2} |\vBh^{n-1}-\vB(t^n-)|^2} .
\end{align*}
This can be done by using the triangular inequality and  regularity of the strong solution, which gives the estimate $\abs{\intQ{\left(\RE(t) - \RE(t^n-)  \right)}} \aleq {\cal O} (\TS)$. Combining above estimates finishes the proof.
\end{proof}

\section{Convergence and error estimates for the Monte Carlo FV method}
\label{G}

In this section we present the main results on the error estimates of the Monte Carlo FV method.
They are obtained as a direct combination of the statistical error estimates \eqref{r3a}
and \eqref{r5}, and the discretization error bounds established in Theorem
\ref{Tdis}.

\subsection{Error estimates up to a stopping time}

Consider $(D_n)_{n=1}^\infty$ a sequence of i.i.d.\ copies of random data $D$, with
the associate exact solution $(\vr, \vm = \vr \vu, \vB)[D]$ and the FV approximate sequence $(\vr, \vm = \vr \vu, \vB)_h [D]$. Note that strong solution
$(\vr, \vu, \vB)[D]$ is {\it a priori} defined up to the time $T_{\rm max}$, while the approximate solutions $(\vr, \vu, \vB)_h [D]$ are evaluated on a given
time interval $[0,T]$.

Our goal is to estimate  the total error
\[
\frac{1}{N} \sum_{n=1}^N (\vr, \vu, \vB)_h ( t \wedge T_M ) [D_n] - \expe{
(\vr, \vu, \vB) ( t \wedge T_M )[D] },\ 0 \leq t \leq T,
\]
where $T_M$ is the stopping time of the exact solution at the level $M$ introduced in \eqref{S14}.
We proceed by the decomposition of the total error into two  parts: the statistical and discretization errors
\begin{align}
\frac{1}{N} &\sum_{n=1}^N (\vr, \vu, \vB)_h ( t \wedge T_M ) [D_n] - \expe{
	(\vr, \vu, \vB)( t \wedge T_M ) [D] } \br &=
\frac{1}{N} \sum_{n=1}^N (\vr, \vu, \vB)_h ( t \wedge T_M ) [D_n]
- \frac{1}{N} \sum_{n=1}^N (\vr, \vu, \vB) ( t \wedge T_M ) [D_n]
\br&+ \frac{1}{N} \sum_{n=1}^N (\vr, \vu, \vB) ( t \wedge T_M ) [D_n] - \expe{
	(\vr, \vu, \vB)( t \wedge T_M ) [D] }.
	\nonumber	
\end{align}
On the one hand, by virtue of \eqref{r3a},
\begin{align} \label{ess1}
&\expe{ \left\| \frac{1}{N} \sum_{n=1}^N (\vr, \vu, \vB) ( t \wedge T_M ) [D_n]
- \expe{
	(\vr, \vu, \vB)( t \wedge T_M ) [D]  } \right\|_{L^\gamma \times L^{\frac{2 \gamma}{\gamma + 1}} \times L^2 (Q; \mathbb \R^7)}} \br
&\leq \expe{ \left\| \frac{1}{N} \sum_{n=1}^N (\vr, \vu, \vB) ( t \wedge T_M ) [D_n]
- \expe{
	(\vr, \vu, \vB)( t \wedge T_M ) [D]  } \right\|_{W^{k,2}(Q; \R^7) }} \br
& \leq c(K,M) N^{-\frac{1}{2}}, \  k \geq 3.
\end{align}
On the other hand, for the discretization error, we have
\begin{align}
&\left\| \frac{1}{N} \sum_{n=1}^N (\vr, \vm = \vr \vu, \vB)_h ( t \wedge T_M ) [D_n]
- \frac{1}{N} \sum_{n=1}^N (\vr, \vm = \vr \vu, \vB) ( t \wedge T_M ) [D_n]	
\right\|_{L^\gamma \times L^{\frac{2 \gamma}{\gamma + 1}} \times L^2 (Q; \mathbb \R^7)} \br
&\leq \frac{1}{N} \sum_{n=1}^N \left\| (\vr, \vm = \vr \vu, \vB)_h ( t \wedge T_M ) [D_n] - (\vr, \vm = \vr \vu, \vB) ( t \wedge T_M ) [D_n]	
\right\|_{L^\gamma \times L^{\frac{2 \gamma}{\gamma + 1}} \times L^2 (Q; \mathbb \R^7)}.
\nonumber	
	\end{align}	
Passing to expectation and using the fact that $D_n$, $n=1,\cdots$, are identically distributed and using the error estimates \eqref{error}, we obtain
\begin{align}
	&\expe{ \left\| \frac{1}{N} \sum_{n=1}^N (\vr, \vm = \vr \vu, \vB)_h ( t \wedge T_M ) [D_n]
	- \frac{1}{N} \sum_{n=1}^N (\vr, \vm = \vr \vu, \vB) ( t \wedge T_M ) [D_n]	
	\right\|_{L^\gamma \times L^{\frac{2 \gamma}{\gamma + 1}} \times L^2 (Q; \mathbb \R^7)} } \br
\leq &\expe{ 	 \left\| (\vr, \vm = \vr \vu, \vB)_h ( t \wedge T_M ) [D] - (\vr, \vm = \vr \vu, \vB) ( t \wedge T_M ) [D]	
	\right\|_{L^\gamma \times L^{\frac{2 \gamma}{\gamma + 1}} \times L^2 (Q; \mathbb \R^7)} }
	\nonumber
	\\ \leq & c(K,M) (\TS^{1/r}+ h^{\alpha /r}),\  r= \max\{2, \gamma \}, \ 0 \leq t \leq T, \ k \geq 6.
	\label{ess2}
	\end{align}

Summing up \eqref{ess1} and \eqref{ess2} we obtain the following result.

\begin{Theorem}[{\bf Error estimates of the Monte Carlo FV scheme}] \label{thm1}
Let $D : \Omega \to X_D$ be random data satisfying	
\begin{align}
	& \norm{(\vr_0,\vu_0,\vB_0)}_{W^{k,2}(Q; \mathbb{R}^7 )} + \|\vr_0^{-1} \|_{L^\infty(Q)} \leq K, \   \ |\vc{g}| \leq K,\  \abs{\vc{b}^{\pm}} \leq K, \br
	& \frac{1}{K} \leq \mu \leq K,\
0 \leq {\lambda + \frac{2}{3} \mu } \leq K,\ \frac{1}{K} \leq \zeta \leq K, \quad \prst\mbox{-a.s.},
\label{ass1}
\end{align}
where $k \geq 6$, and $K>0$ is a deterministic constant. Let $(D_n)_{n=1}^\infty$ be a sequence of i.i.d.\ copies of $D$. Let $(\vr, \vm = \vr \vu, \vB)[D]$ be the strong solution defined on $[0, T_{\rm max}[D])$, and let $(\vr, \vm = \vr \vu,
\vc{B})_h[D]$ be obtained by the FV scheme \eqref{scheme} with the discretization steps $(\Delta t,h) \in (0,1)^2$. Let $T_M = T_M[D]$ be the stopping time of the strong solution at the level $M > 0$ introduced in \eqref{S14}.

Then for any $0 \leq t \leq T$ it holds
\begin{align}
	&\expe{ \left\| \frac{1}{N} \sum_{n=1}^N (\vr, \vm, \vB)_h ( t \wedge T_M ) [D_n]
	- \expe{(\vr, \vm, \vc{B}) ( t \wedge T_M ) [D]}	
	\right\|_{L^\gamma \times L^{\frac{2 \gamma}{\gamma + 1}} \times L^2 (Q; \mathbb \R^7)} }\br &\leq c(K,M,T) \left( N^{-\frac{1}{2}} + \TS^{1/r}+ h^{\alpha /r}  \right),
	\nonumber
	\end{align}
where  $r= \max\{2, \gamma \}$  and  $\alpha$ is given by \eqref{alpha}.
	\end{Theorem}

\subsection{Convergence on the life-span of strong solution}

Similarly to the preceding part, we write
\begin{align}
	\frac{1}{N} &\sum_{n=1}^N (\vr, \vm = \vr \vu, \vB)_h \mathds{1}_{[0, T_{\rm max})} [D_n] - \expe{
		(\vr, \vm = \vr \vu, \vB)\mathds{1}_{[0, T_{\rm max})} [D]} \br &=
	\frac{1}{N} \sum_{n=1}^N (\vr, \vm, \vB)_h \mathds{1}_{[0, T_{\rm max})} [D_n]
	- \frac{1}{N} \sum_{n=1}^N (\vr, \vm, \vB) \mathds{1}_{[0, T_{\rm max})} [D_n]
	\br&+ \frac{1}{N} \sum_{n=1}^N (\vr, \vm, \vB) \mathds{1}_{[0, T_{\rm max})} [D_n] - \expe{
		(\vr, \vm, \vB)\mathds{1}_{[0, T_{\rm max})} [D] }.
	\nonumber	
\end{align}
Applying SLLN we get
\begin{align}
\left\| \frac{1}{N} \sum_{n=1}^N (\vr, \vm, \vB) \mathds{1}_{[0, T_{\rm max})} [D_n] - \expe{
	(\vr, \vm, \vB)\mathds{1}_{[0, T_{\rm max})} [D] } \right\|_{L^\gamma \times L^{\frac{2 \gamma}{\gamma + 1}} \times L^2 (Q; \mathbb \R^7)} \to 0
	\ \prst - \mbox{a.s. as}\ N \to \infty.
	\nonumber
	\end{align}
However, as the data and consequently the solutions are deterministically bounded in the energy space $L^\gamma \times L^{\frac{2 \gamma}{\gamma + 1}} \times L^2 (Q;  \R^7)$, we may pass to expectation obtaining
\begin{equation} \label{sss1}
\expe{\left\| \frac{1}{N} \sum_{n=1}^N (\vr, \vm, \vB) \mathds{1}_{[0, T_{\rm max})} [D_n] - \expe{
	(\vr, \vm, \vB)\mathds{1}_{[0, T_{\rm max})} [D] } \right\|_{L^\gamma \times L^{\frac{2 \gamma}{\gamma + 1}} \times L^2 (Q; \mathbb \R^7)} } \to 0
\end{equation}
as $N \to \infty$.

Finally, the approximation error can be estimated analogously as in \eqref{ess2}:
\begin{align}
	&\expe{ \left\| \frac{1}{N} \sum_{n=1}^N (\vr, \vm , \vB)_h \mathds{1}_{[0, T_{\rm max})} [D_n]
		- \frac{1}{N} \sum_{n=1}^N (\vr, \vm , \vB) \mathds{1}_{[0, T_{\rm max})} [D_n]
		\right\|_{L^\gamma \times L^{\frac{2 \gamma}{\gamma + 1}} \times L^2 (Q; \mathbb \R^7)} }
	\to 0 \br \mbox{as}\ h &\to 0, \ \TS \to 0\  \mbox{for any}\ t \in [0,T].
	\label{sss3}
\end{align}

We have proved the following result.

\begin{Theorem}[{\bf Convergence of the Monte Carlo FV scheme}] \label{thm2}
	Let $D : \Omega \to X_D$ be random data satisfying	
	\begin{align}
		& \norm{(\vr_0,\vu_0,\vB_0)}_{W^{k,2}(Q; \mathbb{R}^7 )} + \|\vr_0^{-1} \|_{L^\infty(Q)} \leq K, \  \ |\vc{g}| \leq K, \ \abs{\vc{b}^{\pm}} \leq K, \br
		& \frac{1}{K} \leq \mu \leq K,\
0 \leq {\lambda + \frac{2}{3} \mu } \leq K,\ \frac{1}{K} \leq \zeta \leq K,\quad \prst\mbox{-a.s.},
		\nonumber	
	\end{align}
	where $k \geq 6$, and $K>0$ is a deterministic constant. Let $(D_n)_{n=1}^\infty$ be a sequence of i.i.d copies of $D$. Let $(\vr, \vm = \vr \vu, \vB)[D]$ be the strong solution defined on $[0, T_{\rm max}[D])$, and let $(\vr, \vm = \vr \vu,
	\vc{B})_h[D]$ be its approximation by the FV scheme \eqref{scheme} with the discretization steps $(\Delta t, h) \in (0,1)^2$.
	
	Then
	\begin{align}
		&\expe{ \left\| \frac{1}{N} \sum_{n=1}^N (\vr, \vm, \vB)_h \mathds{1}_{[0, T_{\rm max})} [D_n]
			- \expe{(\vr, \vm, \vc{B}) \mathds{1}_{[0, T_{\rm max})} [D]}	
			\right\|_{L^\gamma \times L^{\frac{2 \gamma}{\gamma + 1}} \times L^2 (Q; \mathbb \R^7)} } \to 0 \br &\mbox{as}\ N \to \infty,\ h \to 0, \ \Delta t \to 0,
		\label{conv}
	\end{align}
	for any $0 \leq t \leq T$.	
\end{Theorem}

\begin{Remark} \label{Rass1}
By virtue of \eqref{ass1}, the energy inequality \eqref{S20} together with its discrete
counterpart stated in Lemma \ref{EnBa}, yield boundedness of both the approximate and the exact solution in the energy space $L^\gamma \times L^{\frac{2 \gamma}{\gamma + 1}} \times L^2 (Q; \mathbb \R^7)$ by a deterministic constant depending only on $K$ and $T$. Consequently, the convergence stated in \eqref{conv} can be strengthened to
		\begin{align}
		&\expe{ \left\| \frac{1}{N} \sum_{n=1}^N (\vr, \vm, \vB)_h \mathds{1}_{[0, T_{\rm max})} [D_n]
			- \expe{(\vr, \vm, \vc{B}) \mathds{1}_{[0, T_{\rm max})} [D]}	
			\right\|^m_{L^\gamma \times L^{\frac{2 \gamma}{\gamma + 1}} \times L^2 (Q; \mathbb \R^7)} } \to 0 \br &\mbox{as}\ N \to \infty,\ h \to 0, \ \TS \to 0
		\nonumber
	\end{align}
	for any $0 \leq t \leq T$, and any $m \geq 1$.
	\end{Remark}

\section{Numerical simulations}
In this section we present three two-dimensional experiments to illustrate theoretical error estimates results, cf.\ Theorem \ref{thm1}. We will calculate  the error of mean values, denoted by $E_1$
\begin{align}
&E_1(U,h,N) =  \frac1{L} \sum_{\ell=1}^L \left(  \left\| \frac{1}{N} \sum_{n = 1}^N U_h^{\ell,n} (T, \cdot)- \expe{ U (T, \cdot) }  \right\|_{L^p(Q)}  \right),
\end{align}
as well as the error of deviations,  denoted by $E_2$
\begin{align}
E_2(U,h,N) =
\frac1{L} \sum\limits_{\ell=1}^L \left(   \left\| \frac 1 N \sum\limits_{n=1}^N \Big| U_h^{\ell,n}(T, \cdot)   - \frac{1}{N}  \sum\limits_{j=1}^N U_h^{\ell,j}(T, \cdot)   \Big| \ -  \Dev{U(T, \cdot) }  \right\|_{L^p(Q)}  \right).
\end{align}
Here $U_h^{\ell,n} = U[D_{\ell,n}], \, U \in \{ \vr_h, \, \vm_h,\, \vc{B}_h,\,  \vu_h,\, \Gradd \vu_h,\, \Curlh \vc{B}_h  \}$ is the Monte Carlo FV approximation associated to the discrete sample $D_{\ell,n}$, and then $p\in \{\gamma,\, 2\gamma/(\gamma+1), \, 2, \,  2,\,  2, \,  2\}$ shall be chosen according to integrability of the corresponding component.
The exact expectation $\expe{ U (T, \cdot) } $ and the deviation $\Dev{U(T, \cdot) }$ are approximated by numerical solutions on the finest grid $h_{ref}$ emanating from $M_{ref}$ samples of data.
In the simulations we set $L$ to $40$, and  $M_{ref}$ to $500$.

\subsection{Sine wave problem}
Consider $Q =[-1,1]|_{\{-1,1\}} \times [-1, 1]$.  The initial data are taken as follows
 \begin{align*}
\vr_0(x) = 2 + \cos\big(2\pi (x_1+x_2)\big), \ \
\vu_0(x) =Y_1(\omega) \, \big(0,\, \sin(2\pi x_2) \big),\ \
\vc{B}_0(x) = \left(x_2+ Y_2(\omega) \sin \left( \frac{\pi x_2}{2}\right),\, 0 \right),
\end{align*}
the boundary data are
\[
\vu|_{x_2 = \pm 1} = \vc{0}, \quad B_1|_{x_2=1} = 1+Y_2(\omega), \quad B_1|_{x_2=-1} = -1-Y_2(\omega).
\]
Here $Y_1$ and $Y_2$ are  i.i.d.\ random variables obeying the uniformly distribution $\mathcal{U}([-0.1,0.1]^2)$.
The model parameters in the MHD system \eqref{pde} are taken as $\mu = 0.01, \; \lambda = 0, \;\zeta = 0.01$, $\gamma = 5/3$, $a=1$, and $b=0$.

Figure \ref{sine} shows deterministic numerical solutions $(\vr, m_1, B_1)_{h_{ref}}(\omega=0,T)$ obtained on a uniform squared  mesh with the mesh parameter $h_{ref}=2/320$  at the final time $T = 0.6$. We present also  time evolution of $\| \Divh \vc{B}_h\|_{L^{1}(Q)}(\omega=0, t)$ with $ h = 2/(10 \cdot 2^i), i = 1,\dots,5.$  This numerical result confirms the divergence-free property of our FV method shown in Lemma~\ref{lem:divfree};  $\| \Divh \vc{B}_h\|_{L^{1}(Q)}(\omega=0, t)$ are around  machine error.
Furthermore, Figure~\ref{sine-1} shows the mean and deviations of statistical numerical solutions $(\vr, m_1, B_1)_{h_{ref}}(\omega,T)$ with $M_{ref} = 500$ samples.

The statistical errors of numerical solutions obtained by the Monte Carlo FV scheme on a grid with  $h= 2/320$ and $N$ samples, $N=10 \cdot 2^n, n = 1,2,3,4$ are shown in  Figure \ref{sine-err-1}.
The total errors of statistical numerical solutions obtained with $(h,N) = (2/(10 \cdot 2^n), 10 \cdot 2^n), n = 1,2,3,4,$ are shown in
Figure~\ref{sine-err-2}.
The numerical results show half-order convergence rate with respect to $N^{-1}$, and a convergence rate between $1/2$ and $1$ with respect to $(h,N) = (2/(10 \cdot 2^n), 10 \cdot 2^n).$

\begin{figure}[htbp]
	\setlength{\abovecaptionskip}{0.cm}
	\setlength{\belowcaptionskip}{0.2cm}
	\centering
	\begin{subfigure}{0.42\textwidth}
	\includegraphics[width=\textwidth]{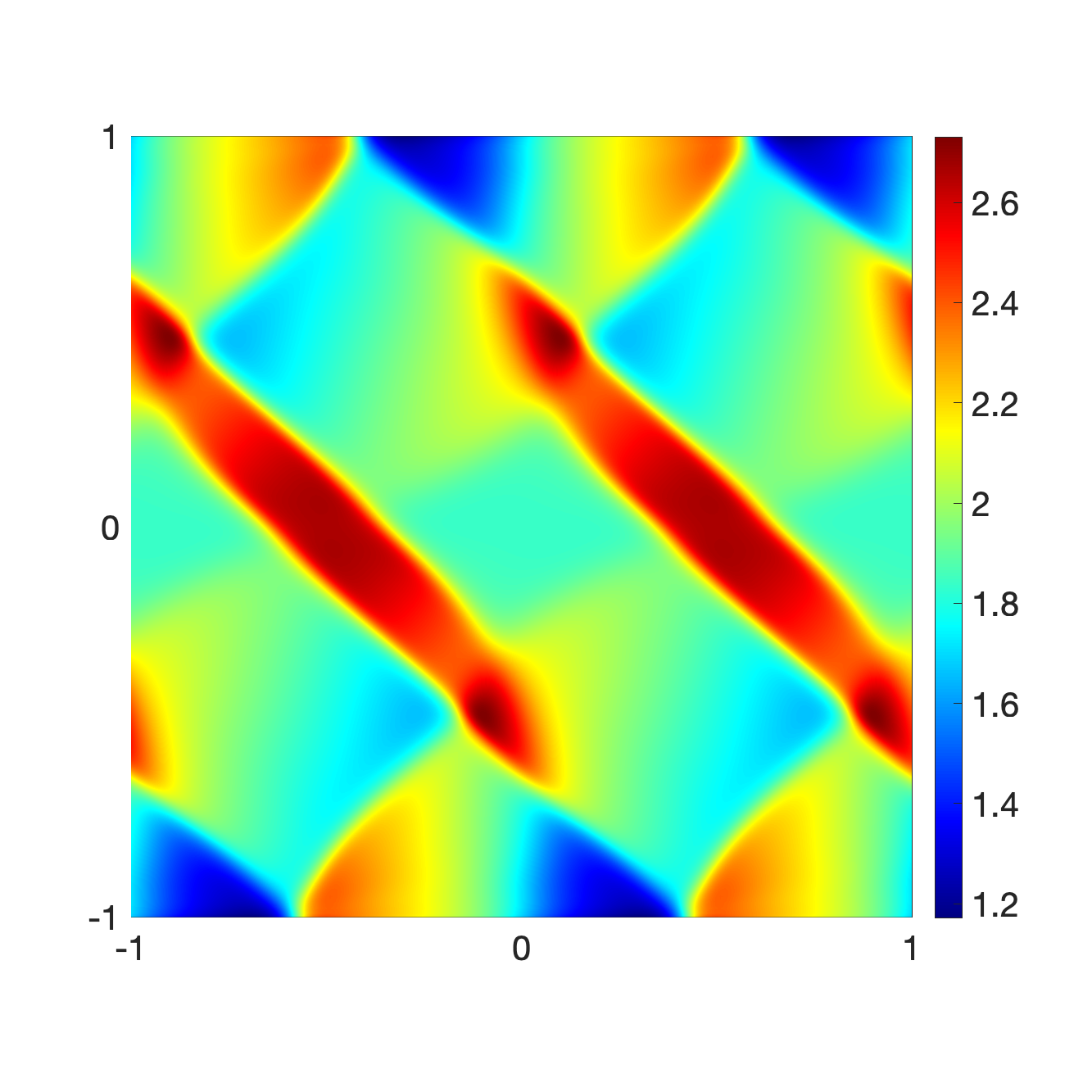}
\caption{ $\vr_{h_{ref}}(\omega=0,T)$}
	\end{subfigure}	
	\begin{subfigure}{0.42\textwidth}
	\includegraphics[width=\textwidth]{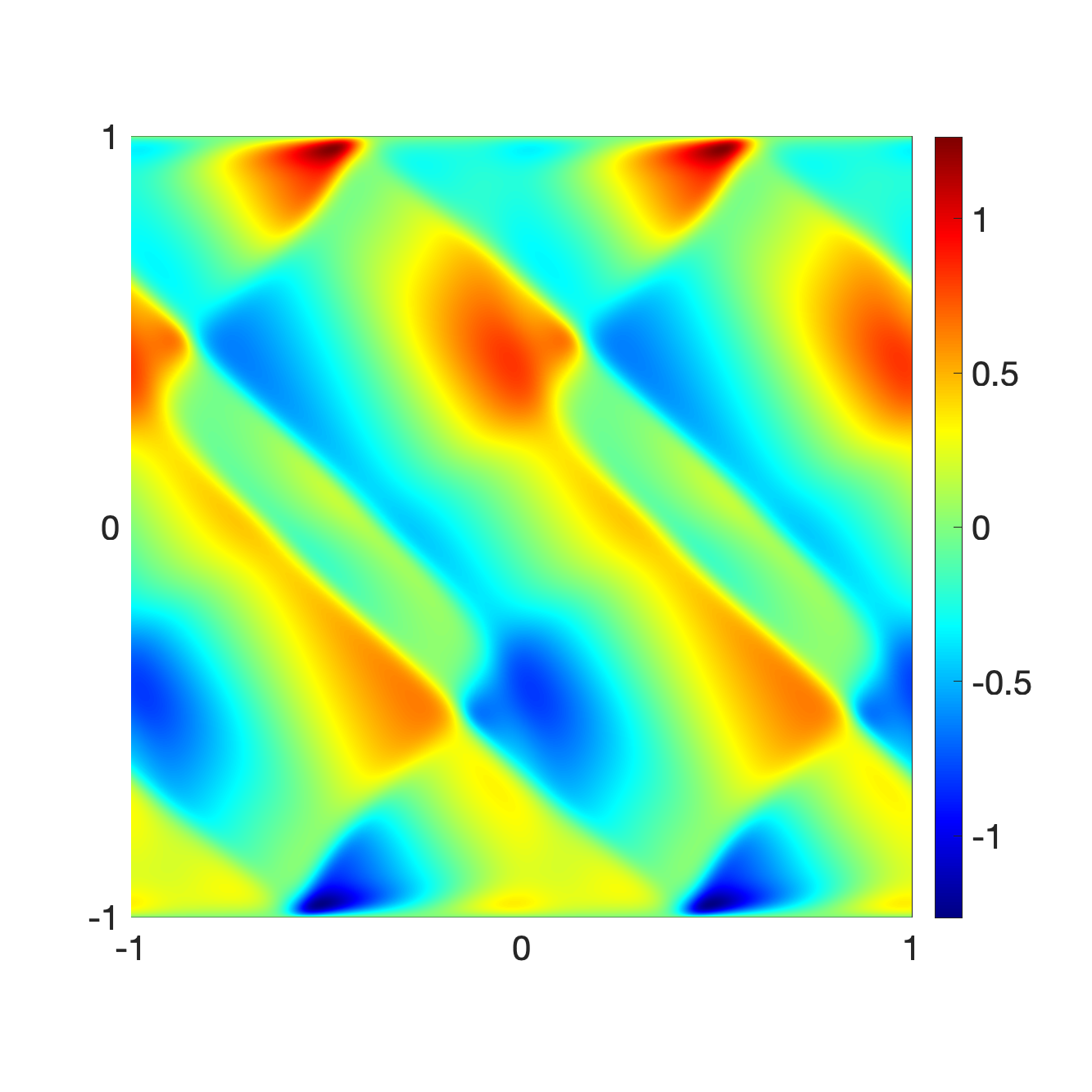}
\caption{ $(m_1)_{h_{ref}}(\omega=0,T)$}
	\end{subfigure}	\\
	\begin{subfigure}{0.42\textwidth}
	\includegraphics[width=\textwidth]{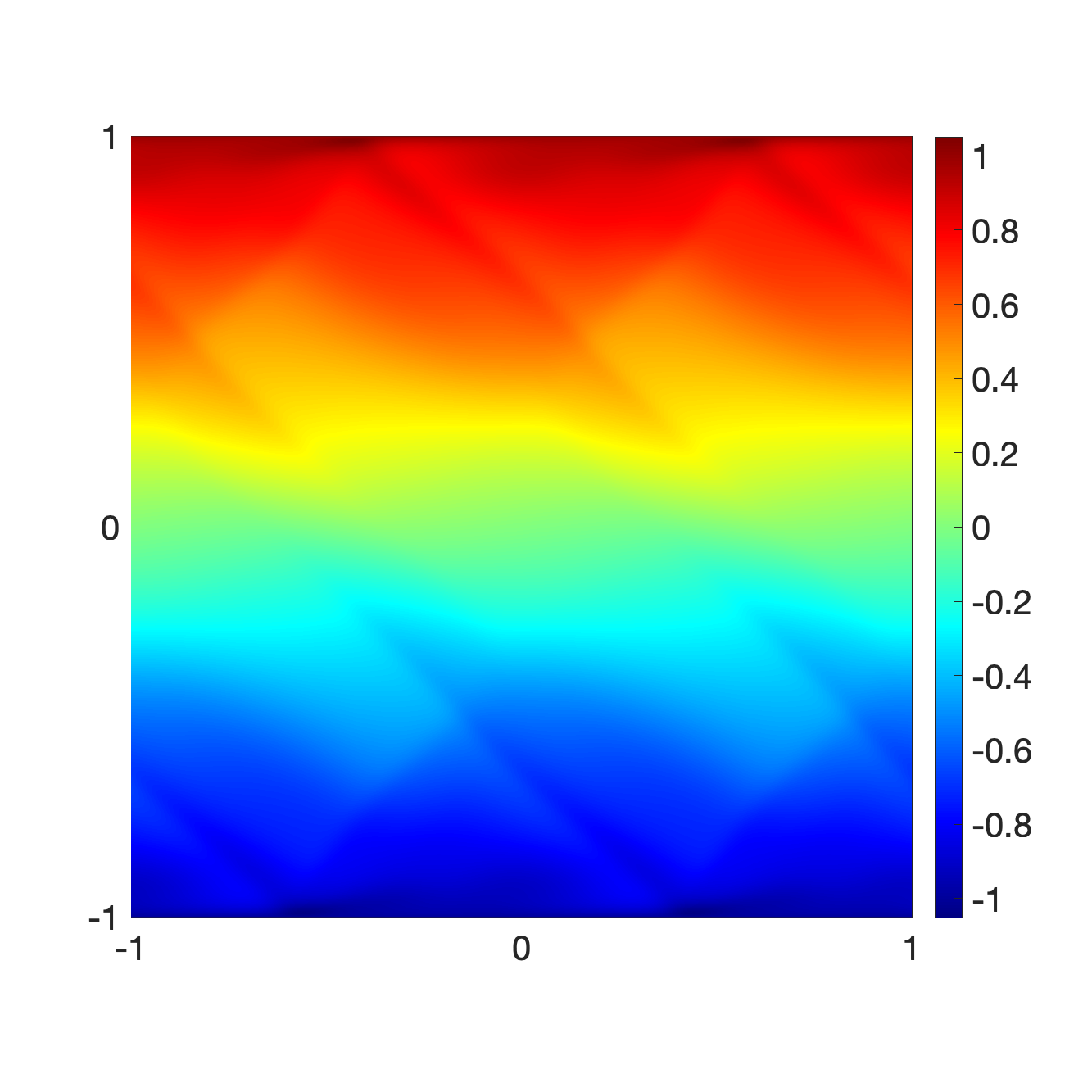}
\caption{ $(B_1)_{h_{ref}}(\omega=0,T)$}
	\end{subfigure}	
	\begin{subfigure}{0.42\textwidth}
	\includegraphics[width=\textwidth]{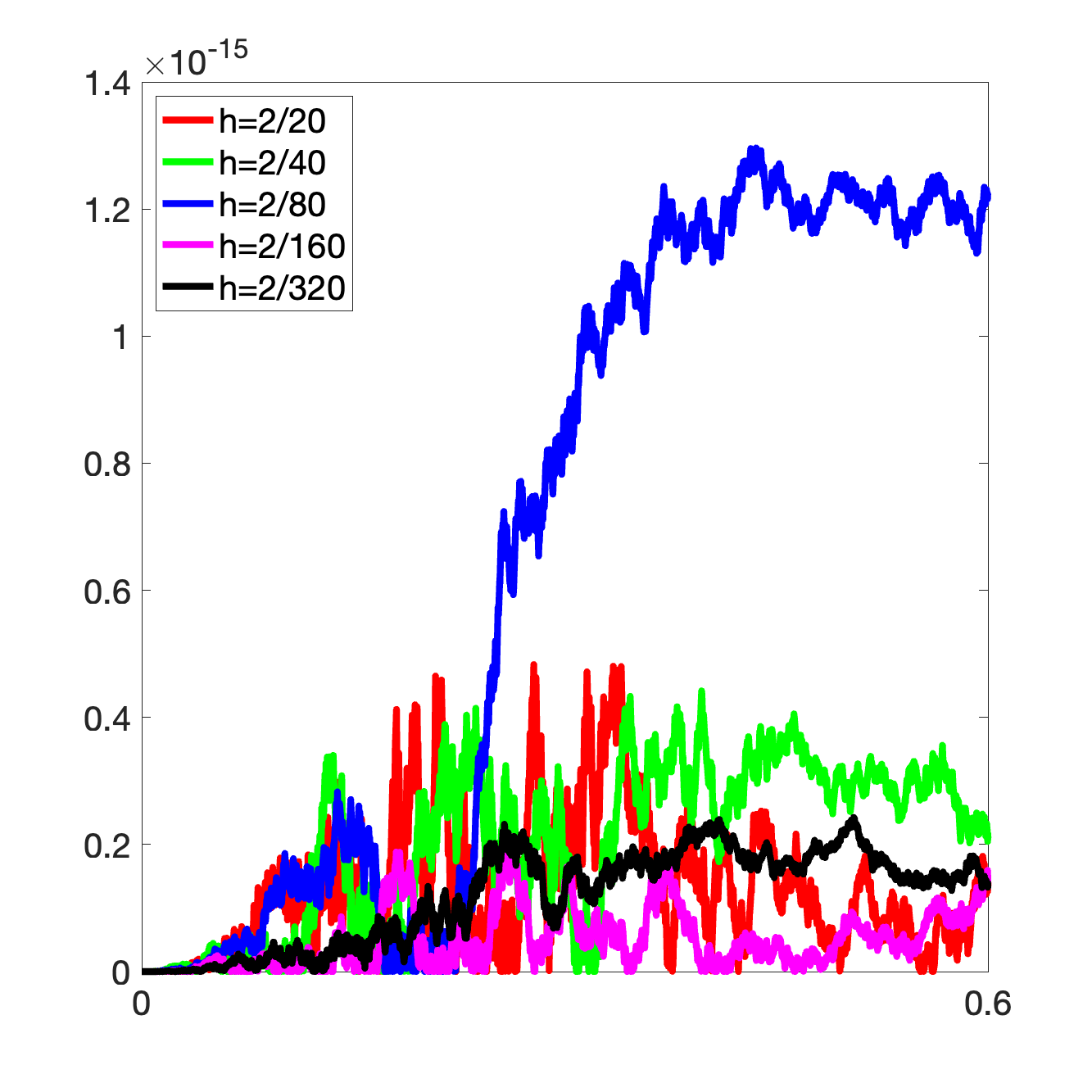}
	\caption{ $\| \Divh \vc{B}_h\|_{L^{1}(Q)}(\omega=0, t)$ }
	\end{subfigure}	
	\caption{{\bf Sine wave problem}. Deterministic FV  solutions $(\vr, m_1, B_1)_{h_{ref}}(\omega=0, T)$ with $h_{ref} = 2/320$ at $T=0.6$ and  the time evolution of $\| \Divh \vc{B}_h\|_{L^{1}(Q)}(\omega=0, t)$ with $h = 2/(10 \cdot 2^i), i = 1,\dots,5$.}\label{sine}
\end{figure}

\begin{figure}[htbp]
	\setlength{\abovecaptionskip}{0.cm}
	\setlength{\belowcaptionskip}{-0.cm}
	\centering
	\begin{subfigure}{0.32\textwidth}
	\includegraphics[width=\textwidth]{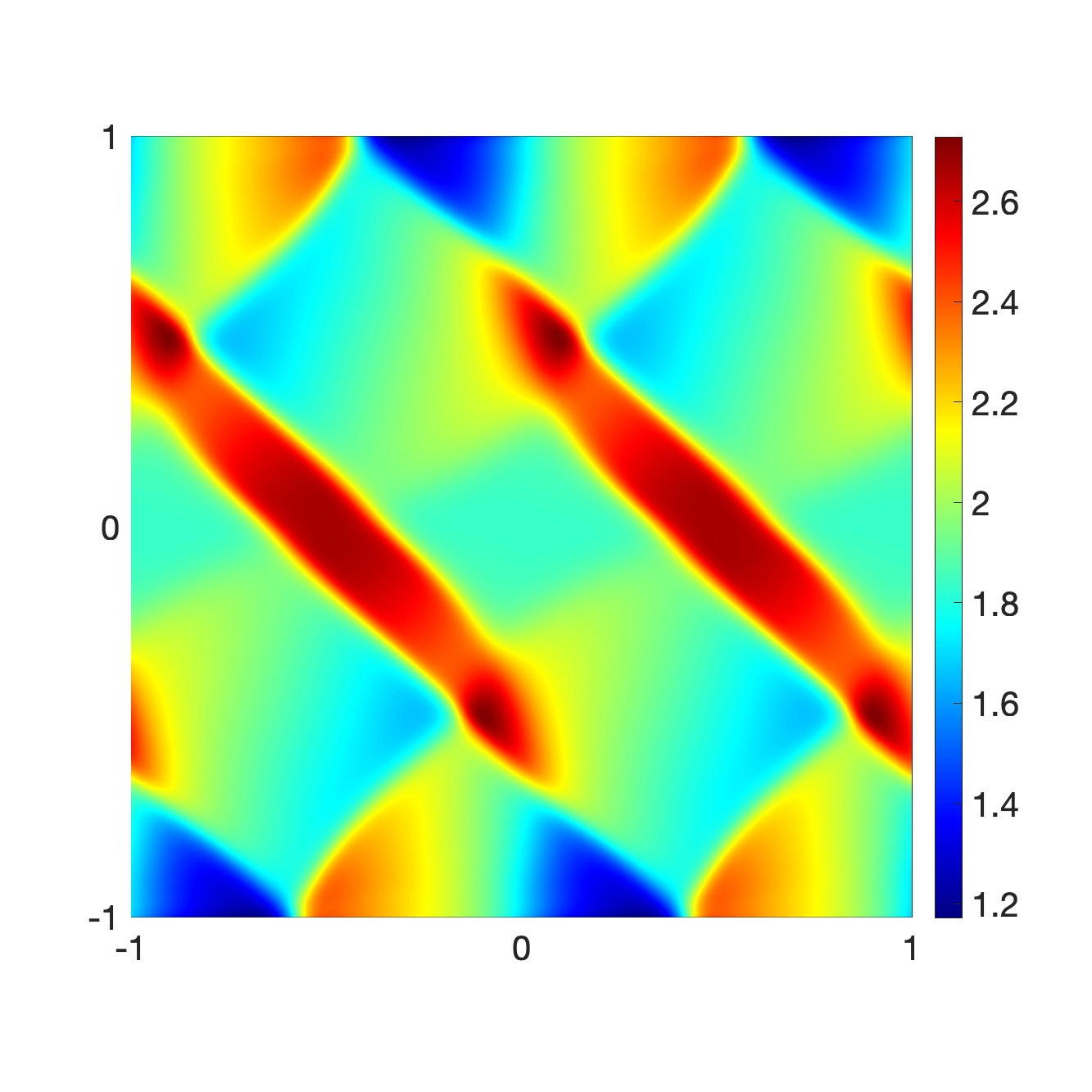}
	\end{subfigure}	
	\begin{subfigure}{0.32\textwidth}
	\includegraphics[width=\textwidth]{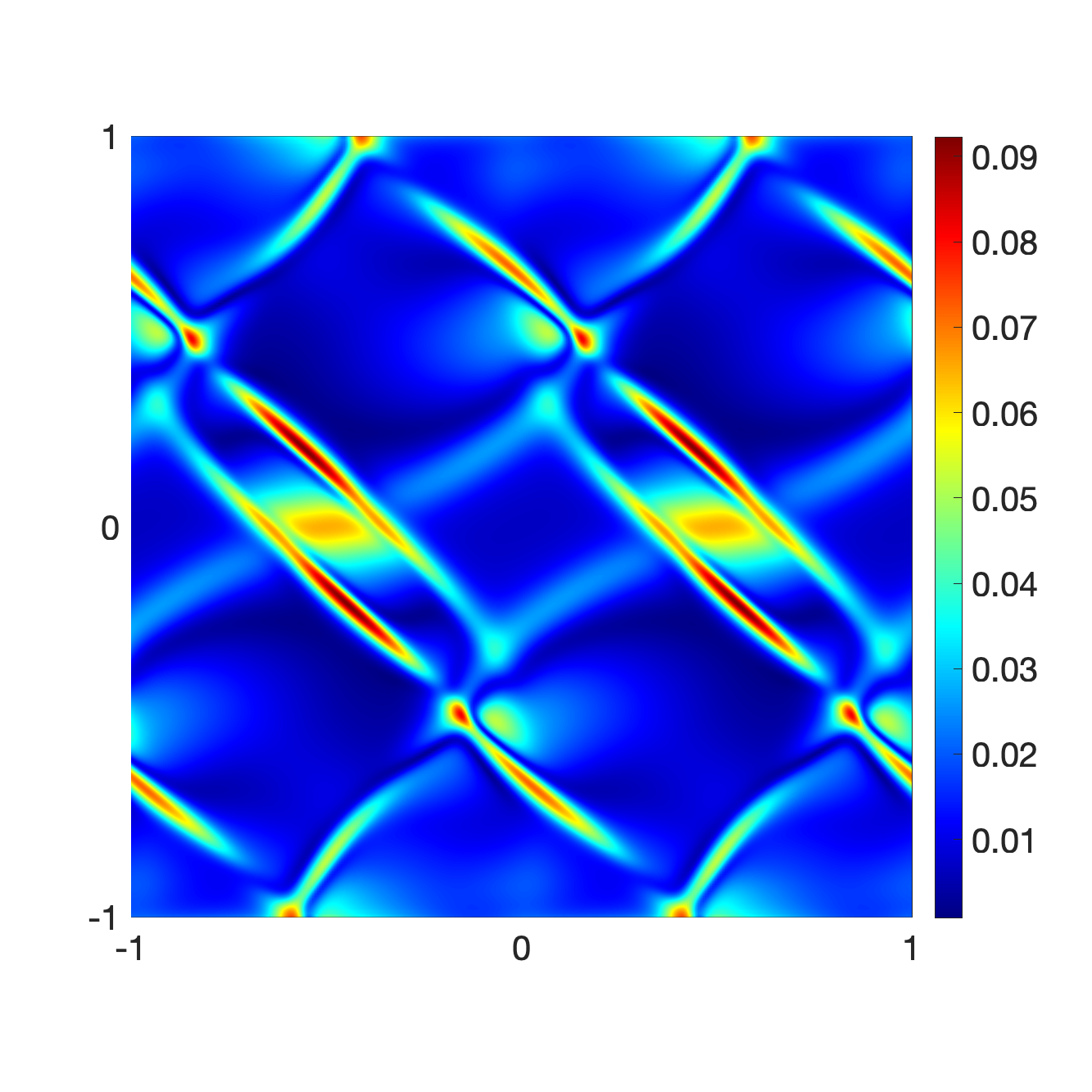}
	\end{subfigure}	
	\begin{subfigure}{0.32\textwidth}
	\includegraphics[width=\textwidth]{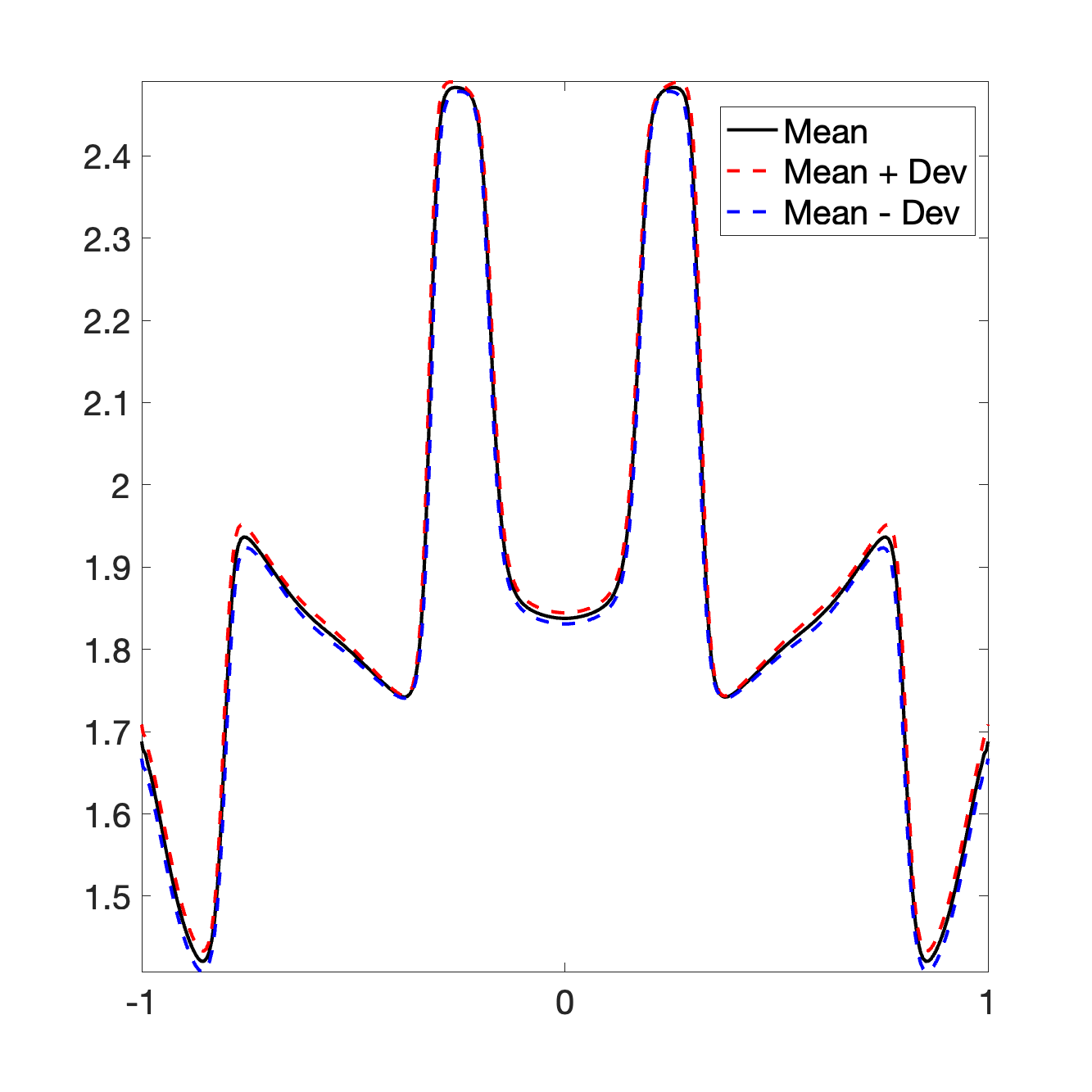}
	\end{subfigure}	\\
	\begin{subfigure}{0.32\textwidth}
	\includegraphics[width=\textwidth]{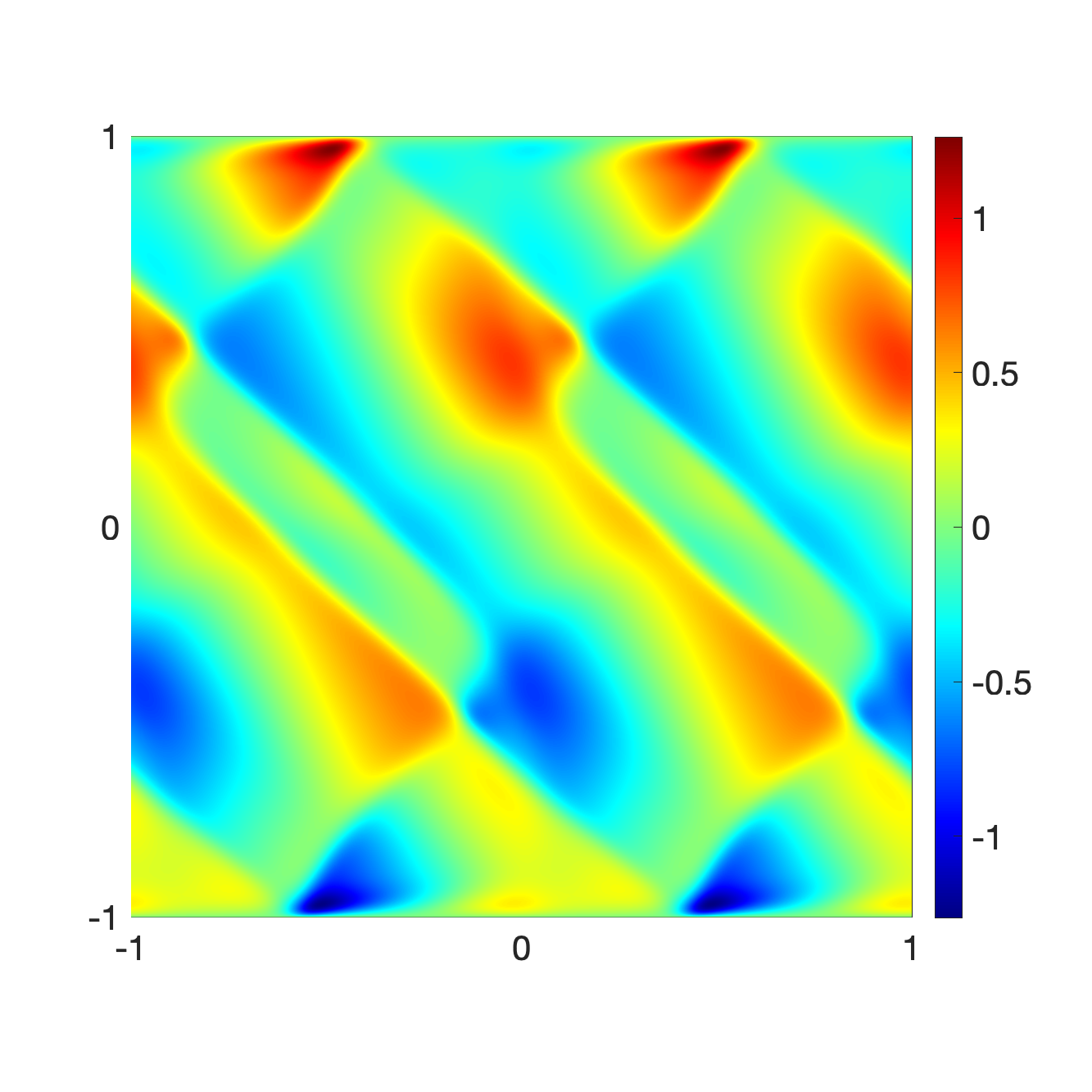}
	\end{subfigure}	
	\begin{subfigure}{0.32\textwidth}
	\includegraphics[width=\textwidth]{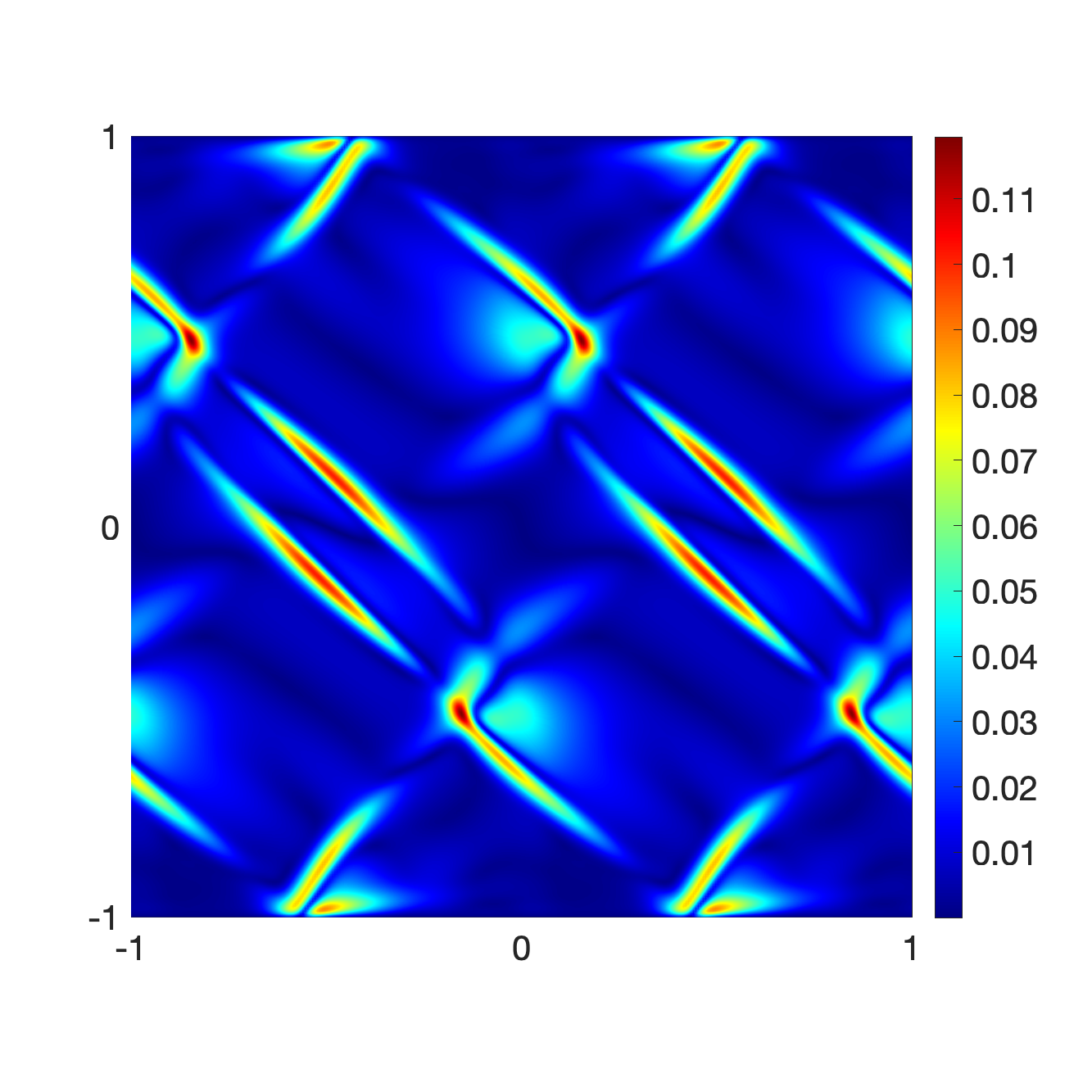}
	\end{subfigure}	
	\begin{subfigure}{0.32\textwidth}
	\includegraphics[width=\textwidth]{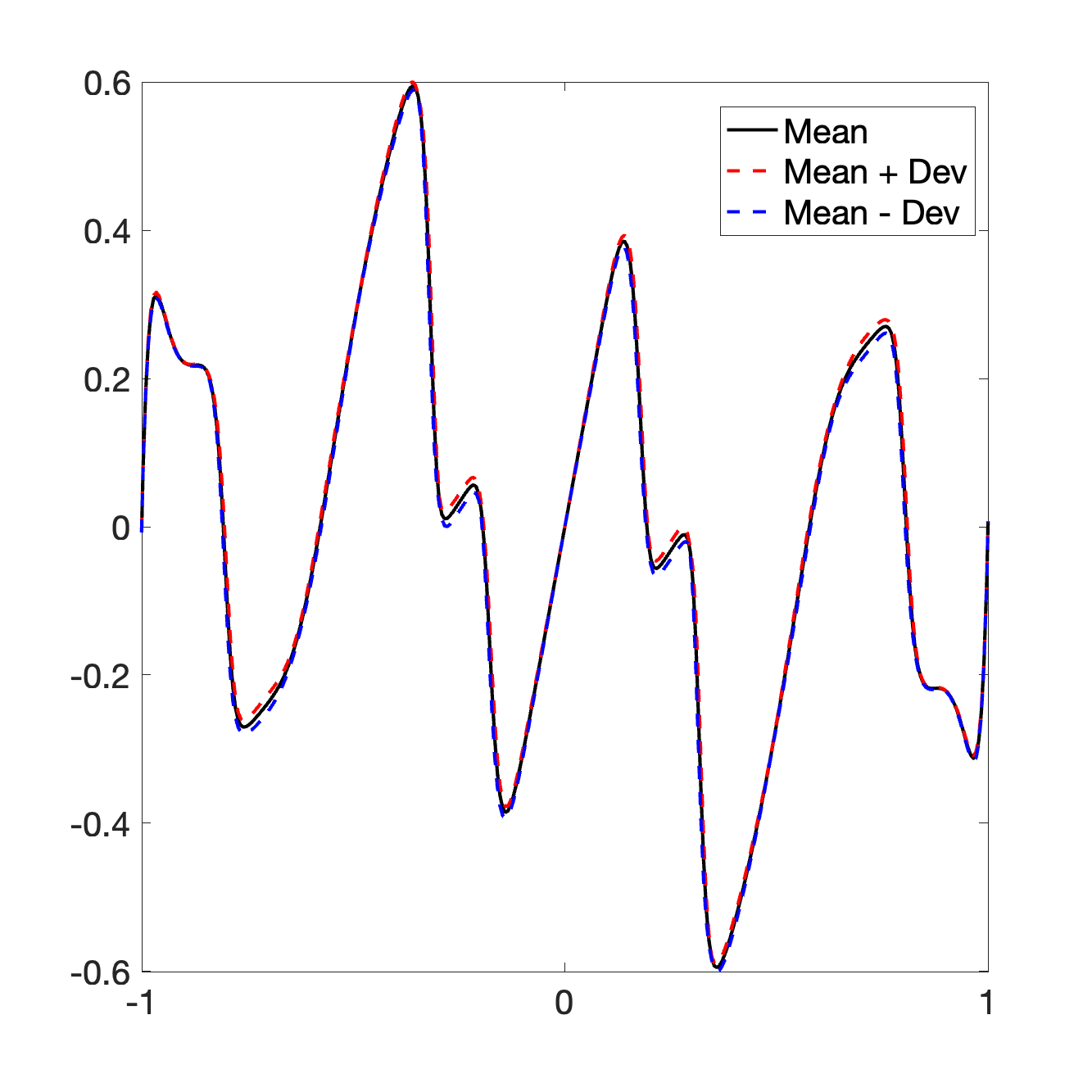}
	\end{subfigure}	\\
	\begin{subfigure}{0.32\textwidth}
	\includegraphics[width=\textwidth]{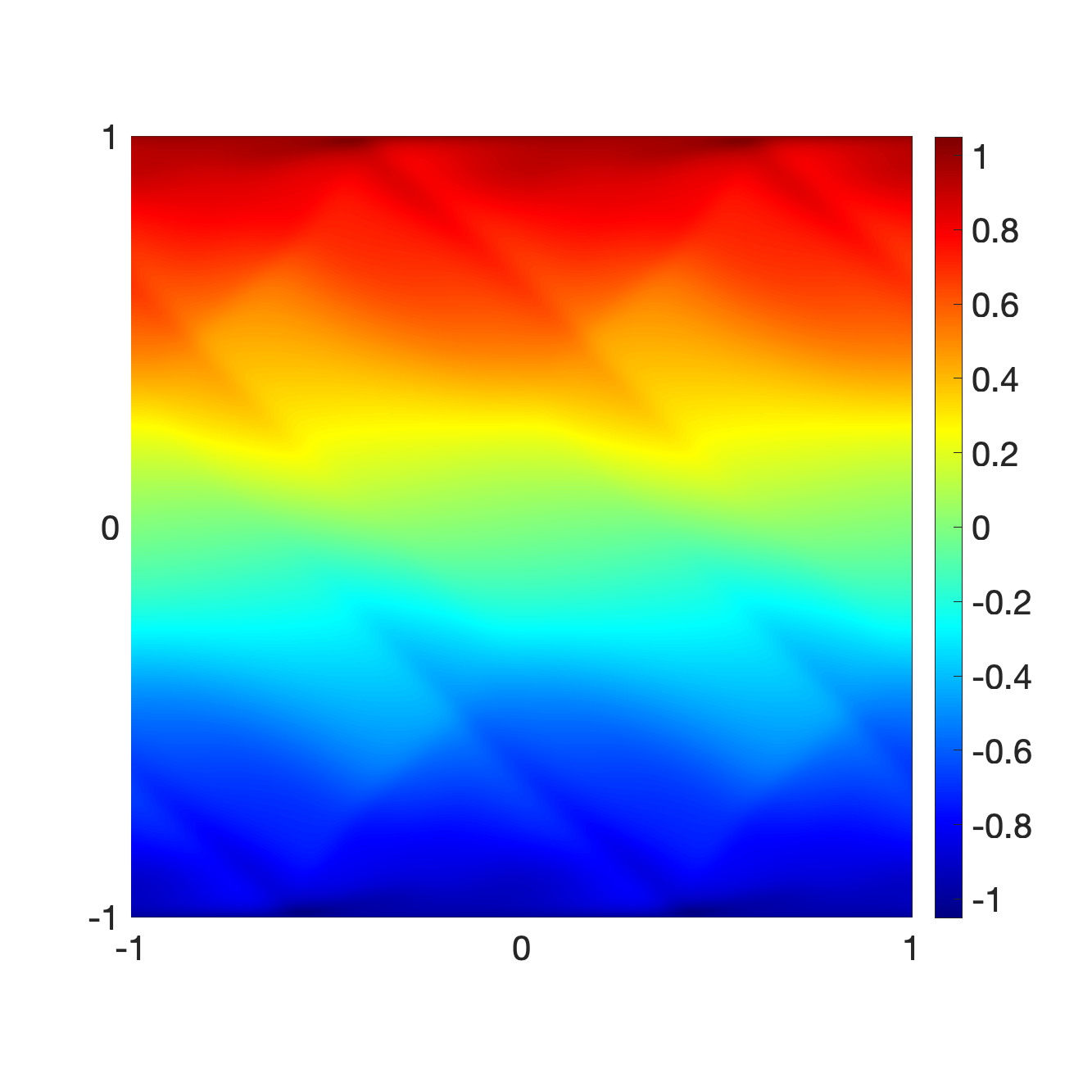}
	\end{subfigure}	
	\begin{subfigure}{0.32\textwidth}
	\includegraphics[width=\textwidth]{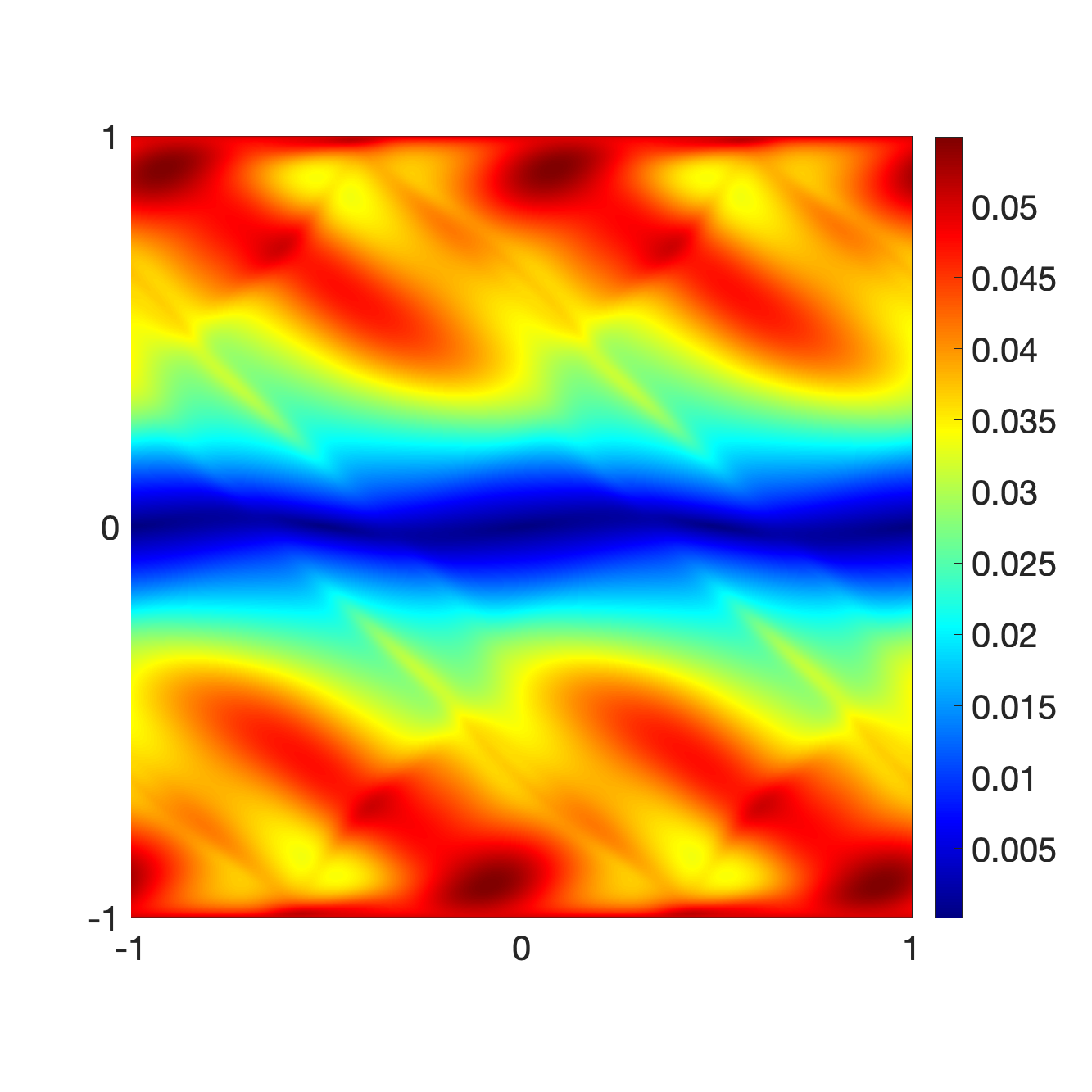}
	\end{subfigure}	
	\begin{subfigure}{0.32\textwidth}
	\includegraphics[width=\textwidth]{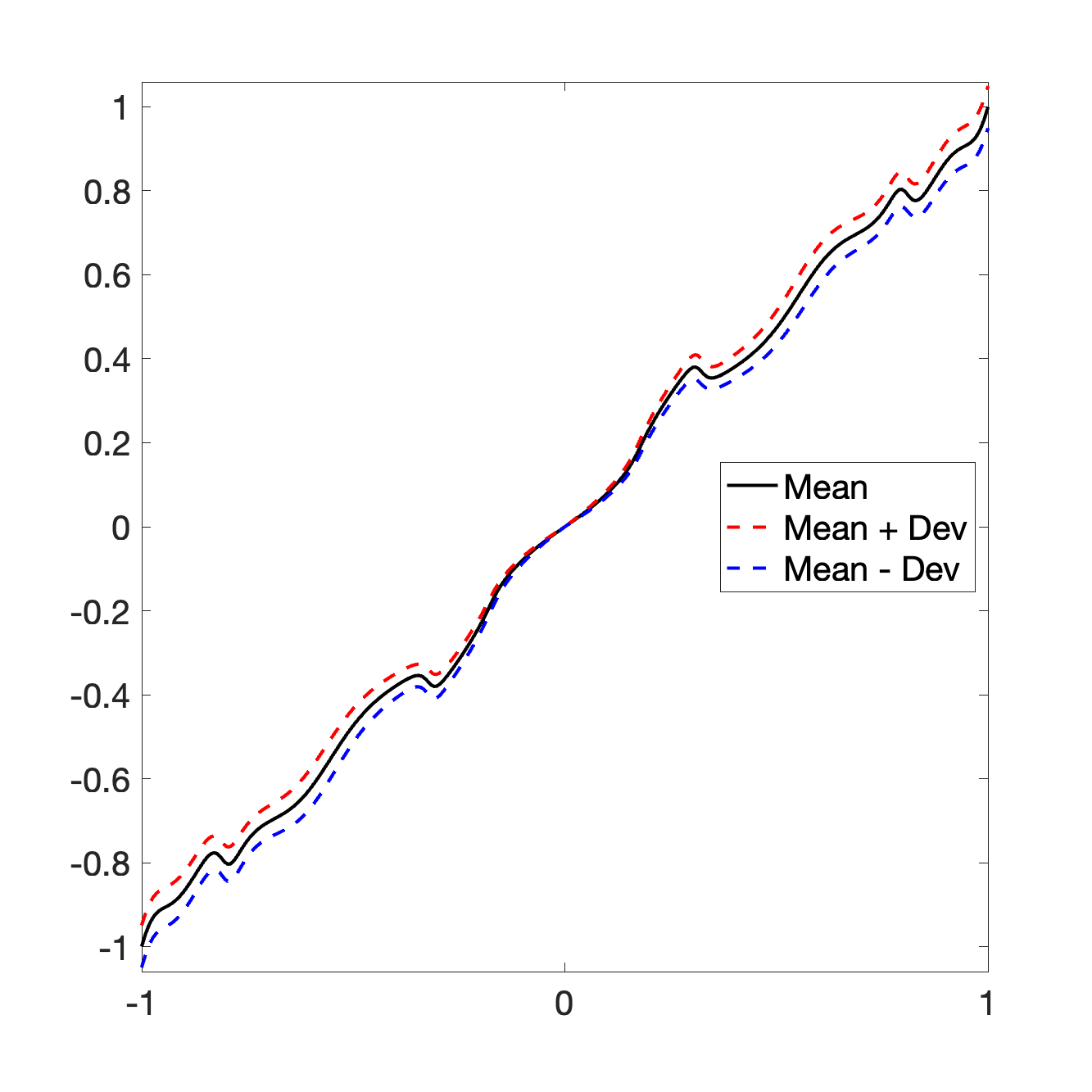}
	\end{subfigure}	
	\caption{ {\bf Random sine wave problem}. FV  solutions $(\vr, m_1, B_1)_{h_{ref}}(\omega,T)$ (from top to bottom) with $h_{ref} = 2/320$ and $500$ samples at $T=0.6$. From left to right: mean, deviation, and mean and deviations along $x=y$.}\label{sine-1}
\end{figure}

\begin{figure}[htbp]
	\setlength{\abovecaptionskip}{0.cm}
	\setlength{\belowcaptionskip}{-0.cm}
	\centering
	\begin{subfigure}{0.45\textwidth}
	\includegraphics[width=\textwidth]{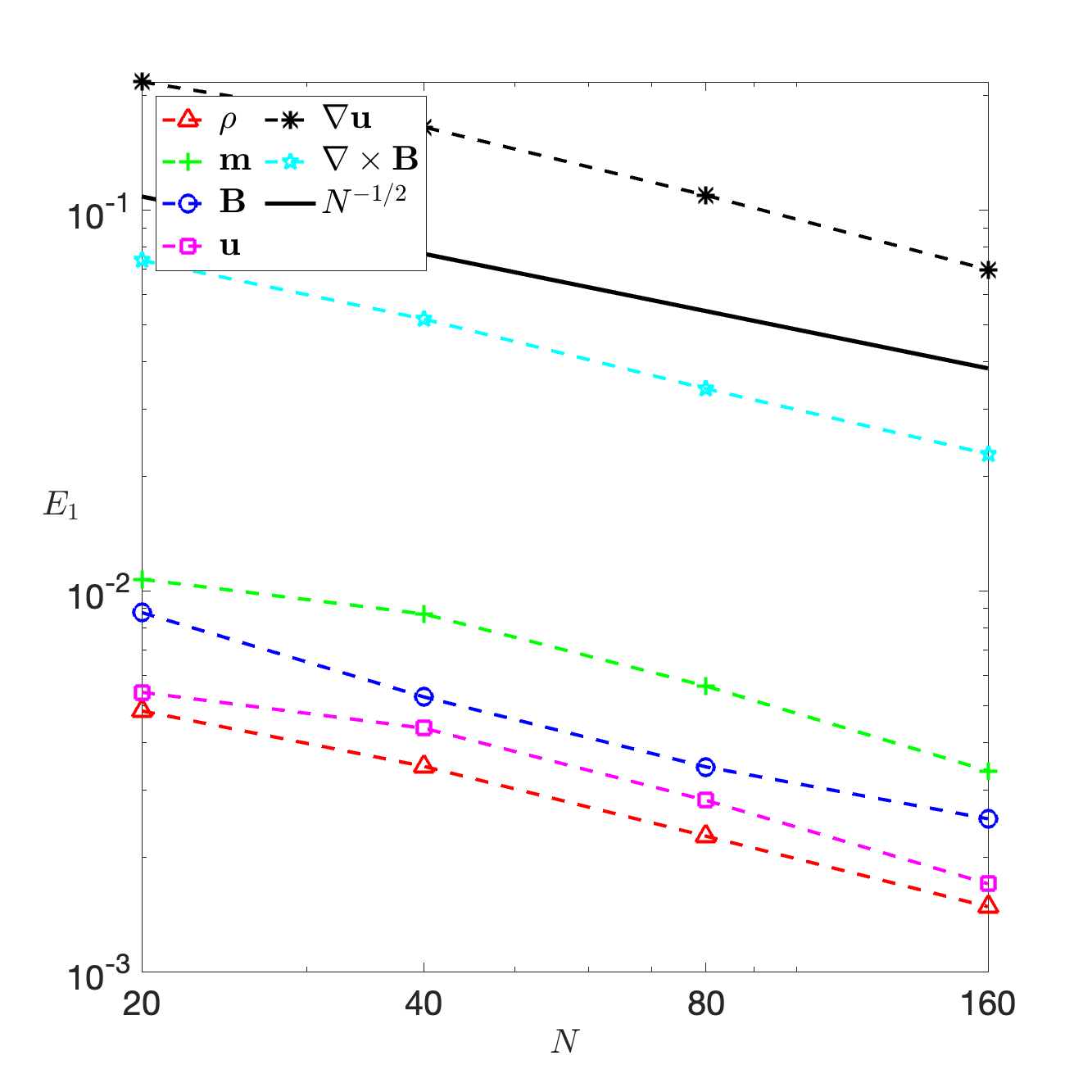}
	\end{subfigure}	
	\begin{subfigure}{0.45\textwidth}
	\includegraphics[width=\textwidth]{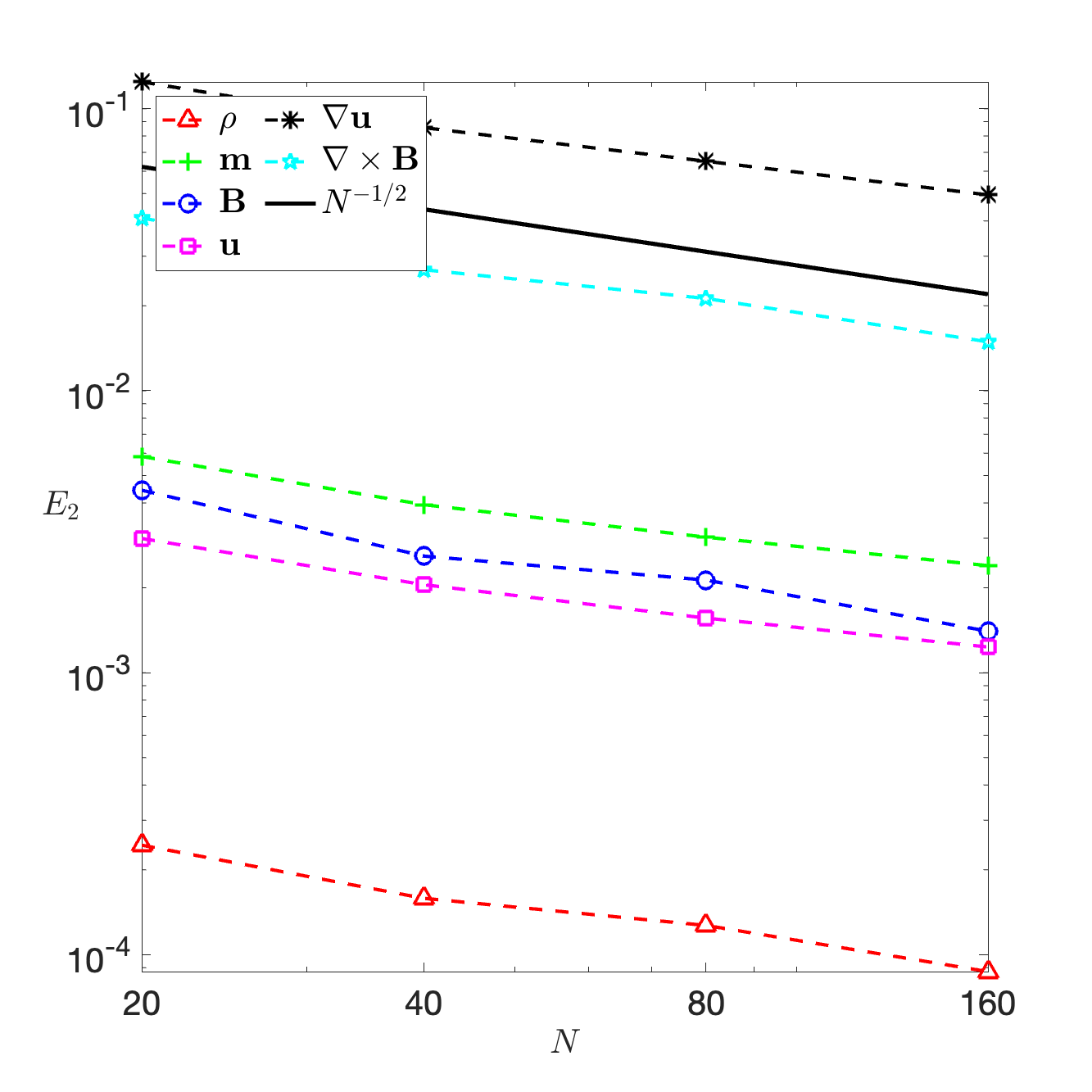}
	\end{subfigure}	
	\caption{ {\bf Random sine wave problem}. {\bf Statistical errors}: $E_1$ (left), $E_2$ (right). In the legend notations $\nabla \vc{u},$ $\nabla \times \vc{B}$ are used for the discrete operators $\Gradd \vu_h,\, \Curlh \vc{B}_h,$ respectively.}\label{sine-err-1}
\end{figure}

\begin{figure}[htbp]
	\setlength{\abovecaptionskip}{0.cm}
	\setlength{\belowcaptionskip}{-0.cm}
	\centering
	\begin{subfigure}{0.45\textwidth}
	\includegraphics[width=\textwidth]{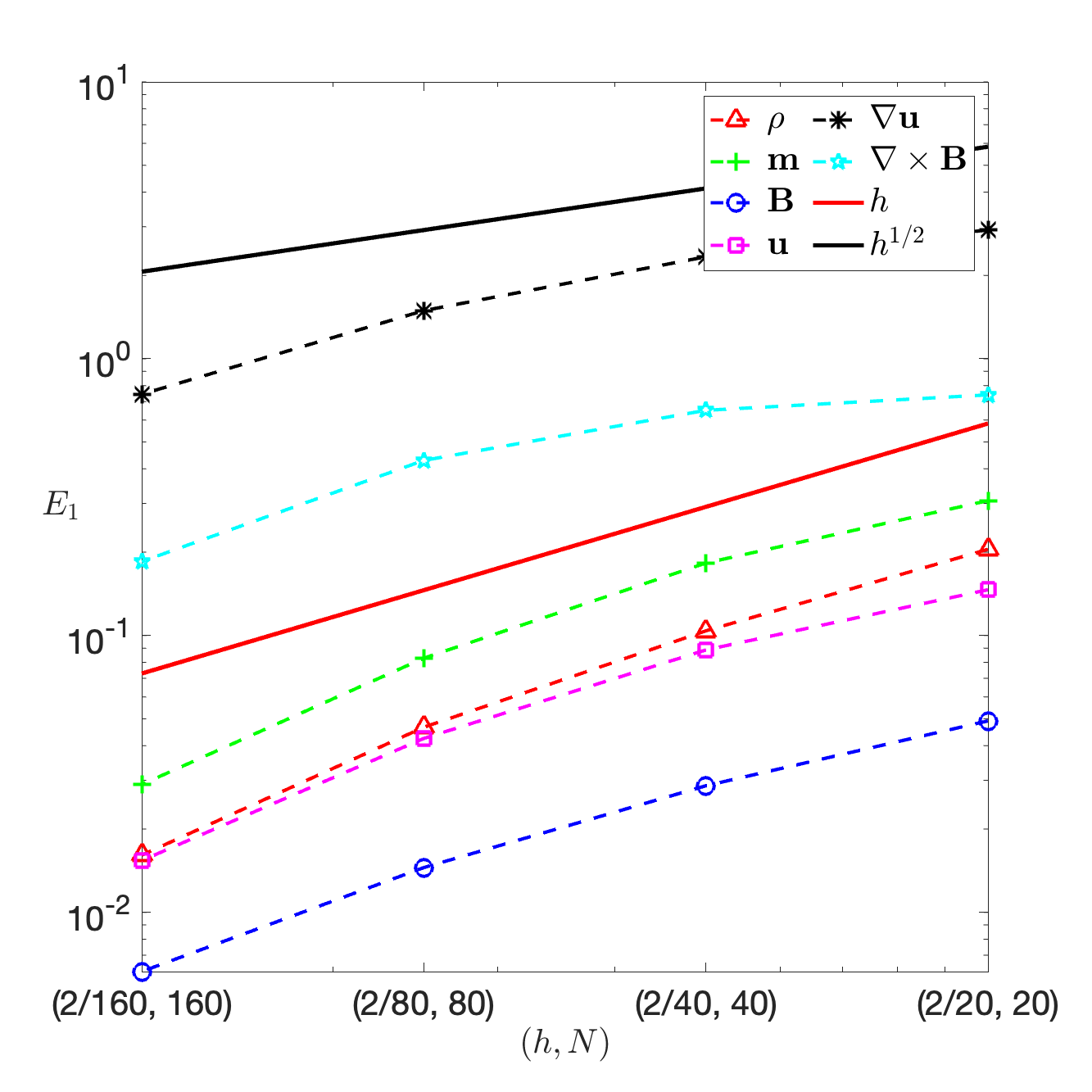}
	\end{subfigure}	
	\begin{subfigure}{0.45\textwidth}
	\includegraphics[width=\textwidth]{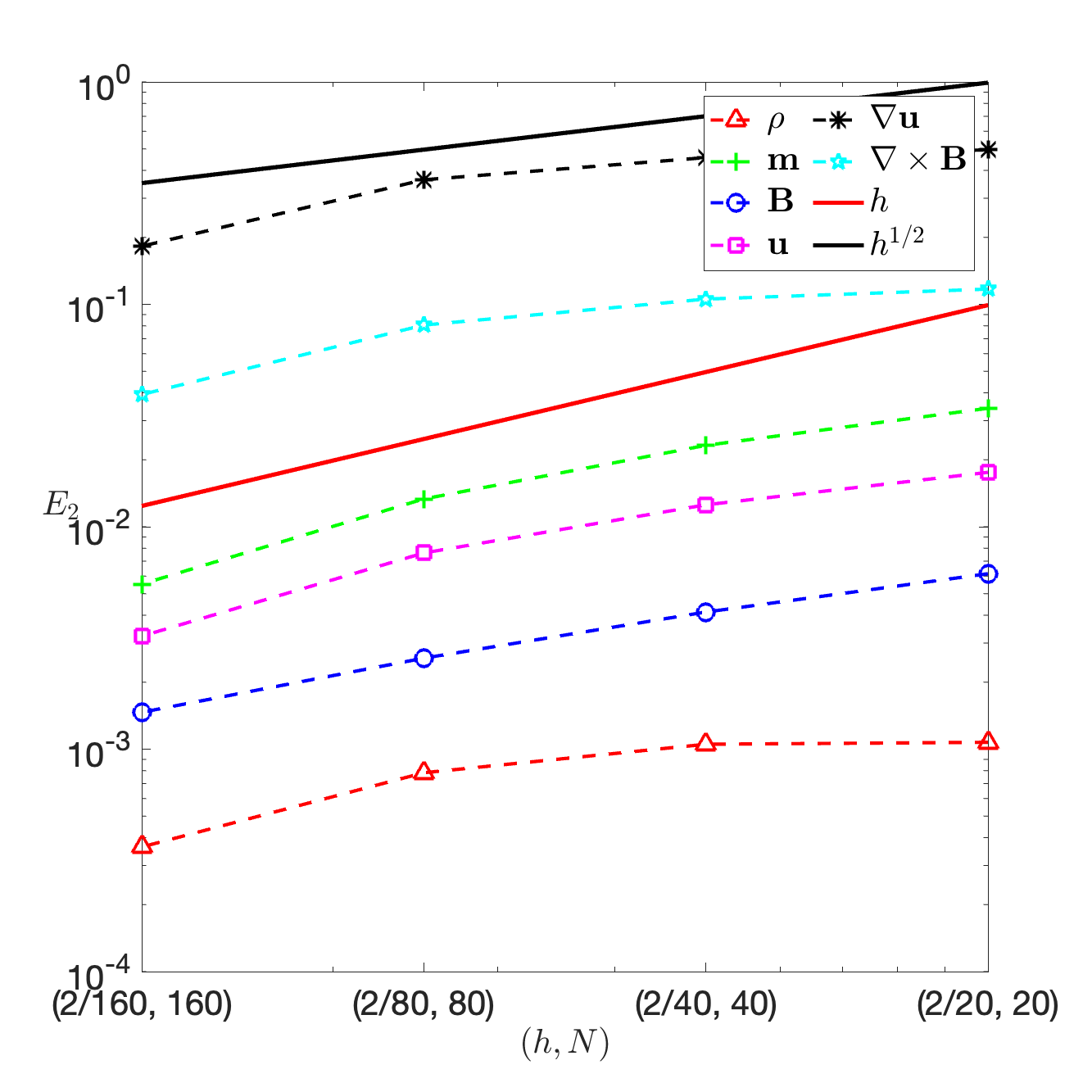}
	\end{subfigure}	
	\caption{ {\bf Random sine wave problem}. {\bf Total errors}: $E_1$ (left), $E_2$ (right).}\label{sine-err-2}
\end{figure}

\subsection{Kelvin-Helmholtz-type problem}
We consider the following computational domain $Q = [0,2]|_{\{0,2\}}\times[-0.5,0.5]$.
The initial data are taken as
\begin{align*}
& \rho(x,0) = \gamma, \quad  u_1(x,0) =  -0.1 \cos(2 \pi x_1) \sin(2 \pi x_2)  + Y_1(\omega) \sin(2 \pi  x_2), \\
& u_2(x,0) = 0.1\sin(2\pi x_1) [1+\cos(2\pi x_2)], \quad B_1 = 0.1 + Y_2(\omega) \sin \left( \pi x_2 \right), \quad  B_2(x,0) = 0
\end{align*}
and the boundary data are
\[
\vu|_{\partial Q} = \vc{0},\quad B_1|_{x_2=-0.5} = 0.1-Y_2(\omega), \quad B_1|_{x_2=0.5} = 0.1 + Y_2(\omega),
\]
where $Y_1, Y_2$ are i.i.d.\ uniform distributed  $Y_1, Y_2 \sim \mathcal{U}(-0.05,0.05)$.
The model parameters are taken as $\mu = 0.001, \, \lambda = 0, \, \zeta = 0.001, \gamma = 5/3$, $a=1$, and $b=0$.

%
%


Figure \ref{KH2} shows deterministic numerical solutions $(\vr, m_1, B_1)_{h_{ref}}(\omega=0,T)$  and streamlines of $\vu_{h_{ref}}$ and $\vB_{h_{ref}}$ obtained on a uniform squared mesh with the parameter  $h_{ref}=1/320$  at the final time $T = 2.2$. We also present the time evolution of $\| \Divh \vc{B}_h\|_{L^{\infty}(Q)}(\omega=0, t)$ with $ h = 1/(10 \cdot 2^i), i = 1,\dots,5$.  The random approximations are displayed in Figures \ref{KH2-1} and \ref{KH2-2}.

Further, Figure~\ref{KH2-err-1} shows  the statistical errors obtained using $N$ samples, $N=10 \cdot 2^n, n = 1,2,3,4$.
The total errors of statistical solutions $(\vr, \vu, \vB)_h^N$  with $(h,N) = (1/(10 \cdot 2^n), 10 \cdot 2^n), n = 1,2,3,4$ are displayed in  Figure~\ref{KH2-err-2}.

\begin{figure}[htbp]
	\setlength{\abovecaptionskip}{0.cm}
	\setlength{\belowcaptionskip}{0.2cm}
	\centering
	\begin{subfigure}{0.42\textwidth}
	\includegraphics[width=\textwidth]{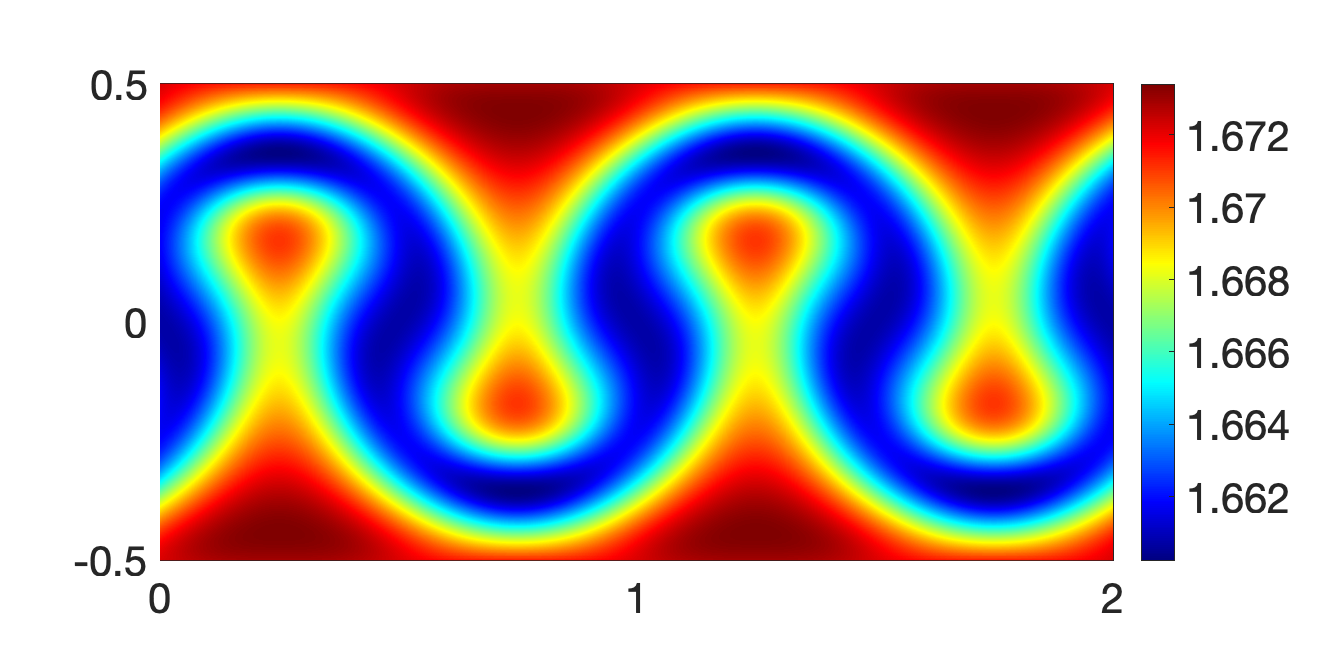}
\caption{ $\vr_{h_{ref}}(\omega=0,T)$}
	\end{subfigure}	
	\begin{subfigure}{0.42\textwidth}
	\includegraphics[width=\textwidth]{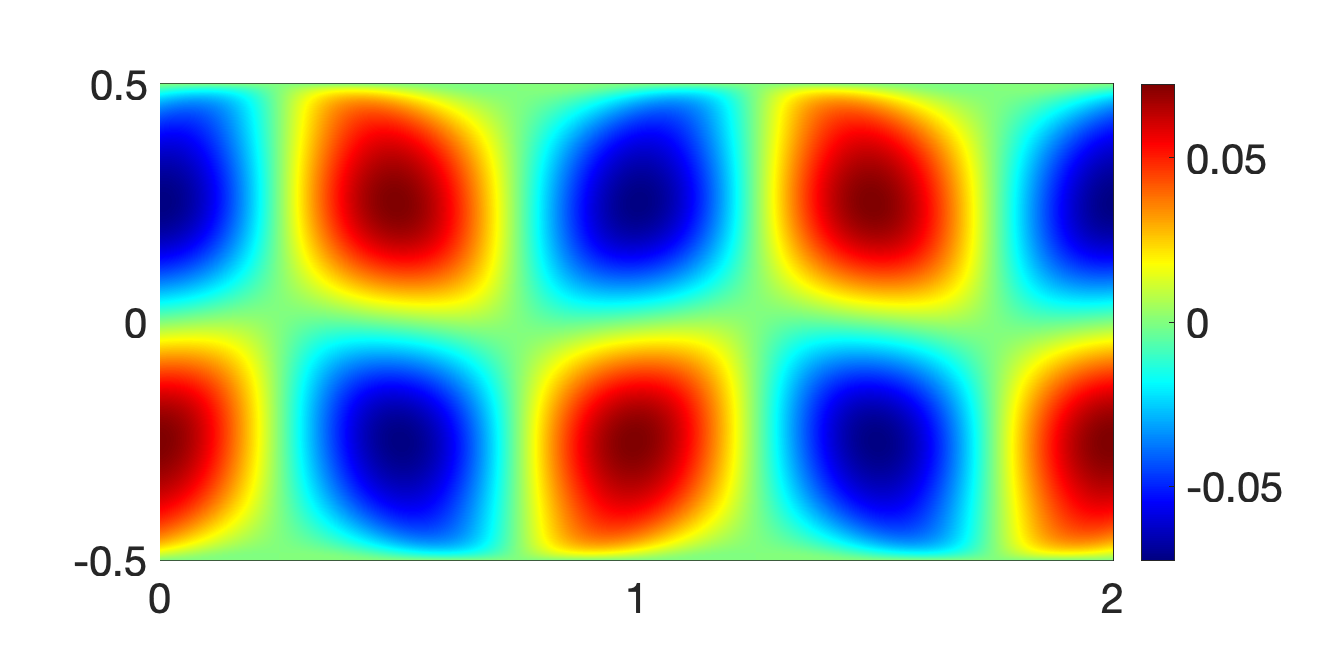}
\caption{ $(m_1)_{h_{ref}}(\omega=0,T)$}
	\end{subfigure}	\\	
	\begin{subfigure}{0.42\textwidth}
	\includegraphics[width=\textwidth]{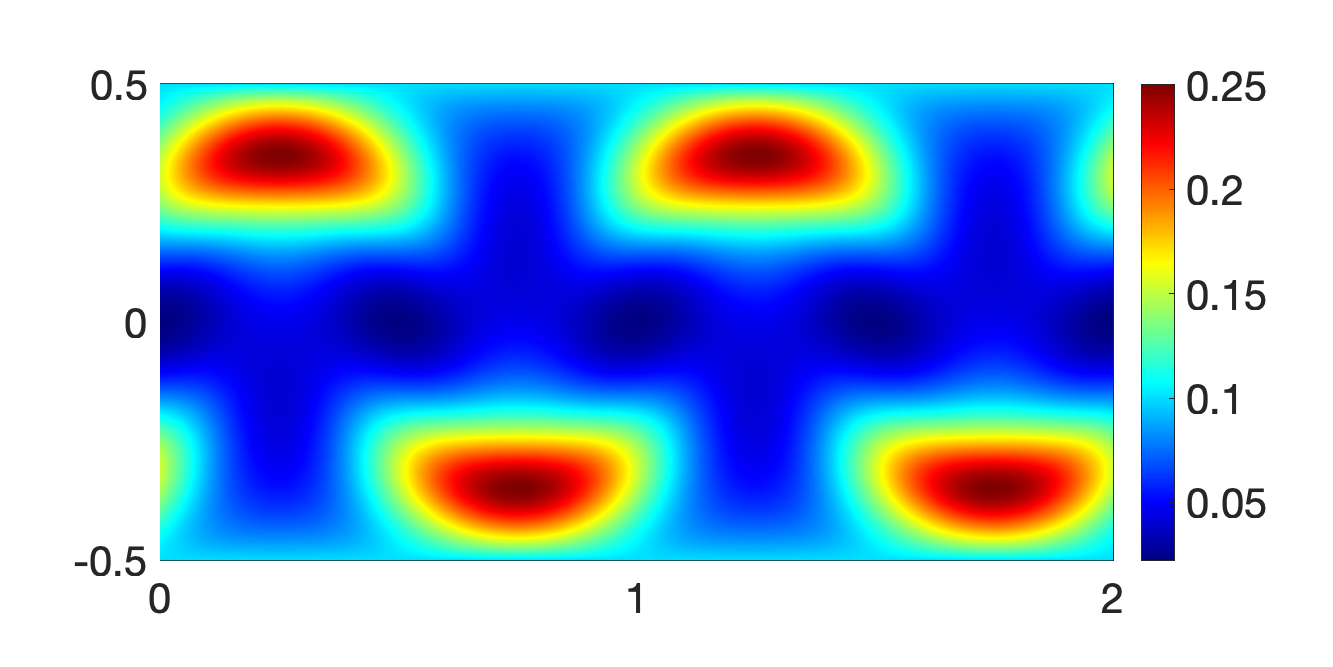}
\caption{ $(B_1)_{h_{ref}}(\omega=0,T)$}
	\end{subfigure}	
	\begin{subfigure}{0.42\textwidth}
	\includegraphics[width=\textwidth]{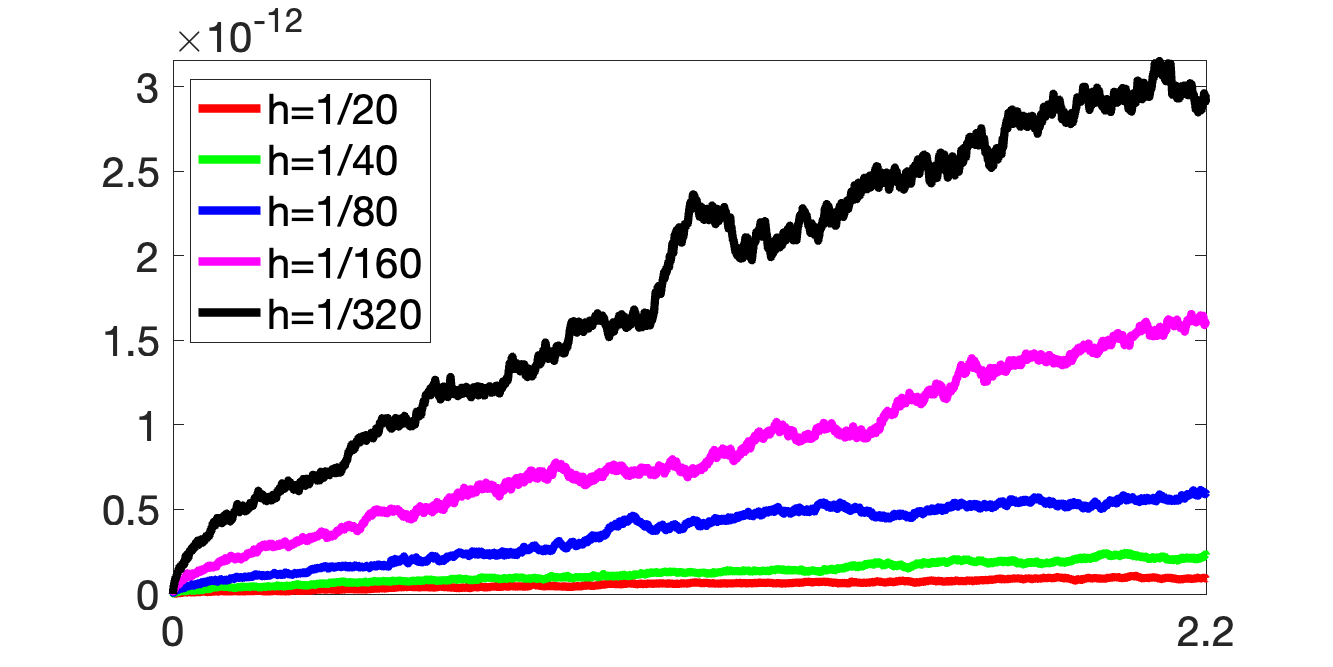}
	\caption{ $\| \Divh \vc{B}_h\|_{L^{\infty}(Q)}(\omega=0, t)$ }
	\end{subfigure}	\\
	\begin{subfigure}{0.42\textwidth}
	\includegraphics[width=\textwidth]{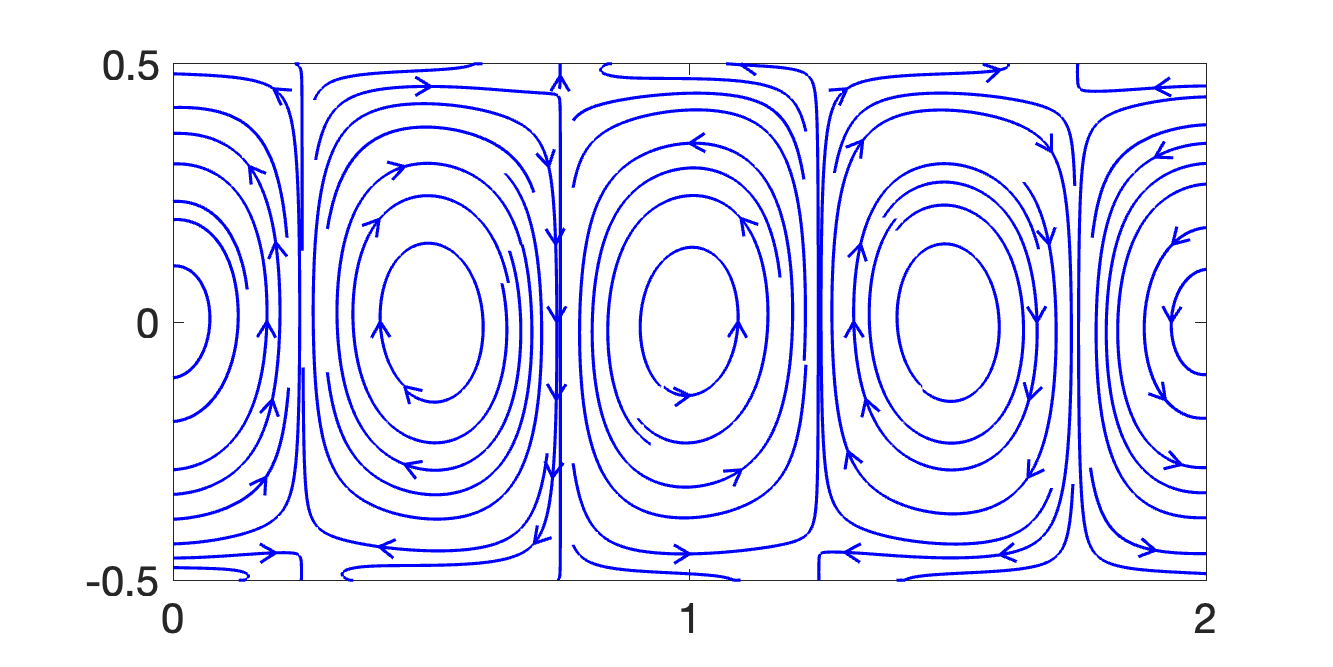}
\caption{ streamline of $\vu_{h_{ref}}(\omega=0,T)$}
	\end{subfigure}	
	\begin{subfigure}{0.42\textwidth}
	\includegraphics[width=\textwidth]{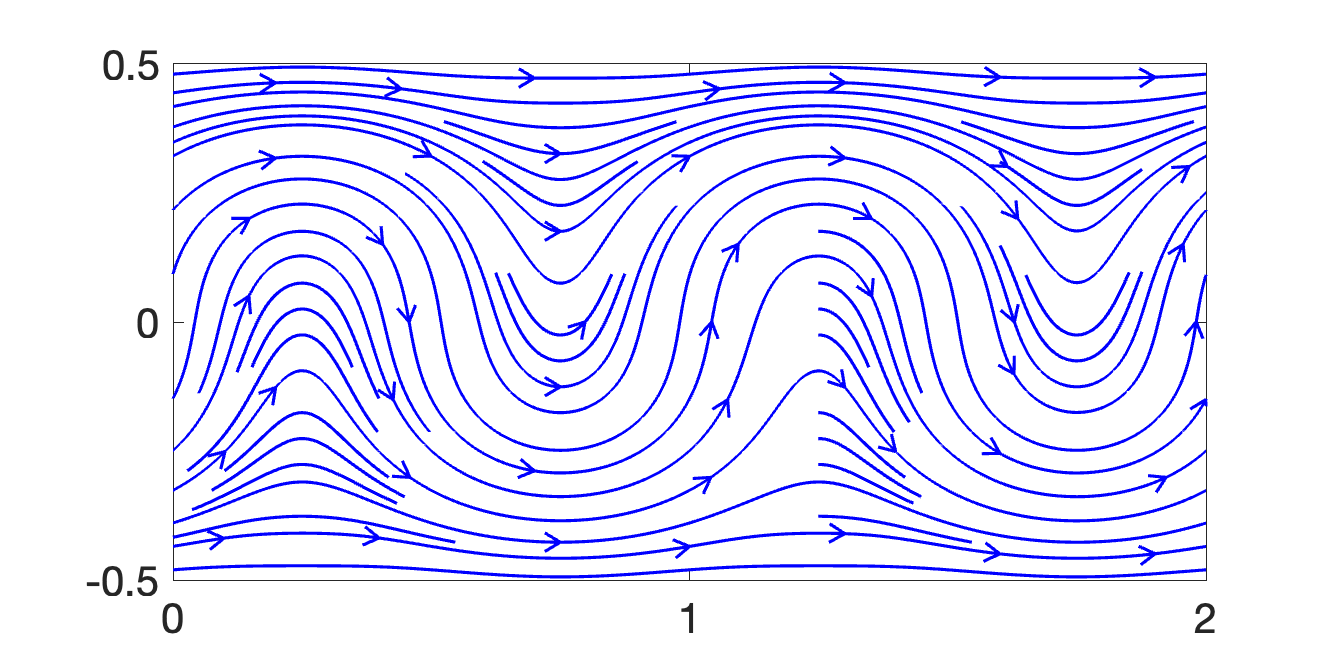}
\caption{ streamline of $\vB_{h_{ref}}(\omega=0,T)$}
	\end{subfigure}	
	\caption{{\bf Kelvin-Helmholtz-type problem}. Deterministic FV  solutions $(\vr, m_1, B_1)_{h_{ref}}(\omega=0,T)$ and streamlines of $\vu_{h_{ref}}$ and $\vB_{h_{ref}}$ with $h_{ref} = 1/320$ at $T=2.2$ and the time evolution of $\| \Divh \vc{B}_h\|_{L^{\infty}(Q)}(\omega=0, t)$ with $h = 1/(10 \cdot 2^i), i = 1,\dots,5$.}\label{KH2}
\end{figure}

\begin{figure}[htbp]
	\setlength{\abovecaptionskip}{0.cm}
	\setlength{\belowcaptionskip}{-0.cm}
	\centering
	\begin{subfigure}{0.32\textwidth}
	\includegraphics[width=\textwidth]{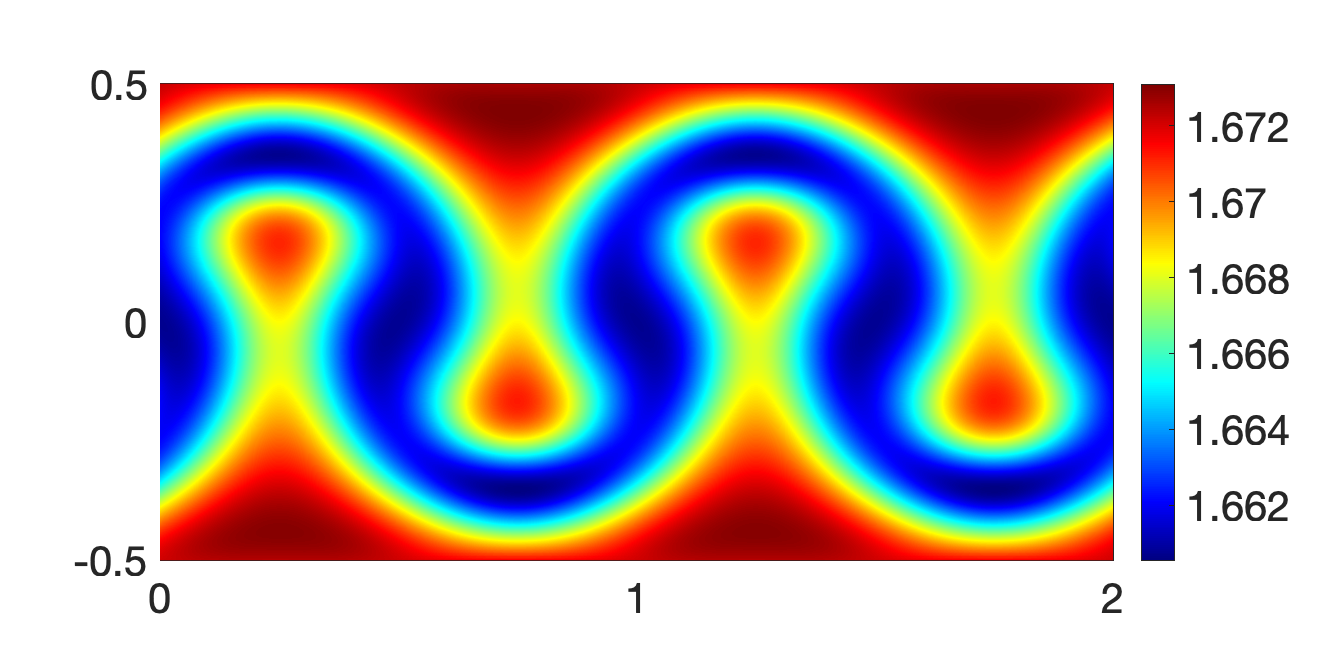}
	\end{subfigure}	
	\begin{subfigure}{0.32\textwidth}
	\includegraphics[width=\textwidth]{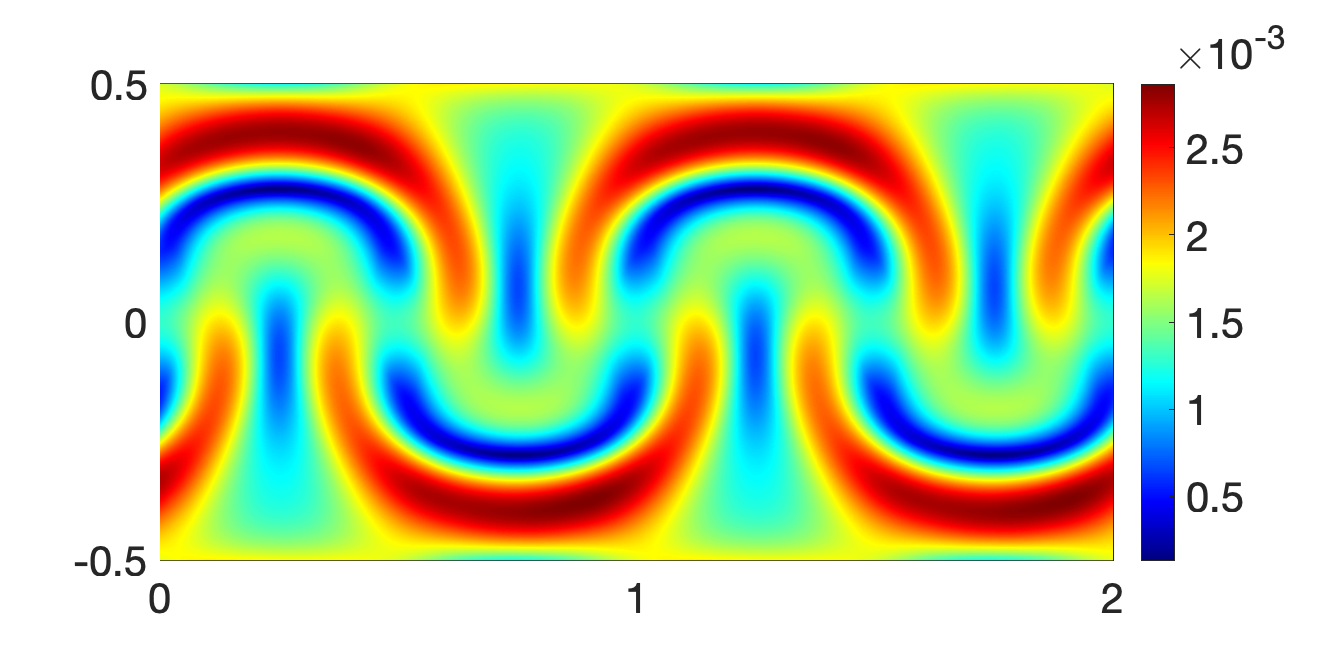}
	\end{subfigure}	
	\begin{subfigure}{0.32\textwidth}
	\includegraphics[width=\textwidth]{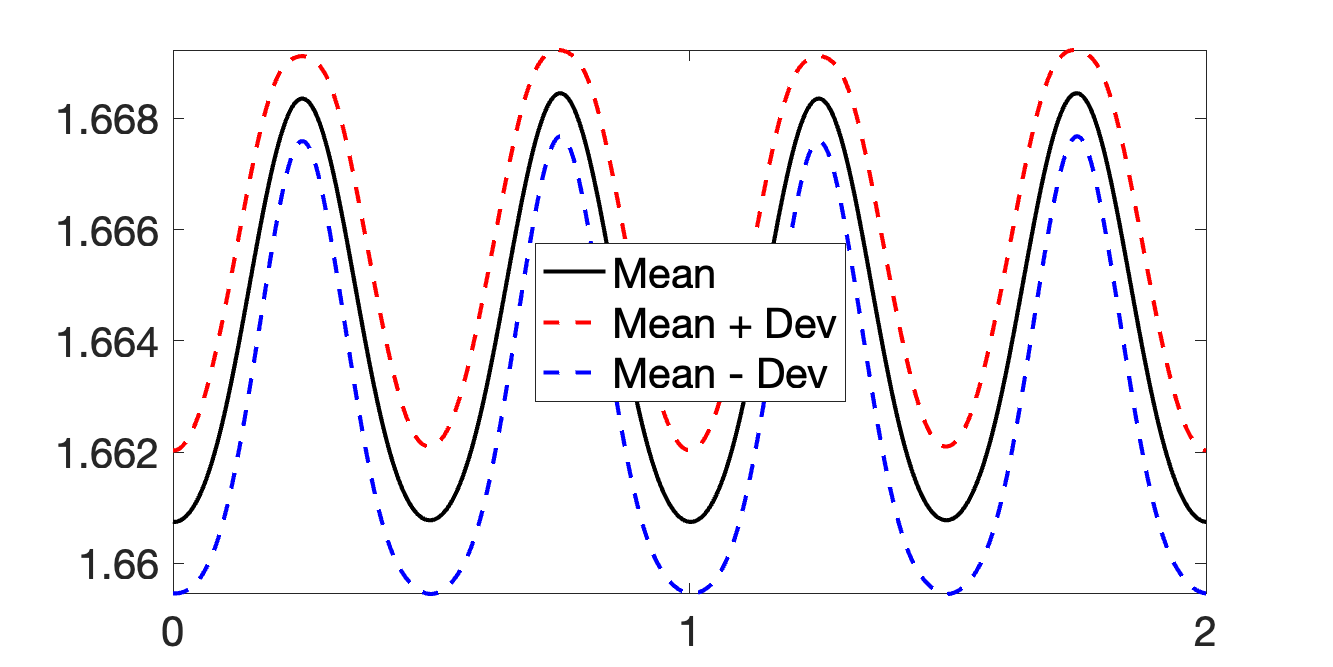}
	\end{subfigure}	\\
	\begin{subfigure}{0.32\textwidth}
	\includegraphics[width=\textwidth]{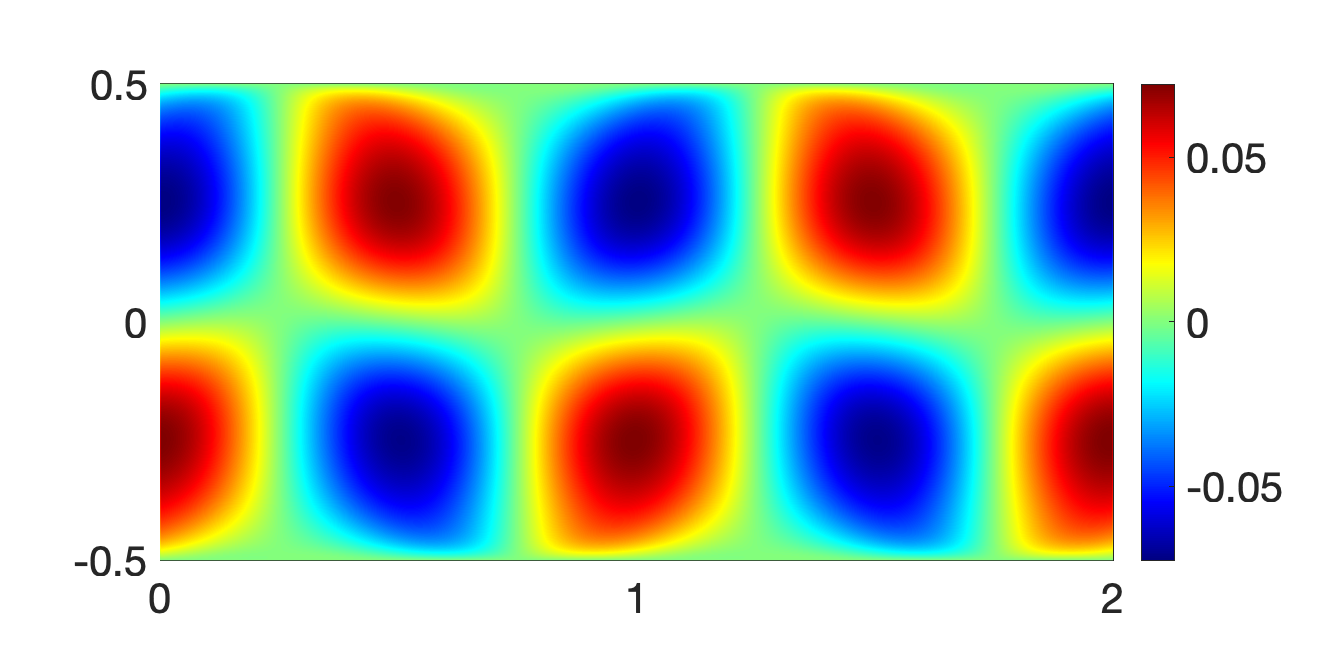}
	\end{subfigure}	
	\begin{subfigure}{0.32\textwidth}
	\includegraphics[width=\textwidth]{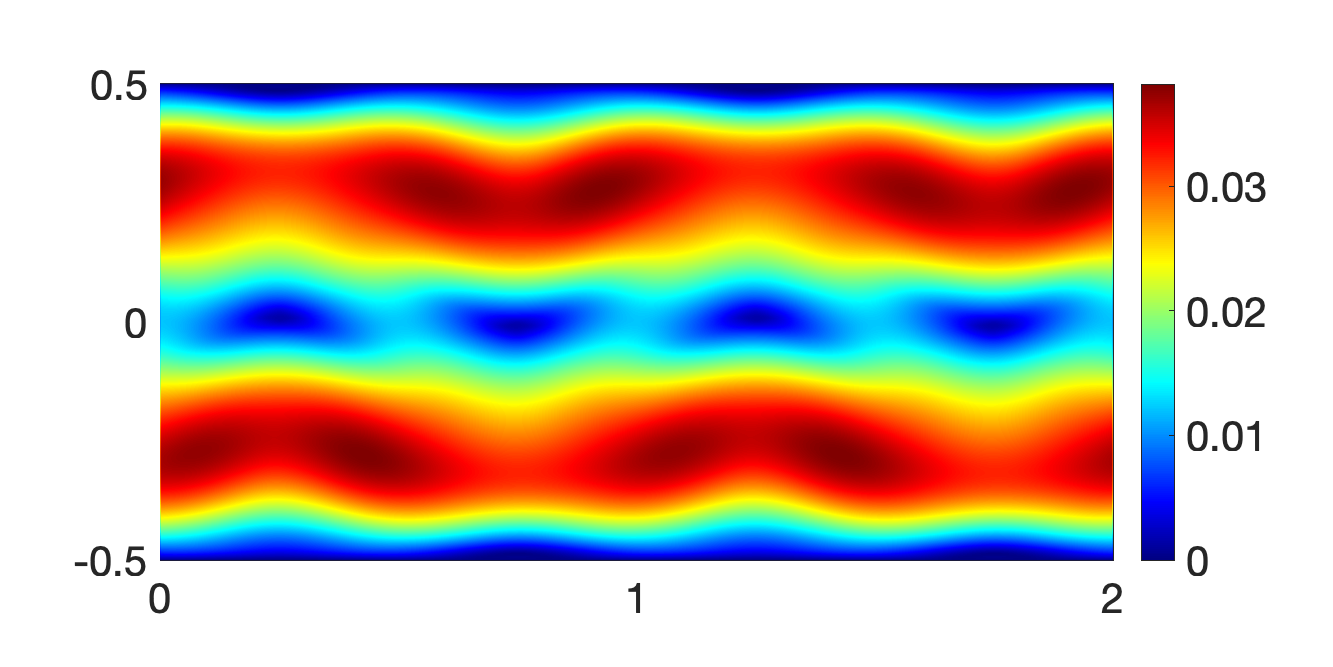}
	\end{subfigure}	
	\begin{subfigure}{0.32\textwidth}
	\includegraphics[width=\textwidth]{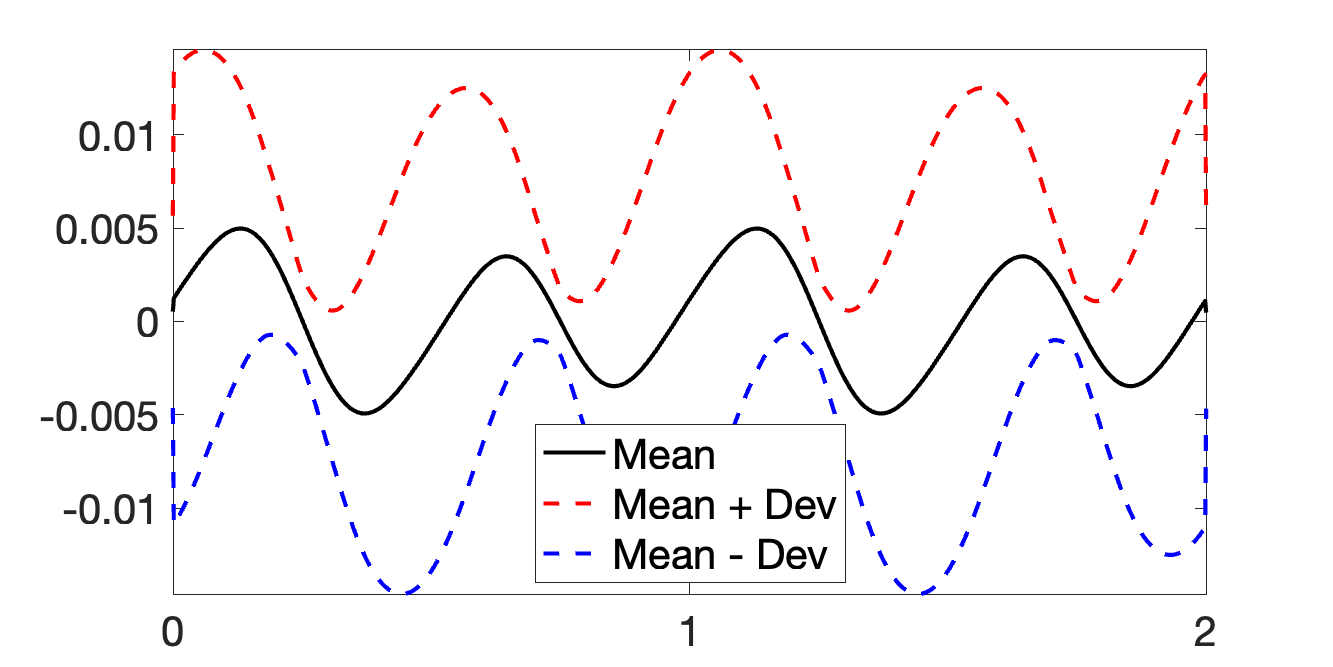}
	\end{subfigure}	\\
	\begin{subfigure}{0.32\textwidth}
	\includegraphics[width=\textwidth]{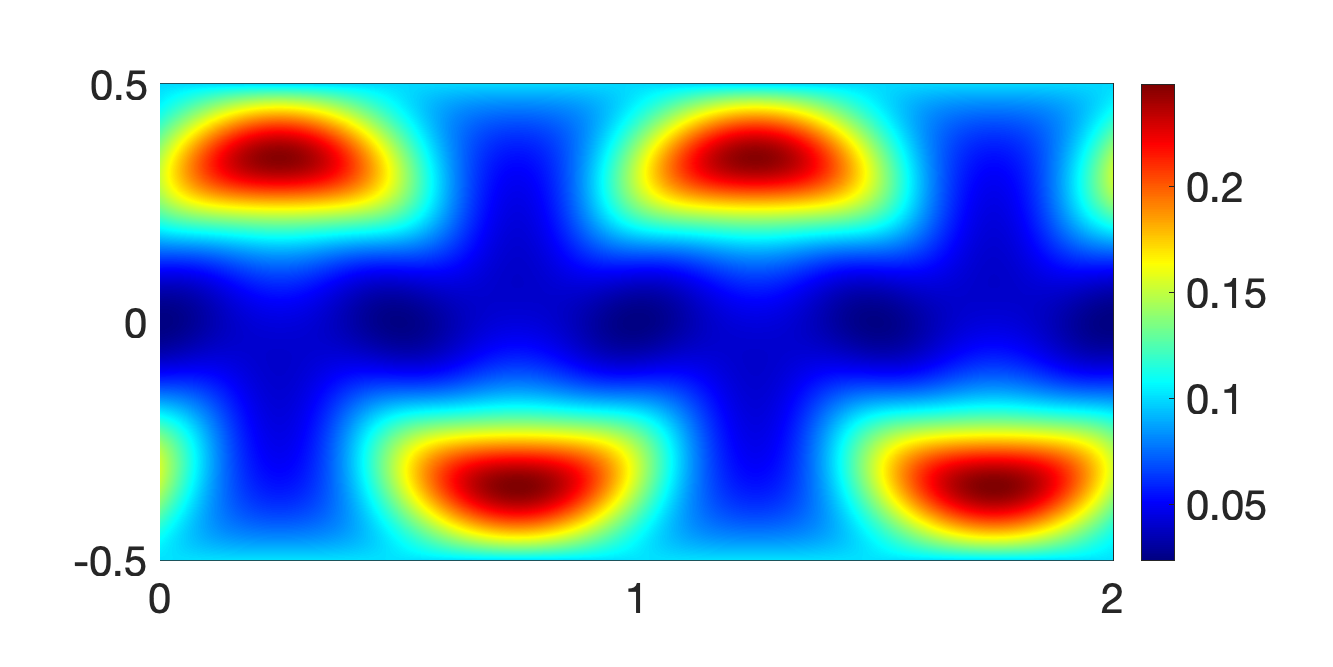}
	\end{subfigure}	
	\begin{subfigure}{0.32\textwidth}
	\includegraphics[width=\textwidth]{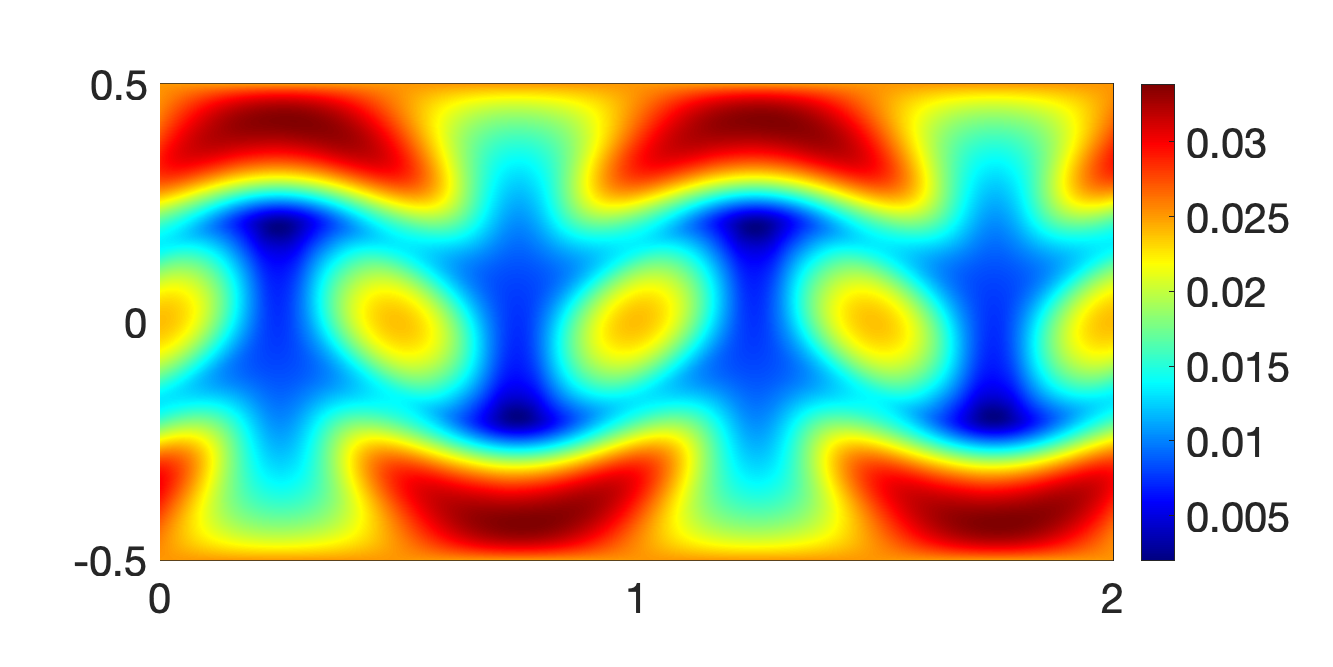}
	\end{subfigure}	
	\begin{subfigure}{0.32\textwidth}
	\includegraphics[width=\textwidth]{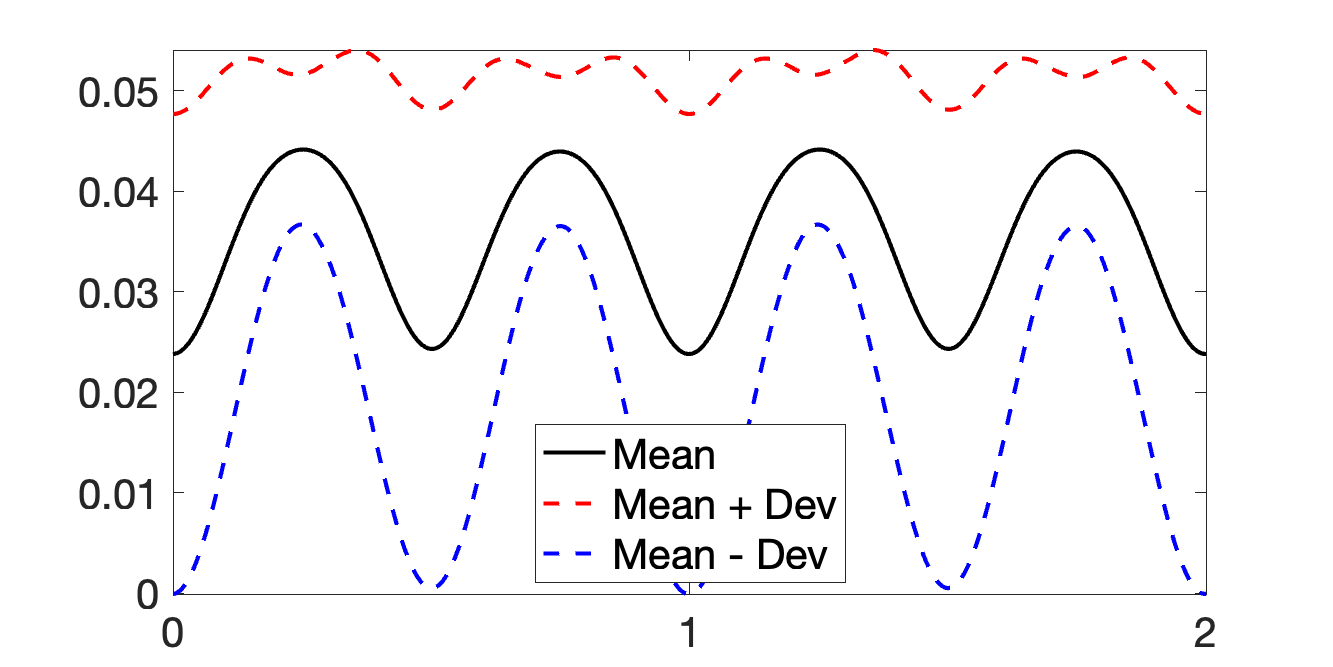}
	\end{subfigure}	
	\caption{ {\bf Random Kelvin-Helmholtz-type problem}. FV  solutions $(\vr, m_1, B_1)_{h_{ref}}(\omega,T)$ (from top to bottom) with $h_{ref} = 1/320, \, T=2.2$ and $500$ samples. From left to right: mean, deviation, mean and deviations along $y=0$.}\label{KH2-1}
\end{figure}

\begin{figure}[htbp]
	\setlength{\abovecaptionskip}{0.cm}
	\setlength{\belowcaptionskip}{-0.cm}
	\centering
	\begin{subfigure}{0.45\textwidth}
	\includegraphics[width=\textwidth]{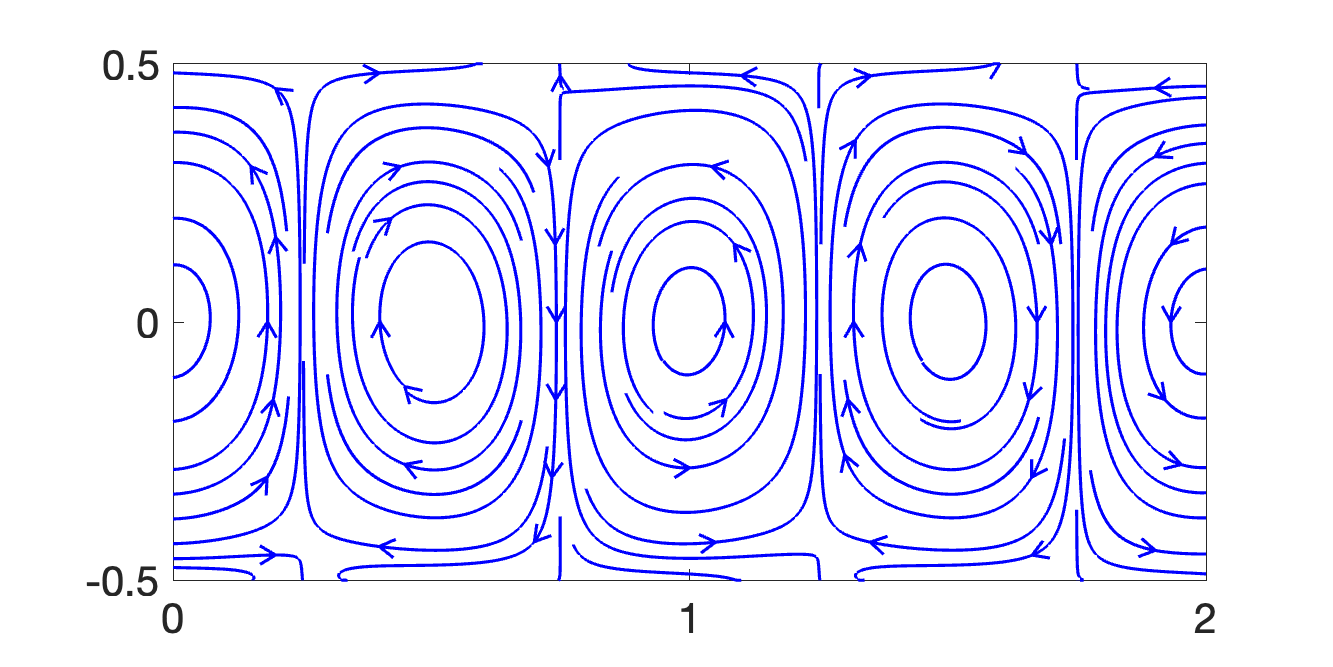}
	\caption{streamline of $\expe{ \vu_{h_{ref}} (T) } $}
	\end{subfigure}	
	\begin{subfigure}{0.45\textwidth}
	\includegraphics[width=\textwidth]{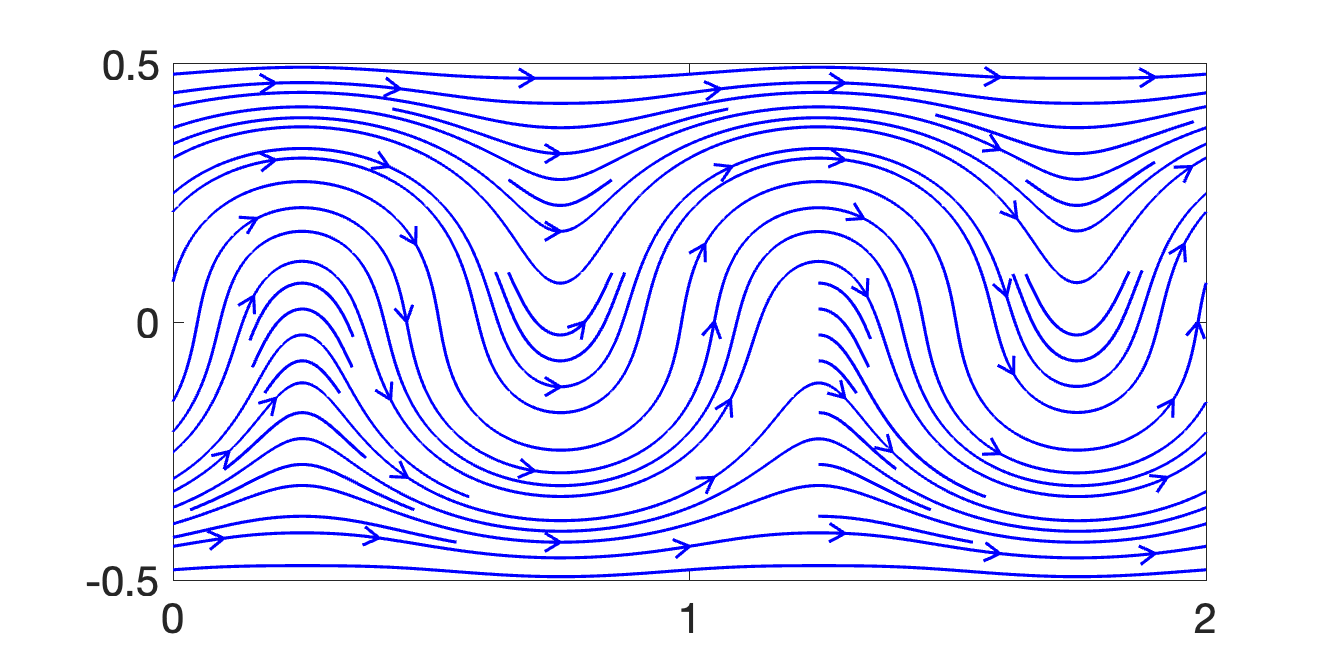}
	\caption{streamline of  $\expe{ \vB_{h_{ref}} (T) } $}
	\end{subfigure}	
	\caption{ {\bf Random Kelvin-Helmholtz-type problem}.  Streamlines of mean values of  $(\vu, \vB)_{h_{ref}}(\omega,T)$ with $h_{ref} = 1/320, \, T=2.2$ and $500$ samples. }\label{KH2-2}
\end{figure}

\begin{figure}[htbp]
	\setlength{\abovecaptionskip}{0.cm}
	\setlength{\belowcaptionskip}{-0.cm}
	\centering
	\begin{subfigure}{0.45\textwidth}
	\includegraphics[width=\textwidth]{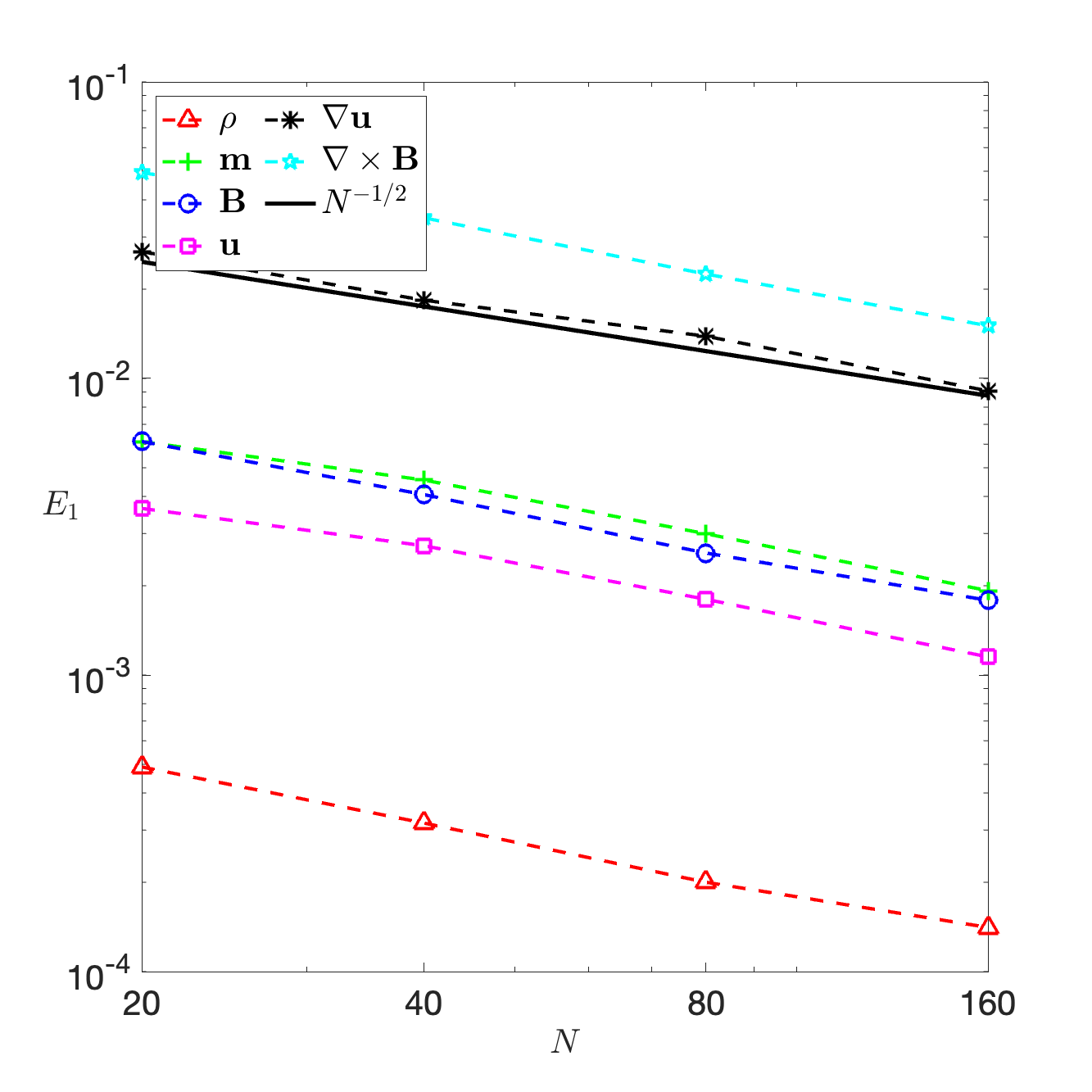}
	\end{subfigure}	
	\begin{subfigure}{0.45\textwidth}
	\includegraphics[width=\textwidth]{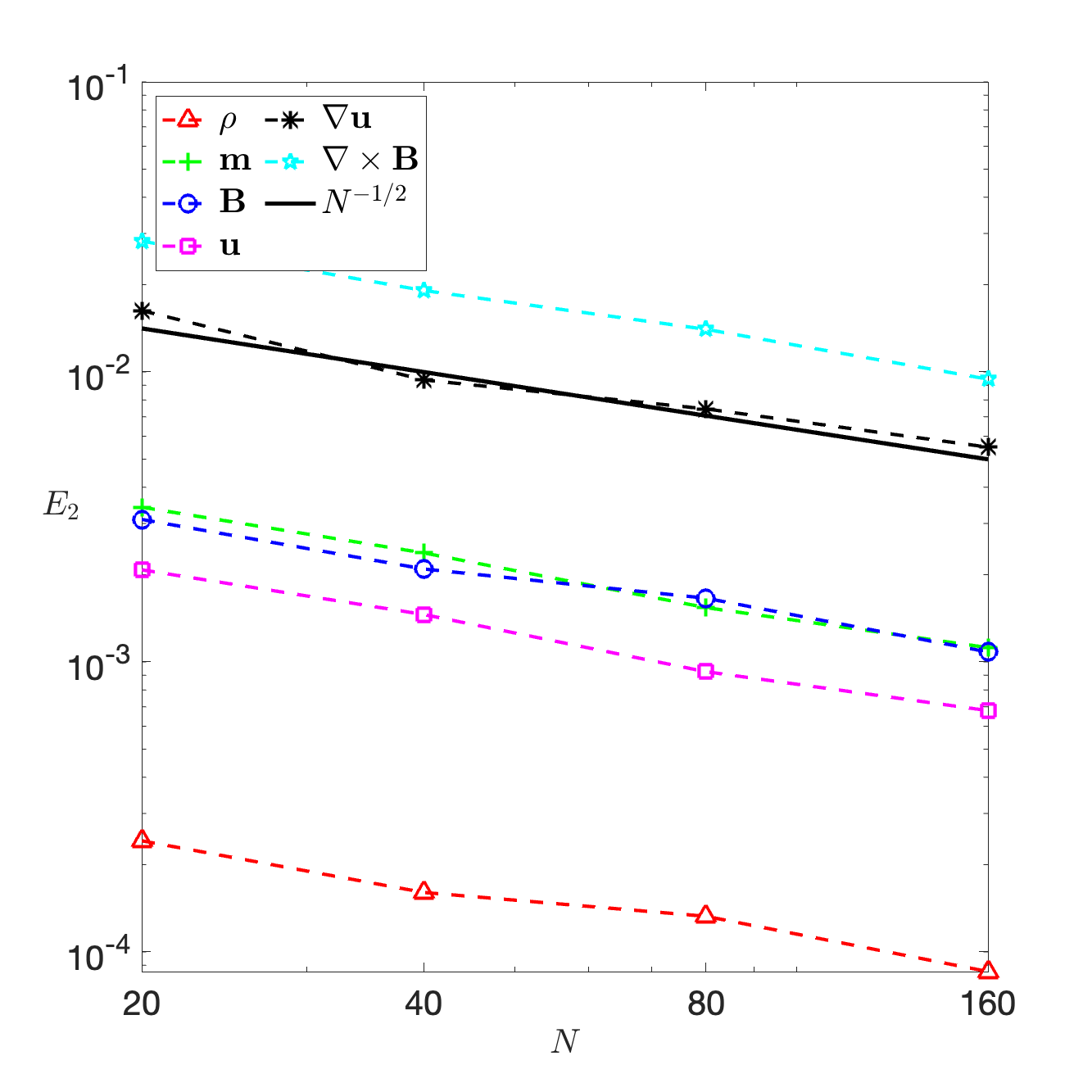}
	\end{subfigure}	
	\caption{ {\bf Random Kelvin-Helmholtz-type problem}. {\bf Statistical errors}: $E_1$ (left), $E_2$ (right).}\label{KH2-err-1}
\end{figure}

\begin{figure}[htbp]
	\setlength{\abovecaptionskip}{0.cm}
	\setlength{\belowcaptionskip}{-0.cm}
	\centering
	\begin{subfigure}{0.45\textwidth}
	\includegraphics[width=\textwidth]{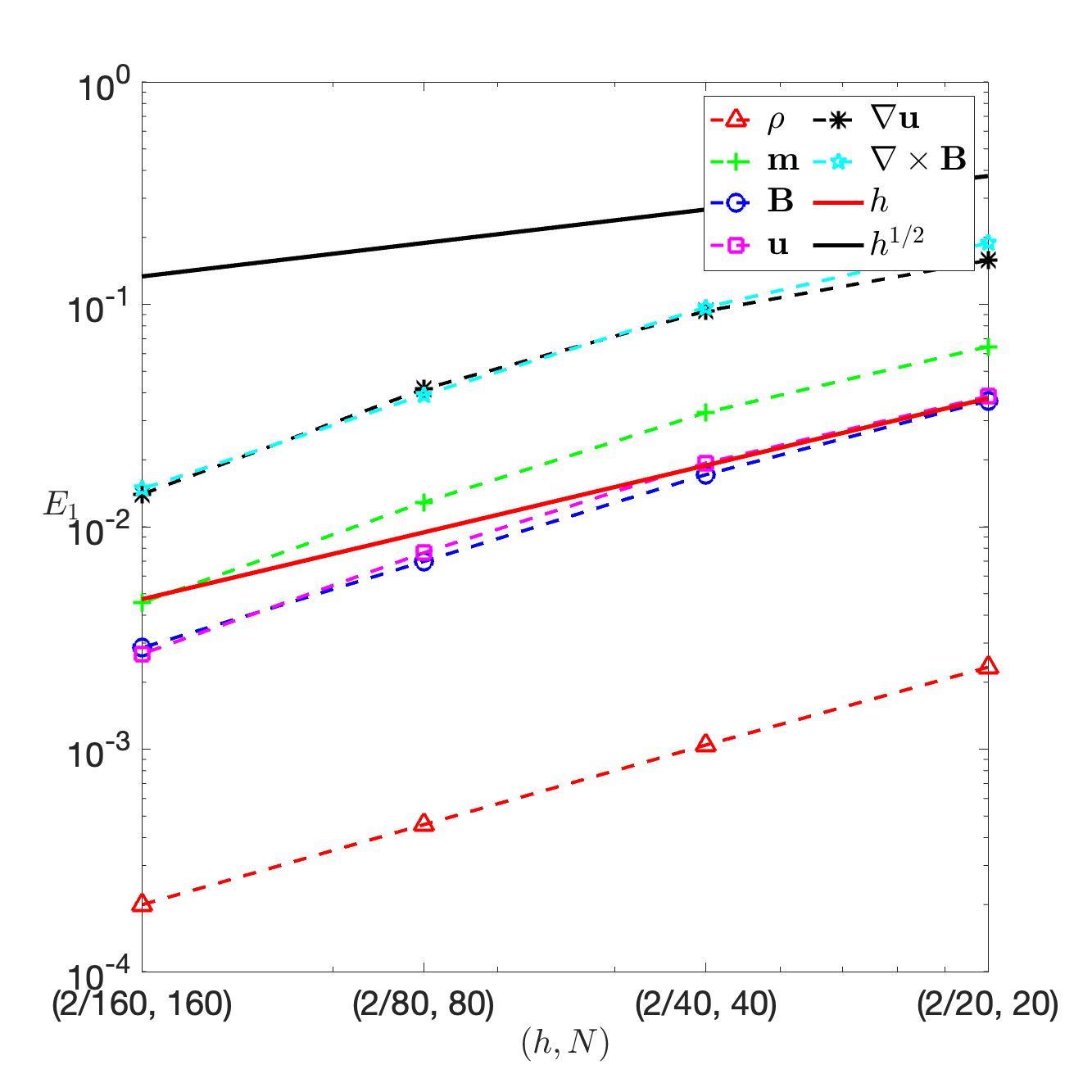}
	\end{subfigure}	
	\begin{subfigure}{0.45\textwidth}
	\includegraphics[width=\textwidth]{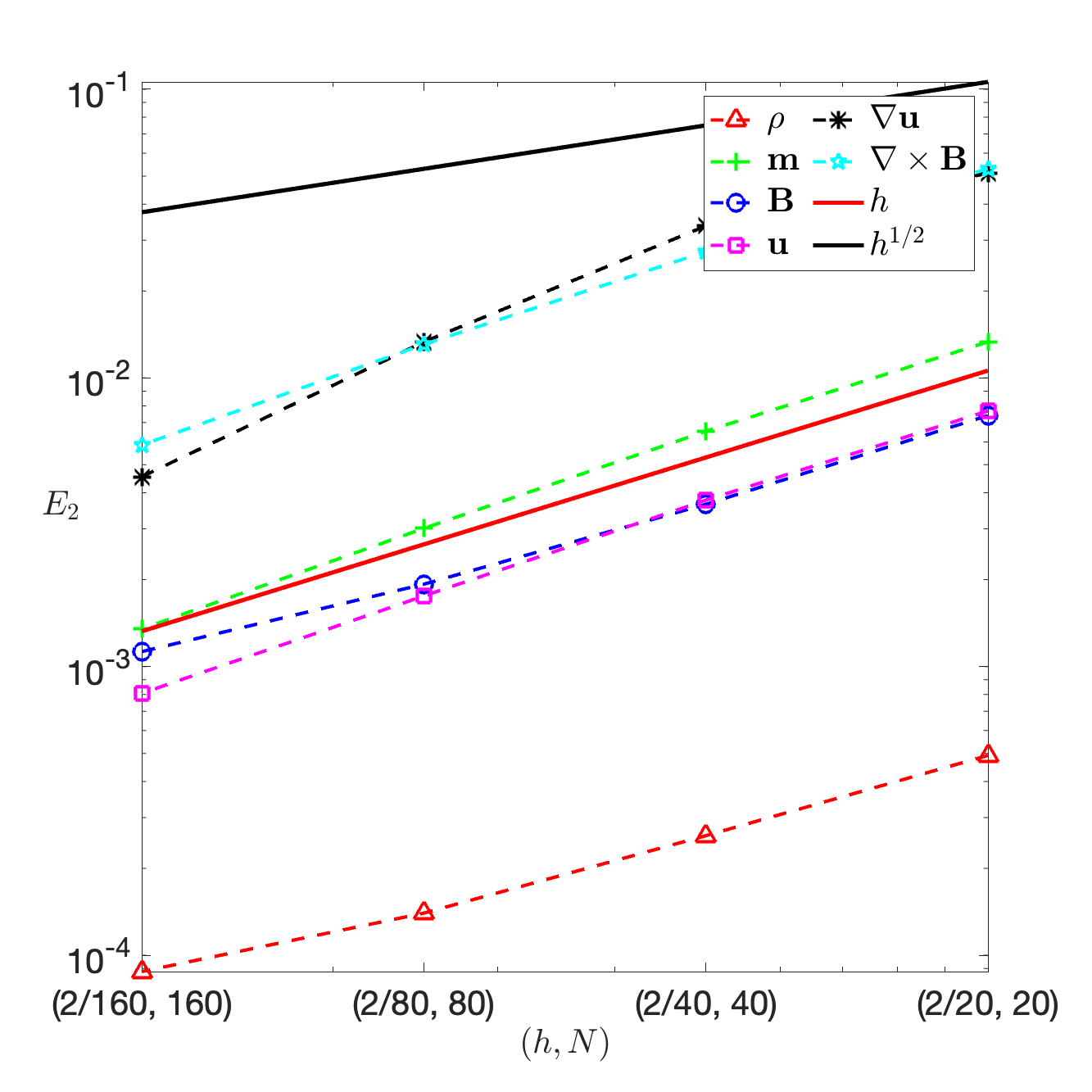}
	\end{subfigure}	
	\caption{ {\bf Random Kelvin-Helmholtz-type problem}. {\bf Total errors}: $E_1$ (left), $E_2$ (right).}\label{KH2-err-2}
\end{figure}


\subsection{Orszag-Tang-type problem}
In this experiment we
consider $Q = [0,2\pi]|_{\{0,2\pi\}}\times[0,2\pi]$. The initial data are given as
\begin{align*}
& \vr(x,0) = \gamma^2, \quad u_1(x,0) = -\sin(x_2) + Y_1(\omega) \sin(x_2), \quad  u_2(x,0) = 0,\\
& B_1(x,0) = -\sin(x_2) + Y_2(\omega) \sin \left( \frac{x_2}{4} \right),\quad B_2(x,0) = \sin( 2 x_1),
\end{align*}
and the boundary data are
\[
\vu|_{\partial Q} = \vc{0},\quad B_1|_{x_2=2\pi} = Y_2(\omega), \quad B_1|_{x_2=0} = 0,
\]
where $Y_1, Y_2$ are i.i.d.\ Gaussian variables and $Y_1, Y_2 \sim \mathcal{N}(0,0.1)$.
The model parameters are taken as $\mu = 0.01, \, \lambda = 0, \, \zeta = 0.01, \gamma = 5/3$, $a=1$, and $b=0$.

Figure \ref{OT} shows deterministic numerical solutions $(\vr, m_1, B_1)_{h_{ref}}(\omega=0,T)$ obtained on a uniform squared mesh with $h_{ref}=2\pi/400$  at the final time $T = 3$, as well as  the time evolution of $\| \Divh \vc{B}_h\|_{L^{\infty}(Q)}(\omega=0, t)$ with $ h = 2\pi/(25 \cdot 2^i), i = 0,\dots,4$.
The mean and deviations of statistical numerical solutions $(\vr, m_1, B_1)_{h_{ref}}(\omega,T)$ with $M_{ref}$ samples are displayed in Figure \ref{OT-MC-1}.

\begin{figure}[htbp]
	\setlength{\abovecaptionskip}{0.cm}
	\setlength{\belowcaptionskip}{0.2cm}
	\centering
	\begin{subfigure}{0.42\textwidth}
	\includegraphics[width=\textwidth]{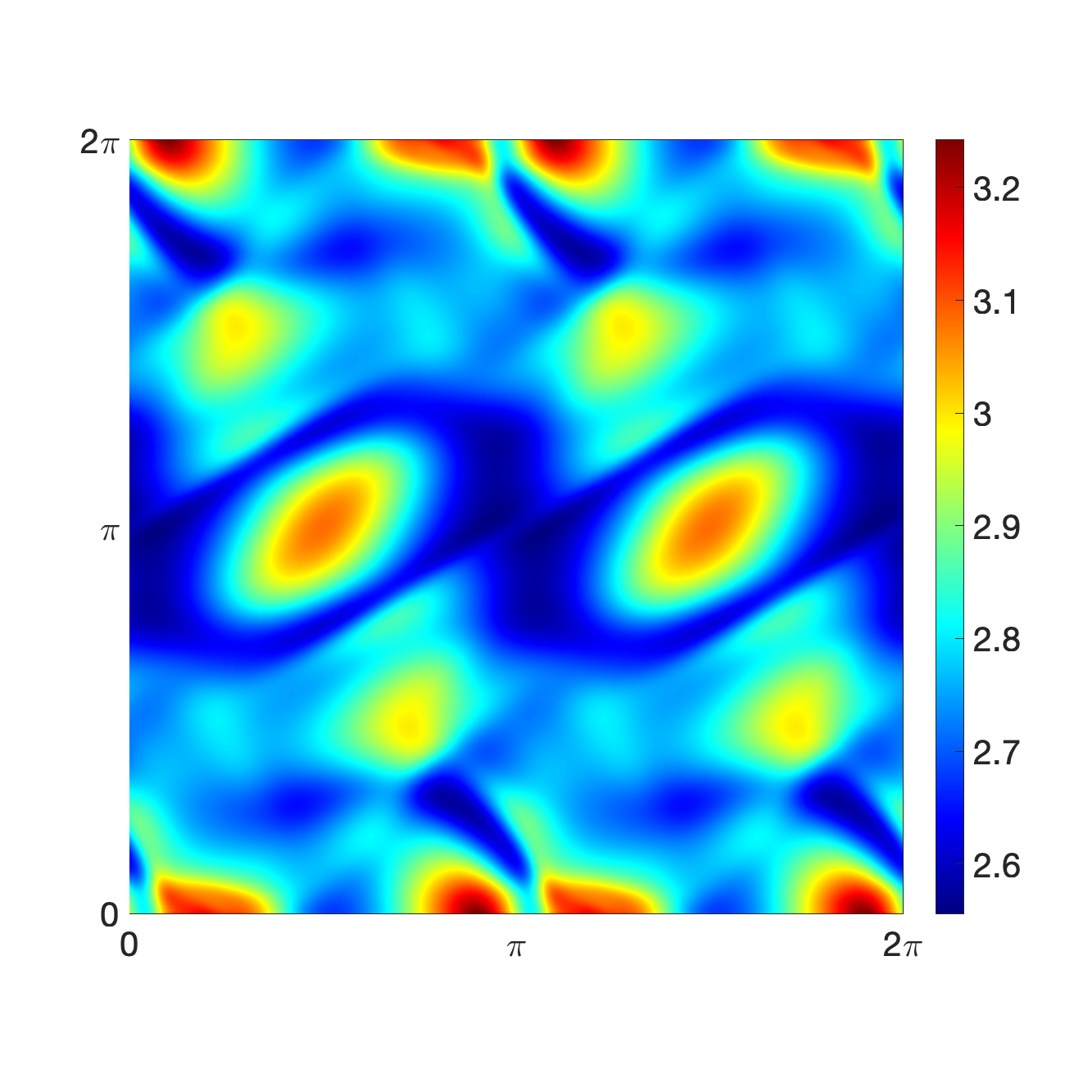}
\caption{ $\vr_{h_{ref}}(\omega=0, T)$}
	\end{subfigure}	
	\begin{subfigure}{0.42\textwidth}
	\includegraphics[width=\textwidth]{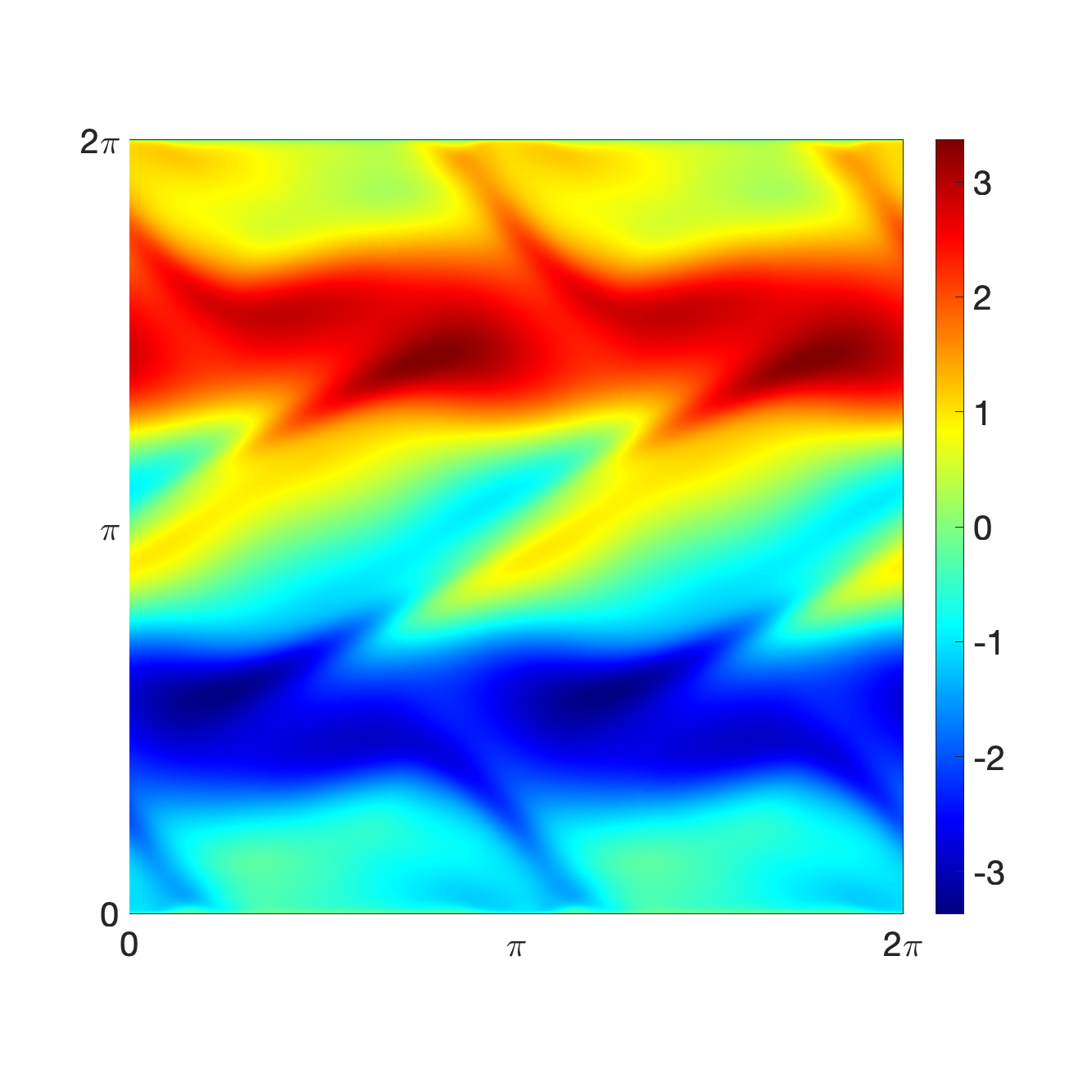}
\caption{ $(m_1)_{h_{ref}}(\omega=0, T)$}
	\end{subfigure}	\\
	\begin{subfigure}{0.42\textwidth}
	\includegraphics[width=\textwidth]{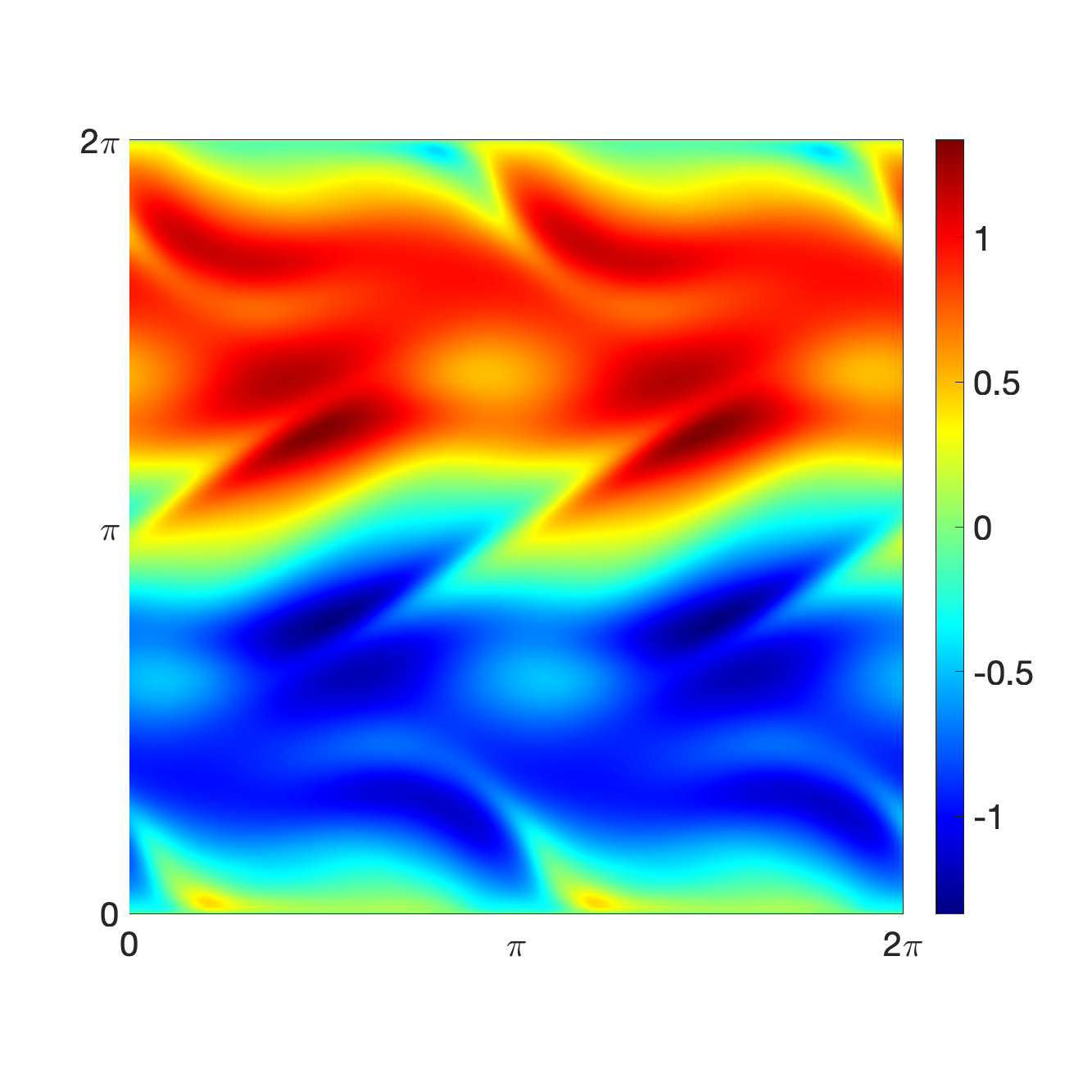}
\caption{ $(B_1)_{h_{ref}}(\omega=0, T)$}
	\end{subfigure}	
	\begin{subfigure}{0.42\textwidth}
	\includegraphics[width=\textwidth]{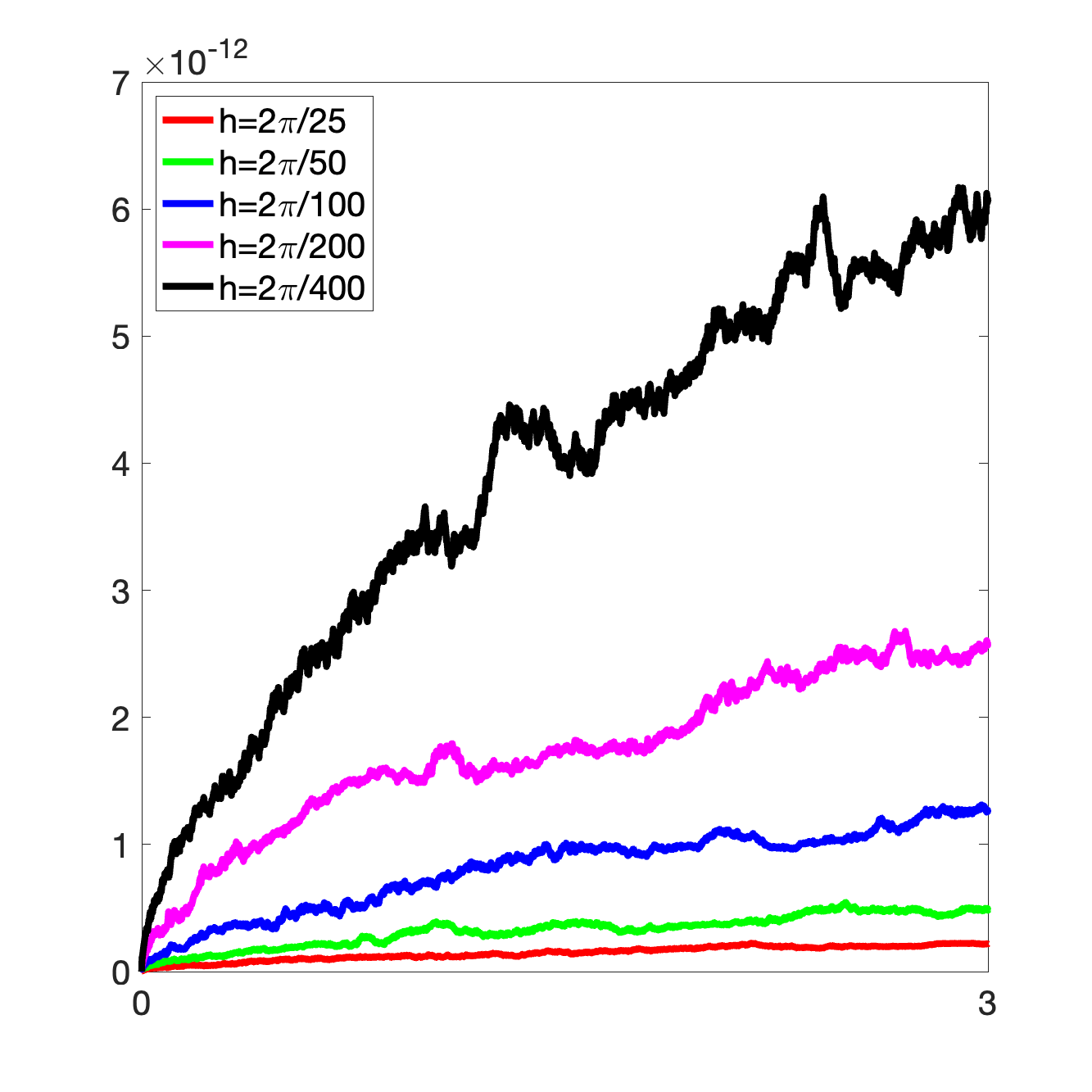}
	\caption{ $\| \Divh \vc{B}_h\|_{L^{\infty}(Q)}(\omega=0, t)$ }
	\end{subfigure}	
	\caption{{\bf Orszag-Tang-type problem}. Deterministic FV  solutions $(\vr, m_1, B_1)_{h_{ref}}(\omega=0,T)$ with $h_{ref} = 2\pi /400$ at  $T=3$ and the time evolution of $\| \Divh \vc{B}\|_{L^{\infty}(Q)}(\omega=0, t)$ with $h = 2\pi/(25 \cdot 2^i), i = 0,\dots,4$.}\label{OT}
\end{figure}

\begin{figure}[htbp]
	\setlength{\abovecaptionskip}{0.cm}
	\setlength{\belowcaptionskip}{-0.cm}
	\centering
	\begin{subfigure}{0.32\textwidth}
	\includegraphics[width=\textwidth]{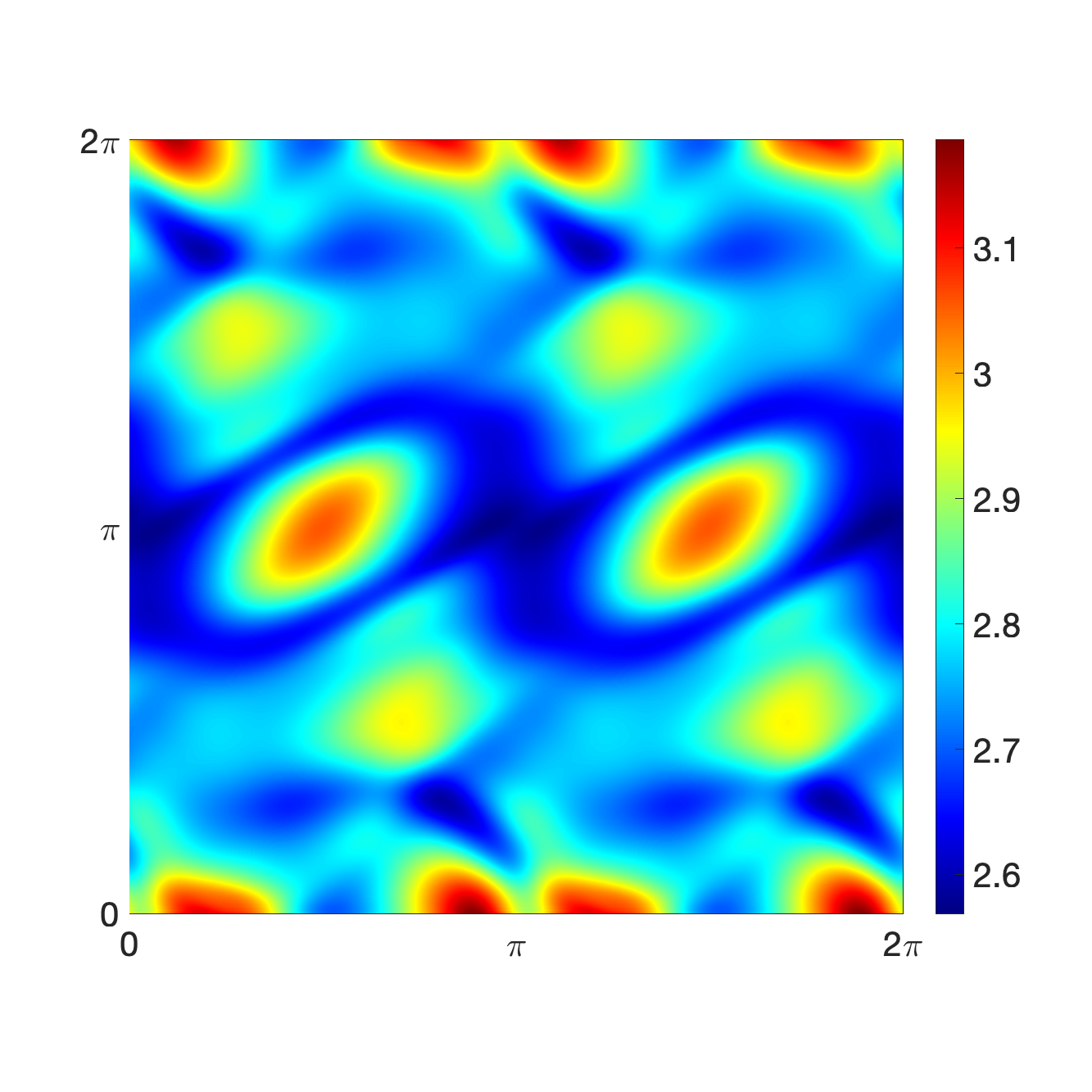}
	\end{subfigure}	
	\begin{subfigure}{0.32\textwidth}
	\includegraphics[width=\textwidth]{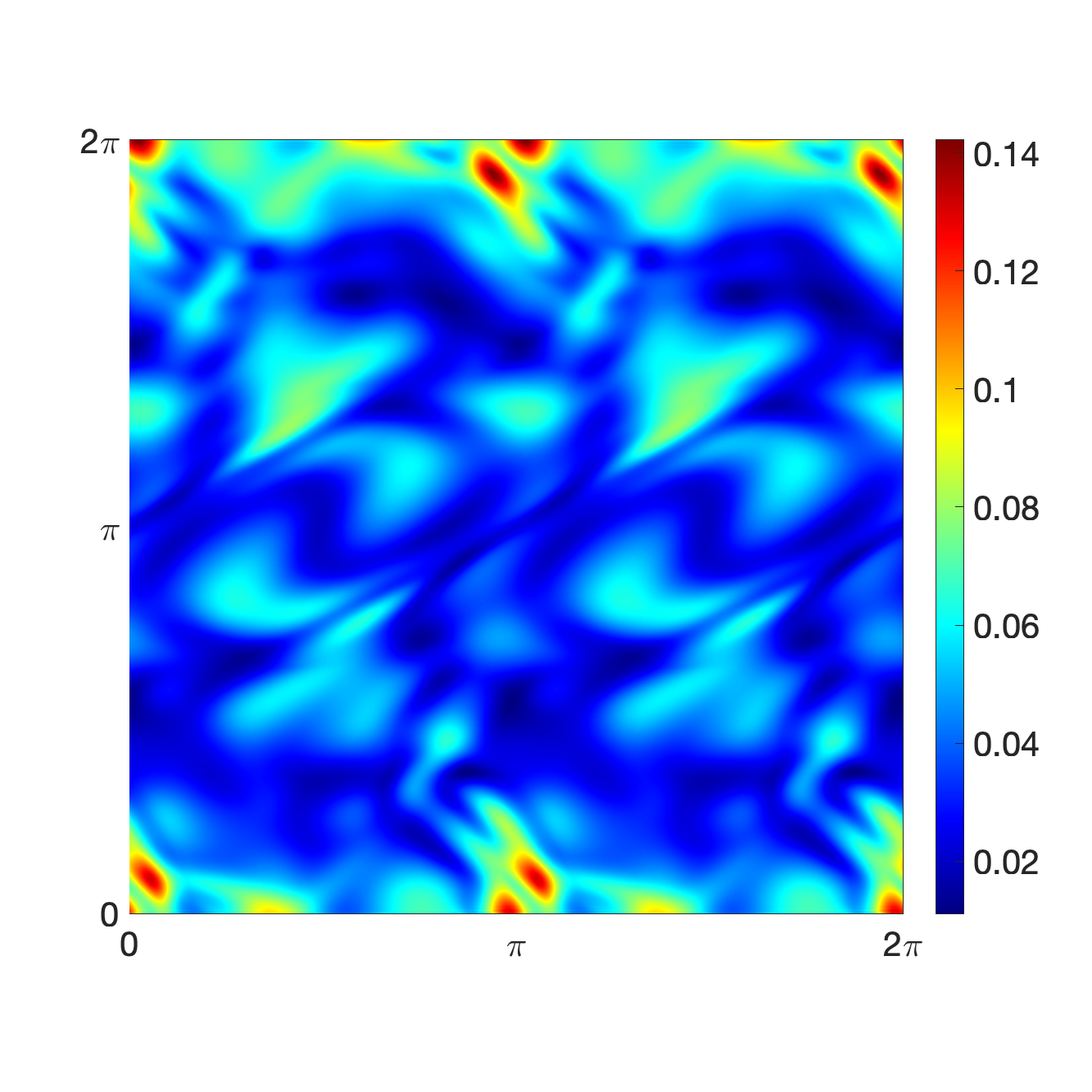}
	\end{subfigure}	
	\begin{subfigure}{0.32\textwidth}
	\includegraphics[width=\textwidth]{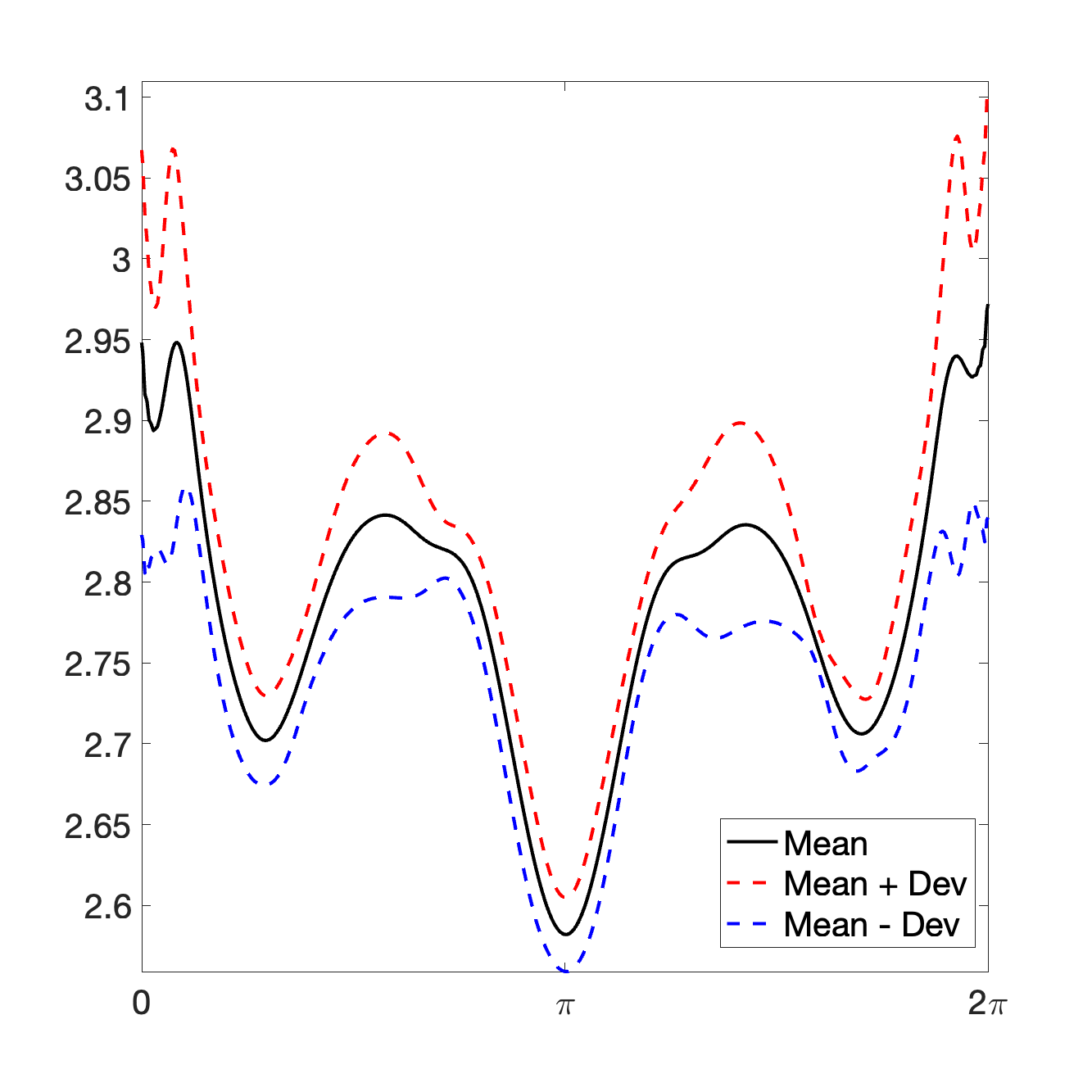}
	\end{subfigure}	\\
	\begin{subfigure}{0.32\textwidth}
	\includegraphics[width=\textwidth]{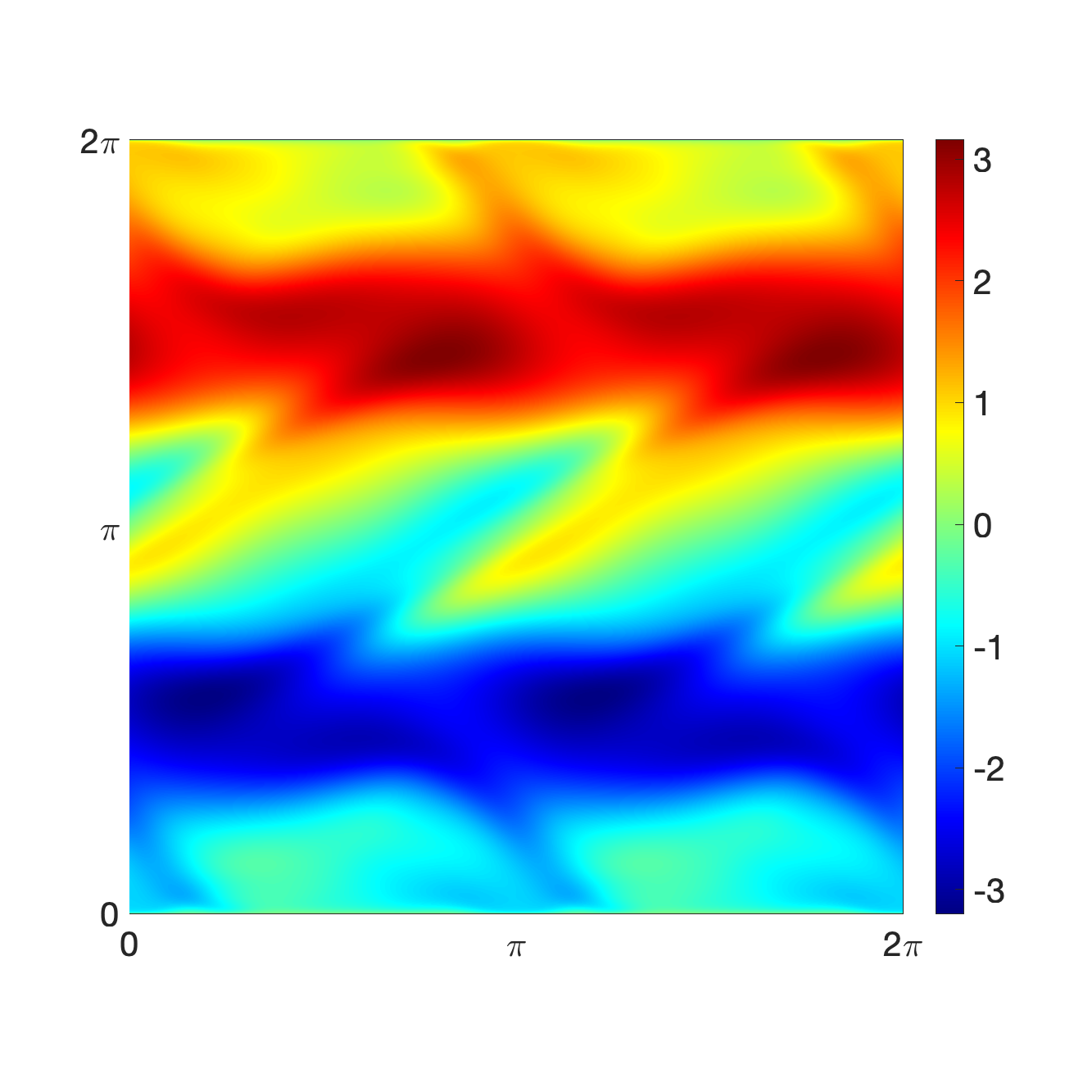}
	\end{subfigure}	
	\begin{subfigure}{0.32\textwidth}
	\includegraphics[width=\textwidth]{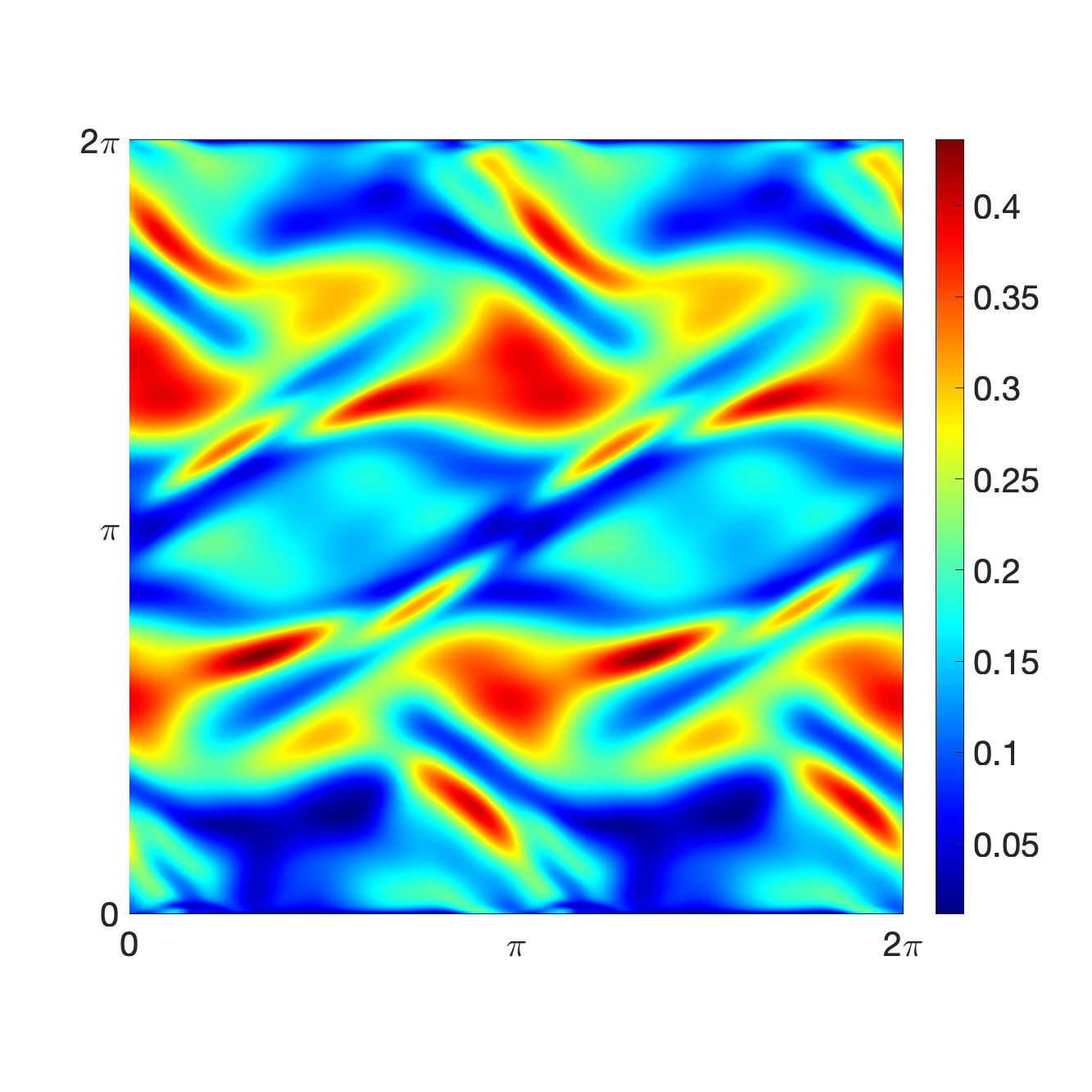}
	\end{subfigure}	
	\begin{subfigure}{0.32\textwidth}
	\includegraphics[width=\textwidth]{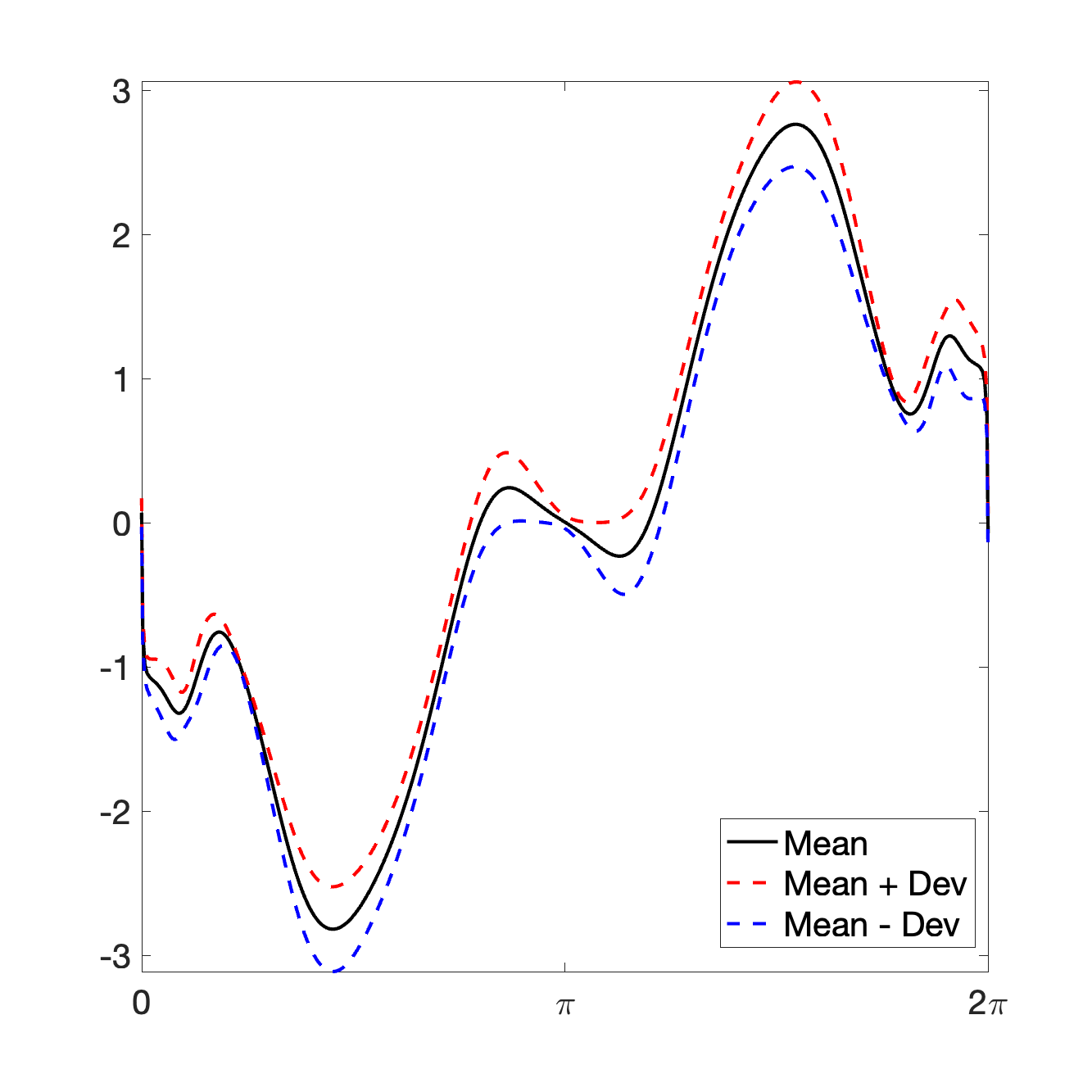}
	\end{subfigure}	\\
	\begin{subfigure}{0.32\textwidth}
	\includegraphics[width=\textwidth]{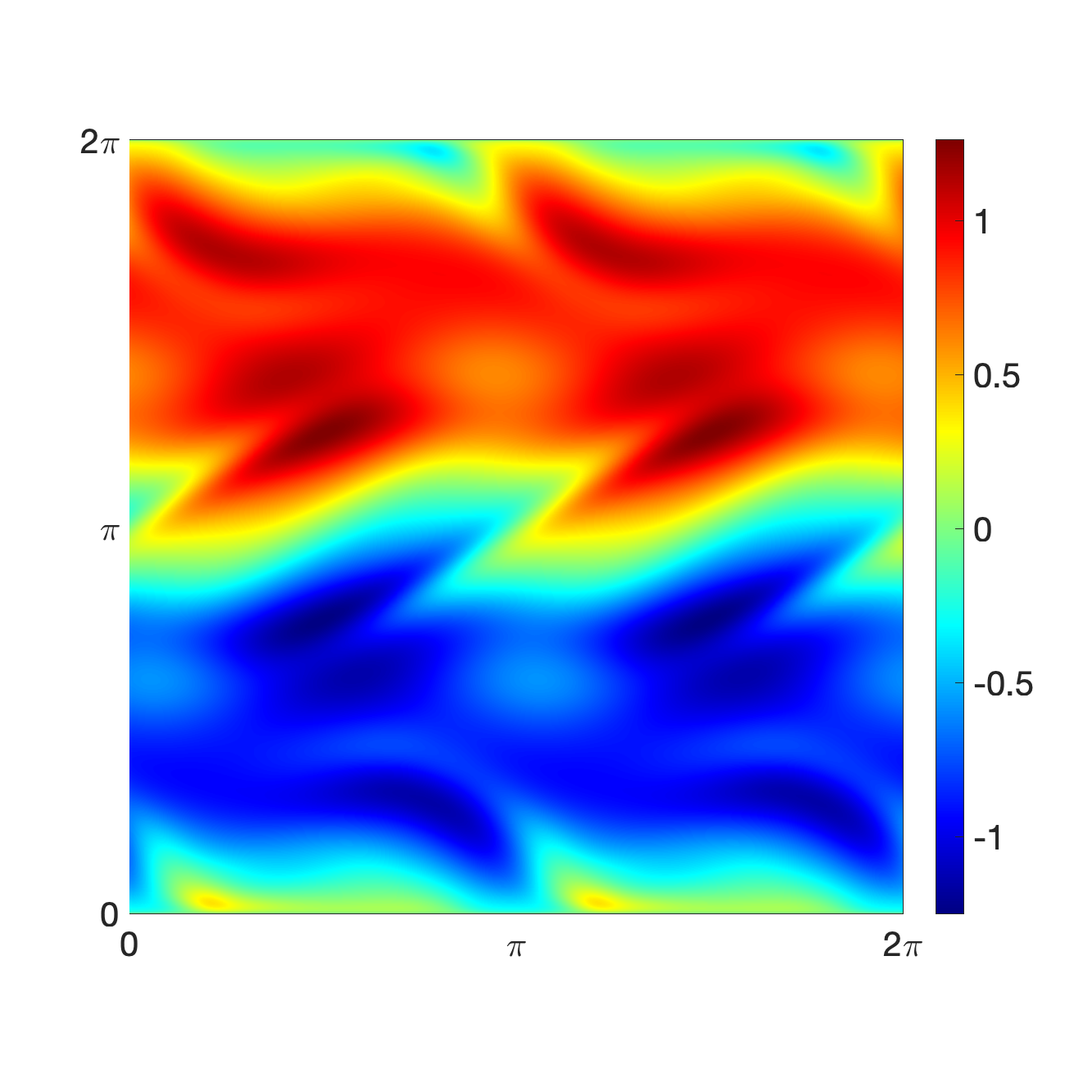}
	\end{subfigure}	
	\begin{subfigure}{0.32\textwidth}
	\includegraphics[width=\textwidth]{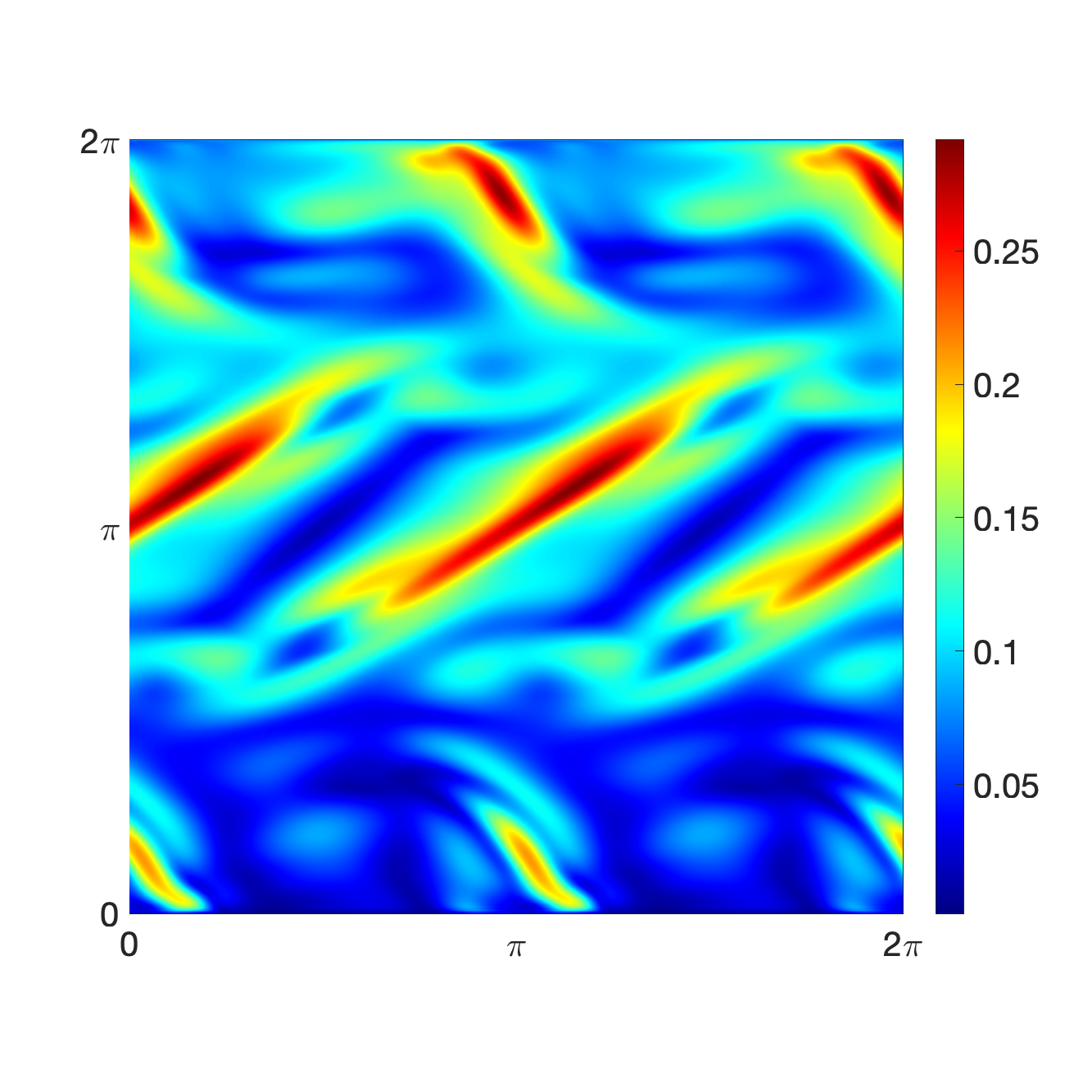}
	\end{subfigure}	
	\begin{subfigure}{0.32\textwidth}
	\includegraphics[width=\textwidth]{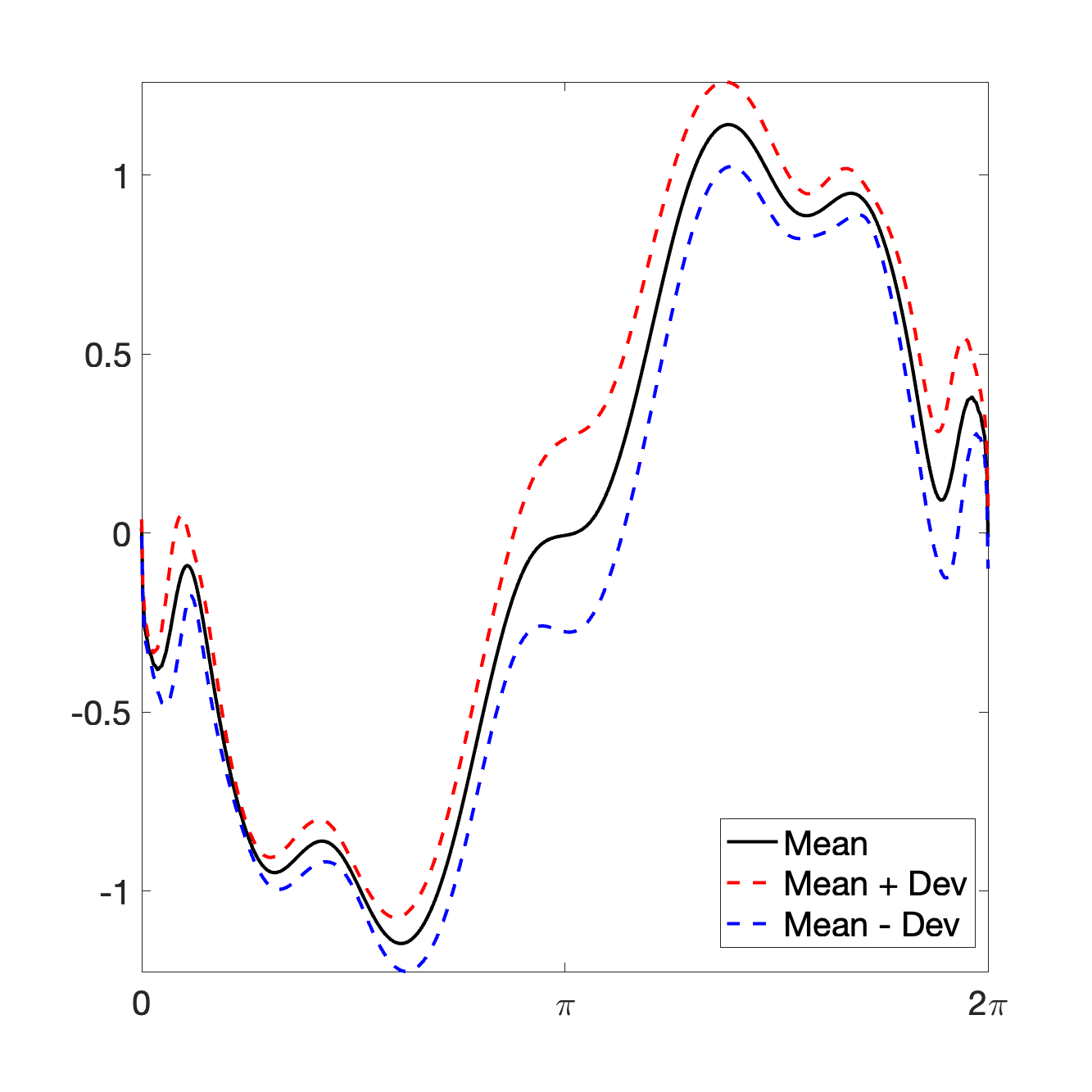}
	\end{subfigure}	
	\caption{ {\bf Random Orszag-Tang-type problem}. FV  solutions $(\vr, m_1, B_1)_{h_{ref}}(\omega,T)$ (from top to bottom) with $h_{ref} = 2\pi/400, \, T=3$ and $M_{ref}=500$ samples. From left to right: mean, deviation, mean and deviations along $x=y$.}\label{OT-MC-1}
\end{figure}

\def\cprime{$'$} \def\ocirc#1{\ifmmode\setbox0=\hbox{$#1$}\dimen0=\ht0
  \advance\dimen0 by1pt\rlap{\hbox to\wd0{\hss\raise\dimen0
  \hbox{\hskip.2em$\scriptscriptstyle\circ$}\hss}}#1\else {\accent"17 #1}\fi}

\appendix


\section{Proof of existence of a numerical solution}\label{exist}

We start by recalling the following version of fixed point theorem, see Gallou\"{e}t, Maltese , Novotn\'y \cite[Theorem A.1]{GMN}.

\begin{Theorem}[{\bf A fixed point theorem}]\label{fixedpoint}
Let $M$ and $N$ be positive integers. Let $  C_1>\ep>0$  and $C_2>0$ be real numbers. Let
\begin{align*}
&V = \{ (r,u) \in  \R ^M \times \R^N, \ r_i >0,\ i=1,\dots, M  \}, \\
&W = \{ (r,u) \in \R ^M \times  \R^N, \, \ep < r_i < C_1, \ i=1, \dots, M   \text{ and }  | u | \le C_2 \}.
\end{align*}
Let  $F$ be a continuous function mapping $V \times [0,1]$ to $\R^M \times \R^N$ and satisfying:
\begin{enumerate}
\item If $ f \in V $  satisfies $ F(f,\xi)=0$ for all $ \xi \in [0,1] $ then $ f \in W $;
\item The equation $F(f, 0)=0$ is a linear system with respect to $f$ and admits a solution in $W$.
\end{enumerate}

Then there exists $ f \in W$ such that $F(f,1) = 0$.
\end{Theorem}

\noindent We are now ready to prove the existence of a solution to the FV scheme \eqref{scheme}.

\begin{proof}[Proof of Lemma \ref{lem-existence}]
Let us denote
\begin{align*}
V &=\left\{ ( \vrh^k, \vuh^k , \vBh^k) \in \Qh \times \bQh\times \bQh;\ \eqref{bc-num} \mbox{ holds
and } \vrh^k>0, \ \Divh \vBh^k=0
\right\},
\\
W &=\Big\{ ( \vrh^k, \vuh^k , \vBh^k  ) \in \Qh \times \bQh \times \bQh;\ \eqref{bc-num} \mbox{ holds and } 0 < \epsilon \leq \vrh^k \leq C_1,
\\
&\hspace{6cm}\norm{\vuh^k} \leq C_2 , \ \norm{\vBh^k} \leq C_3,\  \Divh \vBh^k=0 \Big\},
\end{align*}
where $\norm{\vuh } \equiv \norm{\Gradd \vuh}_{L^2(Q;\R^{d\times d})} + \norm{\vuh }_{L^2(Q;\R^d)}$ and
$\norm{\vBh } \equiv \norm{\Curlh \vBh}_{L^2(Q;\R^d)} + \norm{\vBh }_{L^2(Q;\R^d)}$.
Note that by writing  $\vrh>c$ we mean $\vr_K>c$ for all $K\in \mesh .$

It is obvious that the dimension of the spaces $\Qh$ and $\bQh$ is finite. Indeed the space $\Qh$ can be identified as $\Qh\subset \R^M$, where $M$ is the total number of elements of $\grid$. Analogously, $\bQh \subset \R^{N}$, where $N=d \, M$.

\medskip

{\bf Step 1:} we define the following mapping
  \begin{align*}
   F : \ V  \times [0,1]\rightarrow  \Qh  \times \bQh \times \bQh,   \quad
        (\vrh^k,\vuh^k, \vBh^k, \xi)\longmapsto ( \vr^\star, \vu^\star,\vB^\star) =F(\vrh^k,\vuh^k,\vBh^k , \xi),
  \end{align*}
  where $ \xi \in [0,1]$ and $( \vr^\star, \vu^\star,\vBh^\star) \in  \Qh \times \bQh\times \bQh $ satisfies
\begin{equation} \label{VNS_q1}
\intQ{ \vr^\star \phi_h } = \intQ{D_t \vrh \cdot \phi_h}
- \xi\intfacesint{  \Fup(\vrh  ,\vuh  )\jump{\phi_h}   } ,
\end{equation}
\begin{multline} \label{VNS_q2}
\intQ{  \vu^{\star} \cdot \vh}  = \intQ{D_t (\vrh \vuh) \cdot \vh}
+  \mu \intQ{ \Gradd \vuh^k : \Gradd  \vh }
 - \xi  \intfacesint{  \Fup(\vrh   \vuh,\vuh  ) \cdot\jump{\vh}   }
\\+  \xi \intQ{\left( \nu\, \Divh \vuh^{k} - p(\vrh^{k}) \right) \Divh  \vh}
 - \xi \intQ{ (\Curlh \vBh   \times  \vBh ) \cdot \vh}
 - \xi \intQ{\vrh \vh \cdot \vc{g}}  ,
 \end{multline}
 \begin{equation}\label{VNS_q2B}
\intQ{  \vB^\star  \cdot  \vCh}
=
\intQ{ D_t \vBh  \cdot  \vCh}
+ \zeta \intQ{   \Curlh \vBh \cdot   \Curlh \vCh}
-  \xi \intQ{ \vuh  \times \vBh \cdot   \Curlh \vCh}
\end{equation}
for any $\phi_h \in \Qh$, $\vh\in \bQh$, $\avs{\vh}|_{ \facesext} = 0,$   $\vCh \in \bQh$, $\vn \times \avs{\vCh}|_{\facesext}=0. $

Obviously $F$ is well-defined and continuous since the values of $\vr^\star$, $\vu^\star$ and $\vB^{\star}$ can be determined by setting $\phi_h =  \mathds{1}_{K}$ in \eqref{VNS_q1},  $v_{i,h}= \mathds{1}_K, v_{j,h}=0$ for $j\neq i, \; i,j \in(1,\ldots, d)$ in \eqref{VNS_q2},
 and $C_{i,h}= \mathds{1}_K, C_{j,h}=0$ for $j\neq i, \; i,j \in(1,\ldots, d)$ in  \eqref{VNS_q2B}.

\medskip

{\bf Step 2:} we show  $(\vrh^k, \vuh^k, \vBh^k) \in W$ if
$(\vrh^k, \vuh^k, \vBh^k) \in V$ satisfying $F(\vrh^k, \vuh^k, \vBh^k, \xi)=\bf{0}$ for all $\xi \in [0,1],$ cf.~Hypothesis~1 from Theorem~\ref{fixedpoint}. The system $F(\vrh^k, \vuh^k, \vBh^k, \xi)=\bf{0}$ reads
\begin{subequations}
\begin{equation} \label{VNS_q3}
\intQ{D_t \vrh \, \phi_h}
  - \xi\intfacesint{  \Fup(\vrh  ,\vuh  )\jump{\phi_h}   }   =0,
\end{equation}
\begin{multline} \label{VNS_q4}
\intQ{D_t (\vrh \vuh) \cdot \vh}
+  \mu \intQ{ \Gradd \vuh^k  : \Gradd \vh } - \xi  \intfacesint{  \Fup(\vrh   \vuh,\vuh  ) \cdot\jump{\vh}   }
\\
+ \xi \intQ{\left(\nu \, \Divh \vuh^{k}  - p(\vrh^{k}) \right) \Divh \bfphi_h }  - \xi \intQ{ (\Curlh \vBh   \times  \vBh ) \cdot \vh}
 - \xi \intQ{\vrh \vh \cdot \vc{g}}  =0,
\end{multline}
 \begin{equation}\label{VNS_q4B}
\intQ{ D_t \vBh  \cdot  \vCh}
+ \zeta \intQ{   \Curlh \vBh \cdot   \Curlh \vCh}
-  \xi \intQ{ \vuh  \times \vBh \cdot   \Curlh \vCh} =0.
\end{equation}
\end{subequations}

Taking $\phi_h=1$ as a test function in \eqref{VNS_q3}  we obtain
\begin{equation}\label{VNS_q5}
\norm{\vrh^k}_{L^1(Q)}= \intQ{\vrh^k} =  \intQ{\vrh^{k-1}}  =\intQ{ \vr_0 }>0,
\end{equation}
implying $ \vrh^k \leq \frac{1}{h^d} \intQ{ \vr_0 } \equiv C_1$.

Further, setting $\phi_h =\Hc'(\vrh) -\frac{\vuh^2}{2}$, $ \vh=\vuh^k$  and $\vCh = \vBh - \PiB \vB_B$ as the test functions in \eqref{VNS_q3}, \eqref{VNS_q4}  and \eqref{VNS_q4B}, respectively, it follows from the standard proof of energy estimates, see Lemma \ref{EnBa},  that
 \begin{equation}\label{VNS_q5B}
 \begin{aligned}
&	 D_t \intQ{ \left(\frac{1}{2}  \vrh  |\vuh |^2  + \Hc(\vr) +\frac12 |\vBh - \PiB \vB_B|^2 \right) }
	 +  \mu \norm{\Gradd \vuh }_{L^2(Q;\R^{d\times d})}^2  + \xi \nu \norm{ \Divh \vuh }_{L^2(Q)}^2
	\\&
	+  \zeta \norm{\Curlh \vBh }_{L^2(Q;\R^{d})}^2 +  \widetilde{D_{num}}
=
-\intQ{ ( \xi \vuh  \times \vBh - \zeta \Curlh \vBh) \cdot   \Curlh \PiB\vB_B} + \xi \intQ{\vrh \vuh \cdot \vc{g}},
\end{aligned}
\end{equation}
with
\begin{align*}
\widetilde{D_{num}} = &  \frac{\TS}{2} \intQB{ \Hc''(\xi_1)|D_t \vrh |^2 + \vrh^\triangleleft|D_t \vuh |^2  + |D_t \vBh|^2 }
	+\xi \intfacesint{ \Hc''(\xi_2)  \left(h^\eps  +  \frac12 | \us | \right) \jump{  \vrh  } ^2} \br
	&+ \xi \intfacesint{ \left( \frac12 \vrh^{\rm up} |\us | +h^\eps \avs{ \vrh  } \right)      \abs{\jump{\vuh}}^2   }.
\end{align*}
Repeating the proof of Lemma~\ref{UnBo} we get the the following estimates
\begin{equation}\label{VNS_q6}
\norm{\vuh } \equiv \norm{\Gradd \vuh}_{L^2(Q)} + \norm{\vuh }_{L^2(Q)} \leq C_2,
\end{equation}
and
\begin{equation}\label{VNS_q6B}
\norm{\vBh } \equiv \norm{ \vBh}_{L^2(Q)} + \norm{\Curlh \vuh }_{L^2(Q)} \leq C_3,
\end{equation}
where $C_2$ and $C_3$ depends on the data of the problem.

Further, let $K\in \grid$ be such that $\vr_K^k = \min_{L \in \grid} \vr_L^k$. Then we have $\jump{\vrh^k}_{\sigma \in \facesK} \geq 0 $ and obtain by setting $\phi_h =  \mathds{1}_K$ in \eqref{VNS_q3} that
\begin{align*}
&\frac{ \abs{K} }{\TS} (\vr_K^k  - \vr^{k-1}_{K})  
=  - \xi \sum_{\sigma \in \facesK  \cap \facesint} \abs{\sigma} \vrh^{k,\up} \avg{\vuh^k} \cdot \vn +
\xi \sum_{\sigma \in \facesK  \cap \facesint} \abs{\sigma}  h^{\eps} \jump{\vrh^k}
\\&
\geq  - \xi \sum_{\sigma \in \facesK \cap \facesint} \abs{\sigma}  \vr_K^k \avg{\vuh^k} \cdot \vn
+ \xi \sum_{\sigma \in \facesK \cap \facesint} \abs{\sigma}  (\vr_K^k - \vrh^{k,\up} ) \avg{\vuh^k} \cdot \vn
\\&
= - \xi \abs{K} \vr_K^k  (\Divh \vuh^k)_K
- \xi \sum_{\sigma \in \facesK  \cap \facesint} \abs{\sigma}  \jump{\vrh^k} \left[\avg{\vuh^k} \cdot \vn \right]^-
\\& \geq - \xi \abs{K} \vr_K^k  (\Divh \vuh^k)_K
\geq - \xi \abs{K} \vr_K^k  \abs{\left( \Divh \vuh^k \right)_K}.
\end{align*}
Thus
$
\vrh^k \geq  \vr_K^k \geq  \frac{\vr^{k-1}_{K} }{1 + \TS \xi  \abs{(\Divh \vuh^k)_K} } .
$
Consequently, by virtue of \eqref{VNS_q6}
it holds $\abs{(\Divh \vuh^k)_K} \leq C_2/h^d$ and $ \vrh^k \geq \frac{\vr^{k-1}_{K}  h^d}{1 + \TS \xi  C_2 } \equiv \epsilon  >0 $.
Therefore,  $(\vrh^k, \vuh^k, \vBh^k) \in W.$

\medskip

 {\bf Step 3:} we proceed to show that $F(\vrh^k,\vuh^k, \vBh^k, 0)=\bf{0}$ is a linear system admitting a solution in $W,$ cf.~Hypothesis~2 of Theorem~\ref{fixedpoint}.
Thus, for $ \xi=0$  the system $F(\vrh^k,\vuh^k, \vBh^k, 0)=\bf{0}$ reads
\begin{subequations}\label{VNS_q78}
\begin{equation} \label{VNS_q7}
\vrh^k = \vrh^{k-1},
\end{equation}
\begin{equation}\label{VNS_q8}
\intQ{ \frac{\vrh^k \vuh^{k}  - \vrh^{k-1} \vuh^{k-1}  }{\TS}  \cdot \vh }
+  \mu \intQ{ \Gradd \vuh^k :   \Gradd \vh  }
=0,
\end{equation}
\begin{equation}\label{VNS_q8B}
\intQ{ D_t \vBh^k  \cdot  \vCh}
+ \zeta \intQ{  \Curlh \vBh^k \cdot   \Curlh \vCh} =0.
\end{equation}
\end{subequations}
From \eqref{VNS_q7} it is obvious that $\vrh^k = \vrh^{k-1}>0$.
Substituting \eqref{VNS_q7} into \eqref{VNS_q8} we arrive at a linear system for $\vuh^k$ with a symmetric positive definite matrix. Thus, \eqref{VNS_q8} admits a unique solution.
Further, it is obvious that \eqref{VNS_q8B} is  a linear and coercive equation. The Lax-Milgram theorem  then implies the existence and uniqueness of a solution.
Consequently, we have proved that  $ F(\vrh^k, \vuh^k, \vBh^k, 0)=\bf{0}$ is a linear system admitting a solution in  $W.$

Applying Theorem~\ref{fixedpoint} finishes the proof.
\end{proof}

\section{Proof of compatibility condition}\label{compatibility}

\begin{proof}[Proof of Lemma \ref{Tm3}]
First, thanks to $\vB\cdot \Curl \vC   -  \Curl \vB\cdot  \vC  = \Div(\vC \times \vB)$  and the Stokes theorem we have
\begin{align*}
\intQ{\Big(   \vB\cdot \Curl \vC  -  \Curl \vB\cdot  \vC \Big)}
= \intfacesext{\vn \cdot( \vC \times \vB)} = - \intfacesext{\vn \times \vB_B \cdot \vC}.
\end{align*}
Second, denoting $\vCh = \PiQ \vC$ with the extension $\vn \times \jump{\vCh}|_{\facesext} = 0$ we derive from the integration by part formula \eqref{InByPa-4}  that
\begin{align*}
&\intQB{\Curlh \vBh \cdot  \vC - \vBh \cdot  \Curl \vC  } = \intQB{  \Curlh \vBh \cdot  \vCh - \vBh \cdot  \Curl \vC  } \br
&=\intQB{ \Curlh \vBh \cdot   \vCh - \vBh \cdot  \Curlh  \vCh  } + \intQ{\vBh \cdot  (\Curlh  \vCh -  \Curl \vC  ) }\br
&=\intfacesext{\vn \times \avs{\vBh}\cdot \vCh^{in}} + \intQ{\vBh \cdot  (\Curlh  \vCh -  \Curl \vC   ) }.
\end{align*}
 Summing up above two equalities and thanks to the boundary condition $\vn \times \avs{\vBh}|_{\facesext} = \vn \times \vB_B|_{\facesext}$ we have
\begin{align*}
e_{\Curl \vB} =  \int_0^{\tau} \intfacesext{\vn \times \vB_B \cdot (\vCh^{in} - \vC )} \dt +  \intTauO{\vBh \cdot  ( \Curlh \PiQ \vC - \Curl \vC  ) } := I_1 + I_2,
\end{align*}
and shall control $I_1$ by the interpolation estimates
\begin{align*}
\abs{I_1} \aleq h \norm{\vC}_{L^1(0,T;W^{1,2}(Q;\R^d)}.
\end{align*}

Lastly, we estimate $I_2$ by analyzing $\Curlh \vCh - \Curl \vC $ in two cases -- interior and boundary cells, i.e.
\begin{equation}
\abs{\Curlh \vCh - \Curl \vC } \aleq
\begin{cases}
h, & K \cap \facesext = \emptyset, \\
1,  &K \cap \facesext \neq \emptyset.
\end{cases}
\end{equation}
This can be done by means of the regularity of $\vC$ and the interpolation estimates:
\begin{align*}
& \frac1{|K|}\intK{(\Curlh \vCh - \Curl \vC)}  = \frac{1}{|K|}\sum_{\sigma\in \facesK}|\sigma| \vc{n} \times (\avs{\vCh} - \PiF \vC) \br
&= \frac{1}{h}\left( \sum_{\sigma\in \facesK \cap \facesint}  \vc{n} \times (\avs{\vCh} - \PiF \vC)  + \sum_{\sigma\in \facesK \cap \facesext}  \vc{n} \times (\vCh^{in} - \PiF \vC)  \right)
\end{align*}
and
\[
\abs{\avs{\vCh} - \PiF \vC} \aleq h^2, \ \sigma \in \facesint; \quad
\abs{\vCh^{in} - \PiF \vC} \aleq h, \  \sigma \in \facesext.
\]
Here we have used the extension $\vn \times \jump{\vCh}|_{\facesext} = 0$.
Consequently, denoting by $O$ the set of boundary cells, we have $\abs{O} \aleq h$ and shall control $I_2$ with the  H\"older inequality
\begin{align*}
\abs{I_2} & = \abs{\intTauO{\vBh \cdot  ( \Curlh \PiQ \vC - \Curl \vC  ) } }\aleq h \int_0^{\tau} \int_{Q\setminus O}{\abs{\vBh}} \dxdt+  \int_0^{\tau} \int_{O}{\abs{\vBh}} \dxdt \br
& \aleq h \norm{\vBh}_{L^\infty(0,T;L^2(Q;\R^d))}  + \left( \int_0^{\tau} \int_{O}{\abs{\vBh}^2} \dxdt  \right)^{1/2} \left(  \int_0^{\tau} \int_{O}{ 1 } \dxdt  \right)^{1/2} \br
&\aleq h + h^{1/2} \aleq h^{1/2}.
\end{align*}
Altogether, we finish the proof.
\end{proof}

\section{Calculation of $R_7$}\label{sec-R7}
In this section we reformulate $R_7$ in the form of \eqref{R7} into the form of \eqref{R7-1}.
Noticing that $\vB_B$ is independent of time $t$ and $\vB$ satisfies the MHD problem \eqref{pde}, i.e.
\[
 \pd_t \vB=\Curl (\vu \times \vB -  \zeta \Curl \vB),
\]
we shall simplify \eqref{R7} as follows
\begin{align*}
R_7 =&  -\intTauO{ (\vuh  \times \vBh - \zeta \Curlh \vBh) \cdot   \Curlh \PiB\vB_B}  - \intTauO{\Big(  \Curlh \vBh \times \vBh \cdot \vu \Big)}
\br
&+ \intTauO{ (\vB-\vBh) \cdot \pdt  \vB}  + \intTauO{\Big(  \vuh \times \vBh -  \zeta  \Curlh \vBh \Big) \cdot \Curl \vB_B }
 \br
&  - \intTauO{ \vuh \times \vBh \cdot \Curl  \vB }  + \intTauO{(\vu - \vuh) \cdot \Curl \vB \times \vB}
    \br
&+\intTauO{\Big( \zeta \abs{ \Curl \vB}^2 -  \zeta \Curlh \vBh \cdot \Curl \vB \Big)}
\br
=&-\intTauO{ (\vuh  \times \vBh - \zeta \Curlh \vBh) \cdot   \Big( \Curlh \PiB\vB_B - \Curl \vB_B\Big)} \br
&+ \intTauO{\Big( \vu  \times \vBh \cdot  \Curlh \vBh -  \vuh \times \vBh \cdot \Curl \vB - (\vu - \vuh) \times \vB\cdot \Curl \vB \Big)}
    \br
&+\intTauO{ \zeta \Curl \vB \cdot  \Big( \Curl \vB  -\Curlh \vBh \Big)}+ \intTauO{(\vB-\vBh) \cdot  \Curl (\vu \times \vB -  \zeta \Curl \vB)}
\br
=&-\intTauO{ (\vuh  \times \vBh - \zeta \Curlh \vBh) \cdot   \Big( \Curlh \PiB\vB_B - \Curl \vB_B\Big)} \br
&+ \intTauO{\Big( \vu  \times \vBh \cdot  \Curlh \vBh -  \vuh \times \vBh \cdot \Curl \vB - (\vu - \vuh)  \times \vB \cdot \Curl \vB + \vu \times \vB \cdot (\Curl \vB-\Curlh \vBh)  \Big)}
    \br
&+\intTauO{\zeta \Curl \vB \cdot  \Big( \Curl \vB  -\Curlh \vBh \Big) -  \zeta   (\Curl \vB-\Curlh \vBh) \cdot   \Curl  \vB  \Big)} \br
&+ \intTauO{\Big((\vB-\vBh) \cdot  \Curl (\vu \times \vB-  \zeta \Curl  \vB) -  (\Curl \vB-\Curlh \vBh) \cdot   (\vu \times \vB-  \zeta \Curl  \vB)  \Big)}
\br
=&-\intTauO{ (\vuh  \times \vBh - \zeta \Curlh \vBh) \cdot   \Big( \Curlh \PiB\vB_B - \Curl \vB_B\Big)} \br
&+ \intTauO{\Big( \vu  \times (\vBh-\vB) \cdot  (\Curlh \vBh-\Curl \vB) + (\vu  -  \vuh) \times (\vBh-\vB) \cdot \Curl \vB  \Big)}
    \br
&+ \intTauO{\Big((\vB-\vBh) \cdot  \Curl (\vu \times \vB-  \zeta \Curl  \vB) -  (\Curl \vB-\Curlh \vBh) \cdot   (\vu \times \vB-  \zeta \Curl  \vB)  \Big)}.
\end{align*}

\end{document}